\documentclass[11pt,reqno,letterpaper]{amsart}
\usepackage{color}
\usepackage[colorlinks=true, allcolors=blue,backref=page]{hyperref}
\usepackage{amsmath, amssymb, amsthm}
\usepackage{mathrsfs}
\usepackage{mathtools}
\usepackage[noabbrev,capitalize,nameinlink]{cleveref}
\crefname{equation}{}{}
\usepackage{fullpage}
\usepackage[noadjust]{cite}
\usepackage{graphics}
\usepackage{pifont}
\usepackage{tikz}
\usepackage{bbm}
\usepackage[T1]{fontenc}
\DeclareMathOperator*{\argmax}{arg\,max}

\usetikzlibrary{arrows.meta}

\usepackage{environ}
\usepackage{framed}
\usepackage{url}
\usepackage[linesnumbered,ruled,vlined]{algorithm2e}
\usepackage[noend]{algpseudocode}
\usepackage[labelfont=bf]{caption}
\usepackage{cite}
\usepackage{framed}
\usepackage[framemethod=tikz]{mdframed}
\usepackage{appendix}
\usepackage{graphicx}
\usepackage[textsize=tiny]{todonotes}
\usepackage{tcolorbox}
\usepackage{enumerate}
\allowdisplaybreaks[1]
\usepackage{enumerate}
\usepackage{stmaryrd}
\usepackage[margin=1in]{geometry}

\usepackage[shortlabels]{enumitem}
\crefformat{enumi}{#2#1#3}
\crefrangeformat{enumi}{#3#1#4 to~#5#2#6}
\crefmultiformat{enumi}{#2#1#3}
{ and~#2#1#3}{, #2#1#3}{ and~#2#1#3}

\DeclareSymbolFont{symbolsC}{U}{pxsyc}{m}{n}
\SetSymbolFont{symbolsC}{bold}{U}{pxsyc}{bx}{n}
\DeclareFontSubstitution{U}{pxsyc}{m}{n}
\DeclareMathSymbol{\medcircle}{\mathbin}{symbolsC}{7}

\crefname{algocf}{Algorithm}{Algorithms}

\crefname{equation}{}{} 
\AtBeginEnvironment{appendices}{\crefalias{section}{appendix}} 

\usepackage[color,final]{showkeys} 

\colorlet{refkey}{orange!20}
\colorlet{labelkey}{blue!30}

\crefname{algocf}{Algorithm}{Algorithms}

\numberwithin{equation}{section}
\newtheorem{theorem}{Theorem}[section]
\newtheorem{proposition}[theorem]{Proposition}
\newtheorem{lemma}[theorem]{Lemma}
\newtheorem{claim}[theorem]{Claim}

\crefname{claim}{Claim}{Claims}

\newtheorem*{question*}{Question}
\newtheorem*{assumption*}{Assumption}

\theoremstyle{definition}
\newtheorem{definition}[theorem]{Definition}

\newtheorem{assumption}[theorem]{Assumption}
\newtheorem*{definition*}{Definition}

\theoremstyle{remark}
\newtheorem*{remark}{Remark}

\newcommand{\snorm}[1]{\lVert#1\rVert}
\newcommand{\sang}[1]{\langle #1 \rangle}

\newcommand{\imod}[1]{\mathrm{(mod}~#1\mathrm{)}}

\newcommand{\mb}{\mathbb}
\newcommand{\mbf}{\mathbf}
\newcommand{\mbm}{\mathbbm}
\newcommand{\mc}{\mathcal}
\newcommand{\mf}{\mathfrak}
\newcommand{\mr}{\mathrm}

\newcommand{\ol}{\overline}
\newcommand{\on}{\operatorname}

\newcommand{\wt}{\widetilde}

\newcommand{\eps}{\varepsilon}

\let\originalleft\left
\let\originalright\right
\renewcommand{\left}{\mathopen{}\mathclose\bgroup\originalleft}
\renewcommand{\right}{\aftergroup\egroup\originalright}

\allowdisplaybreaks

\title{On Perfectly Friendly Bisections of Random Graphs}

\author[Minzer]{Dor Minzer}
\author[Sah]{Ashwin Sah}
\author[Sawhney]{Mehtaab Sawhney}
\address{Department of Mathematics, Massachusetts Institute of Technology, Cambridge, MA 02139, USA}
\email{minzer.dor@gmail.com, \{asah,msawhney\}@mit.edu}
\begin{document}

\thanks{Minzer was supported by a Sloan Research Fellowship and NSF CCF award 2227876. Sah was supported by the PD Soros Fellowship. Sawhney was supported by the Churchill Foundation. Sah and Sawhney were supported by NSF Graduate Research Fellowship Program DGE-2141064.}

\begin{abstract}
We prove that there exists a constant $\gamma_{\mr{crit}}\approx .17566$ such that if $G\sim \mb{G}(n,1/2)$ then for any $\eps > 0$ with high probability $G$ has a equipartition such that each vertex has $(\gamma_{\mr{crit}}-\eps)\sqrt{n}$ more neighbors in its own part than in the other part and with high probability no such partition exists for a separation of $(\gamma_{\mr{crit}}+\eps)\sqrt{n}$. The proof involves a number of tools ranging from isoperimetric results on vertex-transitive sets of graphs coming from Boolean functions, switchings, degree enumeration formulas, and the second moment method. Our results substantially strengthen recent work of Ferber, Kwan, Narayanan, and the last two authors on a conjecture of F\"uredi from 1988 and in particular prove the existence of fully-friendly bisections in $\mb{G}(n,1/2)$.
\end{abstract}

\maketitle

\section{Introduction}\label{sec:introduction}

In this paper we consider the circumstances under which the random graph $\mb{G}(n,1/2)$ can be partitioned into two roughly equal size sets $A\cup B$ such that every vertex in $A$ has more substantially more neighbors in $A$ and analogously for $B$. We formalize this with the notion of an $H$-friendly equipartition.

\begin{definition}\label{def:friend-part}
Given a graph $G$ with $n$ vertices an $H$-friendly equipartition is a partition $A_1\cup A_2 = V(G)$ such that $||A_1|-|A_2||\le 1$ and $\deg(v,A_i)\ge H + \deg(v,A_{i+1})$ for all choices of $i\in\{1,2\}$ and $v\in A_i$. (Here we take indices $\imod{2}$, which we do without further notice.)
\end{definition}

We now define a constant $\gamma_{\mr{crit}}$ which will be used throughout the paper. 
\begin{definition}\label{def:const}
Define $\gamma_{\mr{crit}}>0$ to be the constant such that 
\[\log 2+\sup_{\alpha\in\mb{R}}(-\alpha^2 + \log(\mb{P}_{Z\sim{\mc{N}(0,1)}}[Z\ge(\gamma + \alpha)\sqrt{2}])) = 0.\]
\end{definition}
\begin{remark}
Noting that $\gamma\mapsto\log 2+\sup_{\alpha\in\mb{R}}(-\alpha^2 + \log(\mb{P}_{Z\sim{\mc{N}(0,1)}}[Z\ge(\gamma + \alpha)\sqrt{2}]))$ is a strictly decreasing function therefore $\gamma_{\mr{crit}}$ is unique. Furthermore one can prove via numerical computation that $.2484195\le\gamma_{\mr{crit}}\le .2484196$ (\cref{clm:numeric-bounds}).
\end{remark}

Our main theorem proves, conditional on a numerical hypothesis \cref{asm:comp}, that an $H$-friendly partition in $\mb{G}(n,1/2)$ exists with $H$ being order $\Omega(\sqrt{n})$ and pins down the precise leading order behavior for the best-possible $H$.

\begin{theorem}\label{thm:main}
Fix $\eps > 0$ and let $G\sim\mb{G}(n,1/2)$. Given \cref{asm:comp}, with high probability, $G$ has a $(\gamma_{\mr{crit}}/\sqrt{2}-\eps)\sqrt{n}$-friendly equipartition. Furthermore with high probability $G$ does not have a $(\gamma_{\mr{crit}}/\sqrt{2}+\eps)\sqrt{n}$-friendly equipartition.
\end{theorem}

We state \cref{asm:comp} precisely in \cref{sub:computer} and discuss a rather careful floating-point verification of \cref{asm:comp} which is carried out in \cref{app:computer-assist} which gives a procedure to reduce the verification of \cref{asm:comp} to a series of non-exact integration computations. We then verify the most delicate of these numerical computations using interval arithmetic; the necessary numerical procedure and implementation details are discussed in \cref{sub:computer} and in \cref{app:computer-assist}. 

\subsection{Background on friendly partitions}
Questions about friendly bisections and the analogous concept of friendly partitions (where one drops the condition that the partition is equitable) have been studied under a host of names as recognized in recent work of Behrens, Arpino, Kivva, and Zdeborov\'{a} \cite{BAKZ22}; these include satisfactory graph partition \cite{GK00, BTV10}, generalized matching cuts \cite{GS21}, local minimum cut \cite{AMS21}, and assortative partitions \cite{BAKZ22}. By complementing the graph, one is also naturally led to considering the partitioning problem where one aims to have more neighbors across the cut instead of fewer; this variant has similarly gone under a host of different names.

Our primary motivation comes from seeking to understand the behavior dictating an old conjecture due to F\"uredi~\cite{F88} from 1988 which was popularized by its inclusion in Green's list of 100 open problems~\cite[Problem~91]{GreOp}. F\"uredi conjectured the existence of an bisection where only $n-o(n)$ vertices on both sides are required to be $0$-friendly; note in particular that this conjecture is weaker even than the existence of a $0$-friendly bisection. This conjecture of F\"uredi was recently resolved in work of Ferber, Kwan, Narayanan, and the last two authors \cite{FKNSS22}. However as mentioned in \cite{FKNSS22}, based on computer simulations, and even to F\"uredi\footnote{We thank Benny Sudakov for this remark.}, the stronger conjecture that there exists a partition where all nodes are friendly appeared plausible.

This conjecture is also closely related to the existence of gapped states in the Sherrington--Kirkpatrick (SK) model \cite{SK75} predicted in work of Treves and Amit \cite{TA88}. To state the model precisely, considered the a random symmetric matrix $J$ where entries are equally likely to be $\pm 1$. Treves and Amit consider the gap of a signing $\vec{x}\in \{\pm 1\}^{n}$
\[\frac{1}{\sqrt{n}}\min_{i\in [n]}\sum_{j=1}^{n}J_{ij}x_ix_j\]
and conjectured that such a signing $\vec{x}$ exists for gap at least $2\gamma_{\mr{crit}}$. A straightforward translation via assigning edges of $\mb{G}(n,1/2)$ to the state $+1$ and assigning the partition according to the vector $\vec{x}$ in the obvious manner, \cref{thm:main} immediately gives a rigorous proof of the existence of such states (in fact with $|\sum_i x_i|\le 1$). 

Finally, recent work of Behrens, Arpino, Kivva, and Zdeborov\'{a} \cite{BAKZ22} conjectures, based on the replica method in statistical physics, that the analogue of \cref{thm:main} holds when the underlying graph $G$ is a sparse random regular graph. When translating the predictions of these authors into the dense regime corresponding to $\mb{G}(n,1/2)$ (see \cite[Section~3,~Gapped states in spin glasses are computationally hard to find]{BAKZ22}), the results of Behrens, Arpino, Kivva, and Zdeborov\'{a} \cite{BAKZ22} suggest that gapped states such that 
\[\frac{1}{\sqrt{n}}\min_{i\in [n]}\sum_{j=1}^{n}J_{ij}x_ix_j\ge \eps \sqrt{n}\]
are likely computationally difficult to find for any $\eps > 0$. Furthermore the proof of \cref{thm:main} implicitly proves that states which are $(\gamma_{\mr{crit}}/\sqrt{2}-\eps)\sqrt{n}$-friendly exhibit the overlap gap property (OGP) whp; in particular there exists $[\beta_1,\beta_2] \subseteq (0,1/2)$ such that whp any two $H$-friendly partitions do not differ on a fraction of nodes between $\beta_1$ and $\beta_2$. The overlap gap property (and variants) have been used to provide evidence for the computation hardness of numerous constraint satisfaction problems (see e.g.~\cite{GS17, RV17, GS17b, CGPR19, BH22, BS22}); we refer the reader to \cite{Gam21} for a recent survey on the key role of the overlap gap property in optimization over random structures. In particular this evidence in part explains why we divert from the constructive approach adopted in \cite{FKNSS22} (which gives an $O_{\eps}(n^2)$-time algorithm to compute a partition where $(1-\eps)n$ nodes on both sides are friendly) and instead opt for a nonconstructive approach based on the second moment method. 

We note here briefly that in fact \cite{BAKZ22} suggest a frozen 1-RSB structure for the space of $H$-friendly partitions and providing further mathematical proof of these predictions remains of interest. In particular we find the question of whether their exists a polynomial time algorithm to compute a $\gamma\sqrt{n}$-friendly partition in $\mb{G}(n,1/2)$ for $\gamma>0$ a fascinating open question; in particular the work of \cite{BAKZ22} (based on evidence in the sparse graph analog) suggests that the above is likely computationally difficult and giving evidence based on the overlap gap property or more recent variants would be enticing.

\subsection{Discussion of techniques}
The proof of \cref{thm:main} broadly proceeds via the second moment method on the number of $H$-friendly partitions. The first and second moment of the number of $H$-friendly partitions is computed via specially adapting machinery developed in the context of enumeration of graphs with a given degree sequence. However it appears likely that the underlying moments do not match up to a $(1+o(1))$ factor but only a constant factor.

Thus, we can only show that the number of $H$-friendly partitions will be nonzero with some positive constant probability directly, and therefore we require separate tools to boost the constant probability result to high probability. For this we rely on a specially tailored isoperimetric result which is ultimately derived from recent work on Talagrand's conjecture in Boolean function analysis. As such the proof of \cref{thm:main} breaks into a series of essentially separate steps which we now discuss in detail.

\subsubsection{Vertex-adapted isoperimetry in graphs}\label{sub:vertex-iso}
As mentioned, we prove that the first and second moment of the number of bisections which are $(\gamma_{\mr{crit}}/\sqrt{2}-\eps)\sqrt{n}$-friendly match up to a constant factor. Therefore an application of Paley--Zygmund inequality implies the existence of a partition in \cref{thm:main} with at least constant probability. We wish to boost this to $1-o(1)$ probability. Before delving into precise statement proved, let $\mc{F}$ denote the family of graphs on $\{0,1\}^{\binom{[n]}{2}}$ with a $(\gamma_{\mr{crit}}/\sqrt{2}-\eps)\sqrt{n}$-friendly partition and let us consider what naive considerations provide for. Simply viewing a graph as an element in $\{0,1\}^{\binom{[n]}{2}}$, an application of the edge isoperimetric inequality on the hypercube shows that at least $1-\eps$ fraction of graphs are within an edit distance of $O_\eps(n)$ edges from a graph in $\mc{F}$. This immediately gives a partition with at most $O_\eps(\sqrt{n})$ many $(\gamma_{\mr{crit}}/\sqrt{2}-\eps^{1/2})\sqrt{n}$-unfriendly vertices. While this is a promising first step, this proof ignores the underlying $\mc{S}_n$-invariant nature of $\mc{F}$ (since the property of having a $(\gamma_{\mr{crit}}/\sqrt{2}-\eps)\sqrt{n}$-friendly partition is a graph property).

We are thus led to the idea that the optimal vertex isoperimetric statement might be that given a family of graphs on $\{0,1\}^{\binom{[n]}{2}}$ of constant density, if we are allowed to modify at most $O_\eps(1)$ edges at each vertex we may reach nearly all graphs. However, one can consider $\Delta$ so that the probability $\mb{G}(n,1/2)$ has maximum degree at most $\Delta$ is $\eps$. By a result of Bollob\'as \cite{Bol80}, one can see that we must be allowed to modify roughly order $\log(1/\eps)\sqrt{n/\log n}$ edges at some vertex to have a chance at reaching a $1-\eps$ fraction of all graphs.

Our main isoperimetric result (\cref{thm:boolean-function}) proves that the above example is the worst possible up to an absolute constant. This immediately allows one to derive the necessary sharp threshold result, since modifying $o(\sqrt{n})$ edges at every vertex leaves a graph essentially unchanged with respect to having an $(\gamma_{\mr{crit}}/\sqrt{2}-\eps)\sqrt{n}$-friendly partition. Furthermore we note that while this is certainly not the first use of Boolean functions to establish a sharp threshold theorem (see e.g.~\cite{FK96, Fri99, BKS99}), the underlying property considered here is not monotone on the Boolean cube and statement is also novel in proving isoperimetry with respect to a metric (maximum degree distance) that is tailored to the problem at hand.

The proof of this isoperimetric result proceeds via an iterative application of Talagrand's inequality \cite{EG22,EKLM22}. By using (a variant of) Talagrand's inequality applied to symmetric functions one finds that given a family a graphs with density in $[\eps,1-\eps]$, it must have many points of reasonably large (positive) sensitivity, i.e., graphs such that modification of many edges will yield many new graphs. Specifically, we obtain a trichotomy of possible good properties (\cref{lem:trichotomy}). Using this one can consider a sequence of ``expansion moves'' of the underlying family $\mc{F}$ of graphs depending on which case we are in at each step. We can use this to show that it is possible to expand the family to double its density without traversing very far in the relevant metric (maximum degree distance), which is the key \cref{prop:boolean-dyadic}. Proving closeness in the metric requires one to reverse the expansion process and ``trace backwards'' how a graph may have been added to the family under consideration. This isoperimetry argument is carried out in \cref{sec:boolean-functions} and is independent from the remainder of the paper. 

Finally, we note that the main isoperimetry result can also be used to re-derive the threshold of the symmetric binary perceptron (originally proven by Abbe, Li, and Sly \cite{ALS21} and independently Perkins and Xu \cite{PX21}); we comment on this relation further in \cref{sub:boolean-application}.

\subsubsection{Second-moment computation}
For the remainder of the paper the focus is on proving that if $X_\gamma$ is is the number of $\gamma\sqrt{n}$-friendly partitions then $\mb{E}X_\gamma^2\lesssim(\mb{E}X_\gamma)^2$ for $\gamma < \gamma_{\mr{crit}}$ (see \cref{lem:first-moment,lem:second-moment}). We will actually only prove such a statement in the neighborhood of the critical threshold; this is an artifact of our verification of \cref{asm:comp} and extending the range of $\gamma$ in this computation naturally extends the range of validity of \cref{lem:first-moment,lem:second-moment}. Before proceeding further, it is worth noting that the matching of the first and second moment of the number of solutions all the way to the critical threshold is by no means \emph{a priori} obvious; in particular this relies on the the frozen 1-RSB nature of the model which suggest that near criticality the associated solution space consists of subexponential size clusters. 

To get a sense for $\mb{E}X_\gamma$, note that it suffices via linearity of expectation to compute the probability that a given partition is $\gamma\sqrt{n}$-friendly. At a heuristic level, one can replace each edge with a Gaussian random variable of mean $1/2$ and variance $1/4$; such a model has matching mean and variance to $\mb{G}(n,1/2)$. Note however in this Gaussian model that the degree sequence of the vertices in $\mb{G}(n,1/2)$ form a multidimensional Gaussian vector and a straightforward computation implies that the degrees can be expressed in terms of the sum of this vector as well as a collection of independent random variables, one for each vertex. This heuristic, which explicitly appears in work of Riordan and Selby \cite[Section~3]{RS00}, was used to compute the log-probability that a given graph in $\mb{G}(n,1/2)$ has maximum degree bounded by $n/2$. As such this approach is not directly useful since the $\exp(o(n))$-order multiplicative error term is far to large to be directly useful (we remark that however if one only asks for a friendly partition with up to $o(n)$ exceptions such a moment based computation is sufficient due to the edge isoperimetry argument sketched in the \cref{sub:vertex-iso}). However, the work of Riordan and Selby was substantially refined in work of McKay, Wanless, and Wormald \cite{MWW02} to give the asymptotic \emph{probability} that a graph in $\mb{G}(n,1/2)$ has maximum degree bounded by $n/2$. These results in turn relied on powerful tools of McKay and Wormald \cite{MW90} which enumerate dense graphs with sufficiently regular degree sequences.

The key difficulty given the results of \cite{MW90} within the work of McKay, Wanless, and Wormald \cite{MWW02} is summing over all degree sequence with a given maximum degree; in the work of McKay, Wanless, and Wormald \cite{MWW02} this is accomplished by modeling the degree sequence with a certain tilted binomial distribution and reducing the question to certain local central limit theorem considerations for these distributions. Such an approach, while in theory possible for our problem, appears rather unsuitable since one would require (for the second moment) tracking $10$ different tilting patterns and the computations quickly appear to become completely infeasible. We instead adopt a framework given in work of McKay and Wormald \cite{MW97} (extended to bipartite graphs by McKay and Skerman \cite{MS16}) which models the degree sequence of $\mb{G}(n,1/2)$ by first sampling some $p$ near $1/2$ (which controls the total edge density) and then treats each degree as independent binomial distribution conditional on the total sum being even. One may note that this result in fact immediately makes precise the heuristic degree distribution given in the previous paragraph. This is used in work of McKay and Wormald \cite{MW97}, to (say) compute the asymptotic distribution of the median degree in $\mb{G}(n,1/2)$. However, directly as stated the work of McKay and Wormald \cite{MW97} is not suitable for exponentially rare events. Instead, a sufficiently close examination reveals that a modification of their methods allows one to, losing constant factors in the probability, handle such exponentially rare events. This is modification is handled in \cref{sec:degree-models}.

Using the results in \cref{sec:degree-models}, a direct but lengthy computation reduces computing the first and second moments to certain explicit $3$ and $10$ variable Gaussian integrals. While in theory one could apply a grid search to compute the maximum of these integrals, this appears computationally intractable. Instead we use an observation of Gamarnik and Li \cite{GL18} which considered friendly partitions in regular constant degree random graphs, and noted that a certain Gaussian optimization problem could be simplified via using the log-concavity of certain tail probabilities of Gaussian random variables. Via a delicate application of these techniques the first and second moment (involving careful uses of symmetry present in combination with the log-concavity) can be reduced to the problem to checking the optima of a certain univariate and $3$-dimensional functions (which appear in \cref{asm:comp}) match in an appropriate manner. 

Finally, we remark that implicit in the above discussion we have assumed that the degree enumeration results of McKay and Wormald \cite{MW90} (and the extension to bipartite graphs by McKay and Skerman \cite{MS16}) apply uniformly to all degree sequences; unfortunately they only naturally apply to degree sequences where the maximum and minimum degrees are within $O(n^{1/2+\delta})$ of the median degree. As such we use switchings to \emph{a priori} prove that the contribution of atypical degree sequences, even conditional on the exponentially small probability of a given pair of overlapping partitions being $\gamma\sqrt{n}$-friendly, is small. A similar strategy was used in the work of McKay, Wanless, and Wormald \cite{MWW02}, but that work relies on the fact that the event of being below a given maximum degree is monotone in all the edges and this is used crucially in their argument. When considering a pair of overlapping partitions one encounters more complex constraints and thus a more involved argument is necessary.

The argument instead proceeds by using the subset of edges which are naturally decreasing or increasing with respect to the constraints in order to prove that, with super-polynomially high probability, at most a $O(n^{-1/2+\eta})$ fraction of vertices will exactly meet the threshold constraints. Then under the event that there are very few such critical vertices, a relatively straightforward switching argument can then be used prove the necessary maximum and minimum degree events; the fact that there are few critical vertices is used here in order to prove that one has essentially the expected number of switches.

\subsubsection{Verification of \cref{asm:comp}}\label{sub:computer}
We now finally state the crucial numerical computation. We first formally define a series of special functions which will appear throughout the paper. 

\begin{definition}\label{def:special-func}
Fix $\gamma\in\mb{R}$. Then 
\begin{align*}
F_1(\alpha) &:=\log 2-\alpha^2 + \log(\mb{P}_{Z\sim{\mc{N}(0,1)}}[Z\ge(\gamma + \alpha)\sqrt{2}]),\\
f(\beta,\alpha) &:= \mb{P}\bigg[\sqrt{\frac{\beta}{2}}Z_1+\sqrt{\frac{1-\beta}{2}}Z_2\ge\gamma+\alpha\wedge\sqrt{\frac{\beta}{2}}Z_1-\sqrt{\frac{1-\beta}{2}}Z_2\ge\gamma+\alpha\bigg],\\
F_2(\beta,\alpha_1,\alpha_2) &:= 2\log 2-2\beta\log\beta-2(1-\beta)\log(1-\beta) -2 \alpha_1^2-2\alpha_2^2 \\
&\qquad+ 2\beta\log f(\beta,\alpha_1) + 2(1-\beta)\log f(1-\beta,\alpha_2).
\end{align*}
\end{definition}

The crucial computational component which is verified with a computer assisted computation is the following claim. 

\begin{assumption}\label{asm:comp}
Fix $\eps_{\ref{asm:comp}} = 10^{-25}$ and $\gamma \in [\gamma_{\mr{crit}}-\eps_{\ref{asm:comp}}, \gamma_{\mr{crit}} + \eps_{\ref{asm:comp}}]$. We have that
\[\sup_{\substack{\beta\in[0,.001]\\\alpha_1, \alpha_2\in \mb{R}}} F_2(\beta,\alpha_1,\alpha_2) = 2\sup_{\alpha\in \mb{R}}F_1(\alpha)\]
and
\[\sup_{\substack{\beta\in[.001,.999]\\\alpha_1, \alpha_2\in \mb{R}}} F_2(\beta,\alpha_1,\alpha_2) = 4\sup_{\alpha\in \mb{R}}F_1(\alpha).\]
Furthermore let $\alpha(\gamma) := \argmax_{\alpha\in\mb{R}}F_1(\alpha)$. Then the Hessian of $F_2$ evaluated at $(1/2, \alpha(\gamma), \alpha(\gamma))$ is strictly negative definite and the unique optimizer of $F_2(\beta,\alpha_1,\alpha_2)$ for $\beta\in[.001,.999]$ occurs at $(1/2, \alpha(\gamma), \alpha(\gamma))$.
\end{assumption}

As stated \cref{asm:comp} asserts an equality between two presumably transcendental quantities and a bit of care is required in verifying such an assumption. An equivalent manner of phrasing \cref{asm:comp} is that the maximum value of $F_2(\beta,\alpha_1,\alpha_2)$ occurs either at $(1/2, \alpha(\gamma), \alpha(\gamma))$, or at $(0, 0, \alpha(\gamma))$, or at $(1,\alpha(\gamma), 0)$. Naively, this suggests verifying that $F_2$ is concave in the neighborhood of these points and then performing a grid search in the remainder of the parameter space. However such a procedure appears to be computationally prohibitive and thus we opt for a more delicate approach. The crucial feature of $F_2(\beta,\alpha_1,\alpha_2)$ is that it is strictly convex for fixed $\beta$; therefore at least for fixed $\beta$ one can simply produce a pair of values $(\alpha_1,\alpha_2)$ and certify that $F_2(\beta,\alpha_1,\alpha_2)$ is sufficiently small, with sufficiently small derivative in $\alpha_1, \alpha_2$ that it cannot reach the desired supremum.

To handle a range of $\beta$ simultaneously, we produce a majorant for $\tilde{F}(\alpha_1,\alpha_2)$ which handles a range of $\beta$ uniformly. Furthermore this majorant has a particular structure that allows one to produce a near-optimum for $\tilde{F}(\alpha_1,\alpha_2)$ via a certain Newton iteration-type procedure thus allowing for efficient optimization. This allows one to essentially handle $\beta\notin [0,.001]\cup [.495,.505]\cup[.999,1]$. Due to symmetry it suffices to handle $\beta\in [0,.001]\cup [.495,.50]$.

For the first interval, one can verify that for all relevant $(\alpha_1,\alpha_2)$, we have $F(0,\alpha_1,\alpha_2)\ge F(\beta,\alpha_1,\alpha_2)$ via a direct mathematical argument; such an argument is plausible as in the neighborhood of $0$ the derivative in $\beta$ of $F(\beta,\alpha_1,\alpha_2)$ is $-\infty$ (this is a manifestation of the model exhibiting frozen $1$-RSB behavior). Finally for $\beta\in [.495,.5]$, one can use a similar numerical procedure to the initial parts handling $\beta$ not near $0,1/2,1$ to localize $(\alpha_1,\alpha_2)$ to a certain small ball around the desired optimizer and one can computationally verify via grid-search that the Hessian is negative definite.

The various numerical claims in \cref{app:computer-assist} are carried out in Python 3 using \texttt{numpy} and \texttt{scipy}, and then separately in \texttt{mpmath} which use floating point arithmetic and \texttt{python-flint} which is a rigorous library for interval arithmetic. (We note here that \texttt{python-flint} is wrapper for the interval arithmetic package \texttt{Arb} package in C++ \cite{Joh17}.) We give formal proofs for all claims in the numerical appendix except \cref{clm:prob-bound,clm:first-derivative,clm:second-derivative} which are used to verify \cref{clm:hessian-verification}; these numerical computations have only been verified in \texttt{numpy} and \texttt{scipy}, and separately in \texttt{mpmath}. This is due primarily to the authors not being aware of any package which has rigorous multidimensional numerical integration; however, the necessary bounds only require estimates to $2$ or $3$ decimal places and are not nearly as sensitive as the others. We have attached Python code in the arXiv listing; the computer assisted portions of \cref{clm:initial-interval,clm:middle-segment,clm:middle-segment-local} are carried out in \texttt{InitialLocalizationViaConvexity.py} and the computer assisted portion of \cref{clm:hessian-verification} is carried out in \texttt{VerificationOfHessianConditions.py} and \texttt{IntegralSupplement.py}.

\subsection{Notation}\label{sub:notation}
We use standard asymptotic notation throughout, as follows. For functions $f=f(n)$ and $g=g(n)$, we write $f=O(g)$ or $f \lesssim g$ to mean that there is a constant $C$ such that $|f(n)|\le C|g(n)|$ for sufficiently large $n$. Similarly, we write $f=\Omega(g)$ or $f \gtrsim g$ to mean that there is a constant $c>0$ such that $f(n)\ge c|g(n)|$ for sufficiently large $n$. Finally, we write $f\asymp g$ or $f=\Theta(g)$ to mean that $f\lesssim g$ and $g\lesssim f$, and we write $f=o(g)$ or $g=\omega(f)$ to mean that $f(n)/g(n)\to0$ as $n\to\infty$. Subscripts on asymptotic notation indicate quantities that should be treated as constants.

Furthermore logarithms are natural unless a base is specified. For an integer $k\ge 0$ we write $\log^{(k)}x$ for the $k$ times iterated logarithm $\log\cdots\log x$. Then $\log^\ast x$ for $x\ge 1$ is defined as the unique nonnegative integer $k$ so that $\log^{(k)}x\in[1,e)$. Finally, we write $C_{\ref{thm:boolean-function}}$ or $c_{\ref{lem:trichotomy}}$ to mean that there is such a positive absolute constant which makes the relevant proposition hold as stated, but we do not care to specify an explicit value.

\subsection{Organization of the paper}
In \cref{sec:boolean-functions} we prove the main vertex adapted isoperimetry result \cref{thm:boolean-function}. In \cref{sec:setup} we reduce \cref{thm:main} to a pair of moment computations (\cref{lem:first-moment,lem:second-moment}). In \cref{sec:degree-models} we provide the necessary comparison of degree sequence to a degree sequence models. In \cref{sec:atypical-switching} we use switchings to eliminate the contributions for degree sequences which are atypical. In \cref{sec:convex}, we state and prove the necessary log-convexity claims which are used in the proof of \cref{thm:main}. Finally in \cref{sec:moment} we prove \cref{lem:first-moment,lem:second-moment}.

In \cref{app:binomial} we compute precisely various tails of binomial coefficients which are used in the main moments claims and in \cref{app:computer-assist} we carry out the mathematical part of computer-assisted verification of \cref{asm:comp}.

\subsection{Concurrent work}
In concurrent work, Dandi, Gamarnik and Zdeborov\'{a} prove an analogue of \cref{thm:main}, allowing for $o(n)$ exceptional vertices. The proof proceeds via the Lindeberg exchange method to reduce the problem to a question on sparse graphs. Here they then extend work of Gamarnik and Li \cite{GL18} to prove the desired result for sparse graphs; here sparsity is used in a crucial manner as it allows one to compute various probabilities in the configuration model.

\subsection*{Acknowledgments}
The second and third authors thank David Gamarnik for bringing \cite{GL18} to our attention as this served as the starting point for our work. The third author thanks George Barbulescu and Naveen Rameen for help installing the \texttt{python-flint} library and for helping write multiprocessing code to improve the runtime.

\section{Concentration of Friendliness via Boolean Functions}\label{sec:boolean-functions}
We now prove the desired expansion result regarding subset of graphs under the metric of max-degree differences at a vertex.
\begin{definition}\label{def:dist}
For any two graphs on a fixed (labeled) vertex set $V$, let 
\[d(G,H) = \Delta(G\triangle H),\]
where $\triangle$ denotes taking the graph with edge set equal to the symmetric difference of the given edge sets, namely, $(E(G)\setminus E(H))\cup(E(H)\setminus E(G))$, and where $\Delta$ denotes maximum degree.
\end{definition}
Note that $d(G,H)$ trivially defines a metric on the set of graphs. The main result of this section is the following expansion result.
\begin{theorem}\label{thm:boolean-function}
Let $\mc{G}$ be a family of graphs on a labeled vertex set $V$ that is $\mf{S}_n$-invariant \footnote{$\mc{G}$ can be equivalently be viewed as a family of undirected graphs.} and let $\mu(\cdot)$ be the uniform measure on labeled graphs on $n$ vertices\footnote{Equivalently, $\mu(\mc{S}) = \mb{P}_{G\sim\mb{G}(n,1/2)}[G\in\mc{S}]$}. Let $\eps\in(0,1)$, possibly depending on $n$. Suppose that $\mu(\mc{G})\ge\eps$ and let 
\[\mc{G}_T = \{H\colon\min_{G\in\mc{G}}d(G,H)\le T\}.\]
Then 
\[\mu(\mc{G}_{C_{\ref{thm:boolean-function}}\log(1/\eps)\sqrt{n/\log n}})\ge 1-\eps.\]
\end{theorem}
\begin{remark}
This theorem is optimal up to a constant factor (both in $n$ and $\eps$ dependence) due to considering graphs in $\mb{G}(n,1/2)$ with maximum degree below a fixed threshold, using the result of Bollob\'as \cite{Bol80} on this distribution.
\end{remark}
\begin{remark}
An analogue of this result holds for properties of bipartite graphs with sides $n,m$ satisfying $n=\Theta(m)$ which are $\mf{S}_n\times\mf{S}_m$-invariant.
\end{remark}

We first make a preliminary reduction.
\begin{proposition}\label{prop:boolean-dyadic}
Let $\eps\in (0,1/4]$. Let $\mc{G}$ be a family of graphs on a labeled vertex set $V$ that is $\mf{S}_n$-invariant and let $\mu(\cdot)$ be the uniform measure on labeled graphs on $n$ vertices. Suppose that $\mu(\mc{G})\ge \eps$ and let 
\[\mc{G}_{T} = \{H\colon\min_{G\in \mc{G}}d(G,H)\le T\}.\] Then 
\[\mu(\mc{G}_{C_{\ref{prop:boolean-dyadic}}\sqrt{n/\log n}})\ge 2\eps.\]
\end{proposition}
We briefly sketch the reduction.
\begin{proof}[Proof of \cref{thm:boolean-function} given \cref{prop:boolean-dyadic}]
Iterating \cref{prop:boolean-dyadic} allows one to immediately prove $\mu(\mc{G}_{C\log(1/\eps)\sqrt{n/\log n}})\ge 1/2$ for an appropriately chosen absolute constant $C$. To boost from probability $1/2$ to $1-\eps$, let $\mc{G}_0 = \mc{G}_{C\log(1/\eps)\sqrt{n/\log n}}$. Suppose that $\mu(\mc{G}_0) = 1-\delta$ for $\delta\in(0,1/2]$. Consider the graphs which are distance at least $C\sqrt{n/\log n}$ in distance away from $\mc{G}_0$, call it $\mc{G}_1$. Applying \cref{prop:boolean-dyadic} implies that $\mu(\mc{G}_1)\le\delta/2$ (since the $C\sqrt{n/\log n}$-neighborhood is disjoint from $\mc{G}_0$). Iterating this procedure $O(\log(1/\eps))$ times gives the claimed result.
\end{proof}

To prove \cref{prop:boolean-dyadic} we will require a number of standard definitions from Boolean function analysis.
\begin{definition}
Given a function $f\colon\{0,1\}^{n}\to\{0,1\}$, define the \emph{influence} of the $i$th variable to be $I_i[f]= \mb{P}[f(x)\neq f(x\oplus e_i)]$, where $e_i$ is the $i$th basis vector and $\oplus$ is $\mb{F}_2$ addition. Furthermore define the \emph{total influence} of a function $f$ to be $I[f] = \sum_{i\in [n]}I_i[f]$. Finally define the \emph{positive sensitivity} of a function $f$ at a point $x$ to be
\[s_f^{+}(x) = 
\begin{cases}
0&\text{ if } f(x) = 1 \\
\sum_{i\in[n]}f(x\oplus e_i) &\text{ if } f(x) = 0.
\end{cases}
\]
\end{definition}

We will need the following variant of a strengthening of Talagrand's inequality, which follows from the proof given by Eldan, Kindler, Lifshitz, and the first author \cite[Theorem~1.2]{EKLM22}. The only change required is noting that \cite[Lemma~3.1]{EKLM22} can be replaced with $\mb{E}\bigg[\sqrt{s_f^{+}(x)}\bigg]\ge \sqrt{\sum_{i\in[n]}\hat{f}(\{i\})^2}$.
\begin{theorem}\label{thm:Talagrand}
For any function $f\colon\{0,1\}^{n}\to\{0,1\}$,
\[\mb{E}\bigg[\sqrt{s_f^{+}(x)}\bigg]\gtrsim \on{var}(f)\cdot\sqrt{\log\bigg(1+\frac{1}{\sum_{i\in[n]}I_i[f]^2}\bigg)}.\]
In particular, if $\sum_{i\in[n]}I_i[f]^2\le n^{-1/4}$ then
\[\mb{E}\bigg[\sqrt{s_f^{+}(x)}\bigg]\gtrsim \on{var}(f)\sqrt{\log n}.\]
\end{theorem}

We in fact will only apply \cref{thm:Talagrand} for functions $f$ on the set of graphs $\{0,1\}^{\binom{[n]}{2}}$ which are symmetric under relabeling the vertices. Therefore we have that $I_{(i_1,i_2)}[f]$ is independent of the pair $(i_1,i_2)$. From this we now derive the following trichotomy for functions $f$ which are symmetric.

\begin{lemma}\label{lem:trichotomy}
Let $\eps\in(0,1/4]$ and $f\colon\{0,1\}^{\binom{[n]}{2}}\to\{0,1\}$ such that $\mb{E}f\in[\eps,1-\eps]$. Then if $f$ is symmetric under the natural $\mf{S}_n$-action, then one of the following holds, where we sample $x$ from the uniform measure on labeled graphs $\mu$:
\begin{enumerate}[{\bfseries{Inf\arabic{enumi}}}]
    \item\label{Inf1} $\mb{P}[s_f^{+}(x)>0]\ge c_{\ref{lem:trichotomy}}\eps\sqrt{\log n/n}$
    \item\label{Inf2} There is an integer $j\in[\lfloor(\log_2n)/2\rfloor,\lfloor\log_2\binom{n}{2}\rfloor]$ such that
    \[2^{j}\mb{P}[s_f^{+}(x)\in[2^{j},2^{j+1}]]\ge c_{\ref{lem:trichotomy}}\eps\sqrt{n}(\log n)^{5}.\]
    \item\label{Inf3} There are integers $\ell\in[2,\log^\ast n-1]$, $j\in[\lfloor\log_2n + 30\log^{(\ell+1)}n\rfloor,\lfloor\log_2n + 30\log^{(\ell)}n\rfloor]$ such that
    \[2^{j/2}\mb{P}[s_f^{+}(x)\in[2^j,2^{j+1}]]\ge c_{\ref{lem:trichotomy}}\eps\sqrt{\log n}/(\log^{(\ell)}n)^2.\]
\end{enumerate}
\end{lemma}
\begin{proof}
Note that we may assume $n$ is larger than an absolute constant (not dependent on $\eps$). Suppose that \cref{Inf1} fails, i.e., $\mb{P}[s_f^{+}(x)>0] < c_{\ref{lem:trichotomy}}\eps\sqrt{\log n/n}$ where the constant will be chosen later. Furthermore note that
\[I[f]\asymp\sum_{j=0}^{\lfloor\log_2\binom{n}{2}\rfloor}2^{j}\mb{P}[s_f^{+}(x)\in[2^j,2^{j+1}]]\lesssim\sum_{j=\lfloor(\log_2n)/2\rfloor}^{\lfloor\log_2\binom{n}{2}\rfloor}2^{j}\mb{P}[s_f^{+}(x)\in[2^j,2^{j+1}]]+n^{1/2}.\]
If $I[f]\ge n^{1/2}(\log n)^6$ then there exists $j\in[\lfloor(\log_2n)/2\rfloor,\lfloor\log_2\binom{n}{2}\rfloor]$ such that
\[2^{j}\mb{P}[s_f^{+}(x)\in[2^j,2^{j+1}]]\gtrsim\sqrt{n}(\log n)^5,\]
which will satisfy \cref{Inf2} for appropriately chosen $c_{\ref{lem:trichotomy}} > 0$.

Thus, if we further assume that \cref{Inf2} fails, then in the remaining case we may assume $I[f]\le n^{1/2}(\log n)^6$, $\mb{P}[s_f^{+}(x)>0]\le c_{\ref{lem:trichotomy}}\eps\sqrt{\log n/n}$, and $w_j := 2^{j}\mb{P}[s_f^{+}(x)\in[2^j,2^{j+1}]]\le c_{\ref{lem:trichotomy}}\eps\sqrt{n}(\log n)^{5}$ for all $j\in[\lfloor(\log_2n)/2\rfloor,\lfloor\log_2\binom{n}{2}\rfloor]$. Using that $f$ is symmetric we have
\[\sum_{(i_1,i_2)\in\binom{[n]}{2}}I_{(i_1,i_2)}[f]^2 = \binom{n}{2}^{-1}I[f]^2\le n^{-1/2}.\]
Thus we have by \cref{thm:Talagrand} and the condition on $\mb{E}f$ that $\mb{E}[\sqrt{s_f^{+}(x)}]\gtrsim\eps\sqrt{\log n}$. Note that 
\[\mb{E}[\sqrt{s_f^{+}(x)}]\asymp\sum_{j=0}^{\lfloor\log_2\binom{n}{2}\rfloor}2^{-j/2}w_j\]
and that
\[\sum_{j=0}^{\lfloor\log_2n+100\rfloor}2^{-j/2}w_j\le 2^{(\log_2n+100)/2}\mb{P}[s_{f}^{+}(x)>0]\lesssim c_{\ref{lem:trichotomy}}\eps\sqrt{\log n}.\]
Furthermore
\[\sum_{j=\lfloor\log_2n+30\log\log n\rfloor}^{\lfloor\log_2\binom{n}{2}\rfloor}2^{-j/2}w_j\le 2^{-(\log_2n+30\log\log n)/2} c_{\ref{lem:trichotomy}}\eps\sqrt{n}(\log n)^{5}\lesssim c_{\ref{lem:trichotomy}}\eps(\log n)^{-1}.\]
Combining these three inequalities we have (choosing $c_{\ref{lem:trichotomy}}$ appropriately small) that
\begin{equation}\label{eq:large-intervals}
\sum_{j=\lfloor\log_2n+100\rfloor}^{\lfloor\log_2n+30\log\log n\rfloor}2^{-j/2}w_j\gtrsim \eps\sqrt{\log n}.
\end{equation}
Now $\ell\in[2,\log^\ast n-1]$, $j\in[\lfloor\log n+30\log^{(\ell+1)} n\rfloor,\lfloor\log n+30\log^{(\ell)} n\rfloor]$ cover the interval of values of $j$ between $\lfloor\log_2n+100\rfloor$ and $\lfloor\log_2n+30\log\log n\rfloor$. Thus, if none of the desired indices satisfy \cref{Inf3} then the total contribution on the left of \cref{eq:large-intervals} would be bounded via
\[\sum_{j=\lfloor\log_2n+100\rfloor}^{\lfloor\log_2n+30\log\log n\rfloor}2^{-j/2}w_j\le\sum_{\ell\in[2,\log^\ast n-1]}\frac{c_{\ref{lem:trichotomy}}\eps\sqrt{\log n}}{\log^{(\ell)}n}\lesssim c_{\ref{lem:trichotomy}}\eps\sqrt{\log n}\]
which provides a contradiction to \cref{eq:large-intervals} if $c_{\ref{lem:trichotomy}}$ is sufficiently small.
\end{proof}

Given these facts we are now in position to prove \cref{prop:boolean-dyadic} and hence \cref{thm:boolean-function}.
\begin{proof}[Proof of \cref{prop:boolean-dyadic}]
We define a series of expansions of $\mc{F}_0 = \mc{G}$ as follows, each of being an $\mf{S}_n$-symmetric Boolean function. For each $i\ge 1$, if $\mu(\mc{F}_{i-1})\ge 2\eps$ we stop. Otherwise suppose that $\mu(\mc{F}_{i-1})\in[\eps,2\eps)$. Let $f_{i-1} = \mbm{1}_{\mc{F}_{i-1}}$ and note that $f_{i-1}$ is a symmetric Boolean function.
\begin{enumerate}[{\bfseries{Alg\arabic{enumi}}}]
    \item\label{Alg1} If $\mb{P}[s_{f_{i-1}}^{+}(x)>0]\ge c_{\ref{lem:trichotomy}}\eps\sqrt{\log n/n}$, define $\mc{F}_i=\mc{F}_{i-1}\cup\{x\colon s_{f_{i-1}}^{+}(x)>0\}$.
    \item\label{Alg2} Else if there is $j\in[\lfloor(\log_2n)/2\rfloor,\lfloor\log_2\binom{n}{2}\rfloor]$ satisfying
    \[2^j\mb{P}[s_{f_{i-1}}^{+}(x)\in[2^j,2^{j+1}]]\ge c_{\ref{lem:trichotomy}}\eps\sqrt{n}(\log n)^5\]
    then define $\mc{F}_i = \mc{F}_{i-1}\cup\{x\colon s_{f_{i-1}}^{+}(x)\in[2^j,2^{j+1}]\}$.
    \item\label{Alg3} Else if there are $\ell\in[2,\log^\ast n-1]$, $j\in[\lfloor\log_2n + 30\log^{(\ell+1)} n\rfloor,\lfloor\log_2n + 30\log^{(\ell)} n\rfloor]$ such that 
    \[2^{j/2}\mb{P}[s_{f_{i-1}}^{+}(x)\in[2^j,2^{j+1}]]\ge c_{\ref{lem:trichotomy}}\eps\sqrt{\log n}/(\log^{(\ell)}n)^2\] then define $\mc{F}_i = \mc{F}_{i-1}\cup\{x\colon s_{f_{i-1}}^{+}(x)\in[2^j,2^{j+1}]\}$.
\end{enumerate}
By \cref{lem:trichotomy}, at least one of these situations occurs so this is well-defined. Note that at each stage of this expansion, symmetry of $f_{i-1}$ implies that $f_i$ is also symmetric. Combined with the fact $2\eps\le 1-\eps$, we see the application of \cref{lem:trichotomy} is always valid. Additionally, note that every $H\in\mc{F}_i$ is either in $\mc{F}_{i-1}$ or obtained from a graph in $\mc{F}_{i-1}$ by flipping an edge in or out.

Note that $s_{f_{i-1}}^+(x) > 0$ is only possible for $x\notin\mc{F}_{i-1}$. Thus $\mu(\mc{F}_i\setminus\mc{F}_{i-1})\gtrsim\eps\sqrt{\log n/n}$ if \cref{Alg1} occurs, $\mu(\mc{F}_i\setminus\mc{F}_{i-1})\gtrsim\eps\sqrt{n}(\log n)^5/2^j$ if \cref{Alg2} with parameter $j$ occurs, and $\mu(\mc{F}_i\setminus\mc{F}_{i-1})\gtrsim\eps\sqrt{\log n}/(2^{j/2}(\log^{(\ell)}n)^2)$ if \cref{Alg3} with parameters $\ell$ and $j$ occurs. 

Since we stop the process once the density is at least $2\eps$, there must be at most $O(\sqrt{n/\log n})$ steps of \cref{Alg1}, at most $O(2^{j}/(\sqrt{n}(\log n)^{5}) + 1)$ steps of \cref{Alg2} for each possible parameter $j$, and at most $O(2^{j/2}(\log^{(\ell)}n)^2/\sqrt{\log n})$ steps of \cref{Alg3} for each possible choice of parameters $\ell, j$, where the implicit constants are independent of $\eps$. Thus the process described above must eventually stop at some time $t$ with $\mu(\mc{F}_t)\ge 2\eps$. We now prove any graph in $\mc{F}_t$ is close to $\mc{F}_0$ in graph distance in the sense of \cref{def:dist}, which will finish.

Let us consider some $H\in\mc{F}_t$. Let $H_t=H$. Additionally, let $R_j$ for $j\in[\lfloor(\log_2n)/2\rfloor,\lfloor\log_2\binom{n}{2}\rfloor]$ and $R_{\ell,j}$ for $\ell\in[2,\log^\ast n-1]$ and $j\in[\lfloor\log_2n+30\log^{(\ell+1)}n\rfloor,\lfloor\log_2n+30\log^{(\ell)}n\rfloor]$ be ``record graphs'' which start as the empty graph on $V=V(H)$. For $i\in[t]$ in decreasing order, we perform the following:
\begin{enumerate}[{\bfseries{Alg\arabic{enumi}'}}]
    \setcounter{enumi}{-1}
    \item\label{Alg0'} If $H_i\in\mc{F}_{i-1}$ let $H_{i-1}=H_i$. Decrement $i$ and restart the loop.
    \item\label{Alg1'} Otherwise, if further $\mc{F}_i\setminus\mc{F}_{i-1}$ was generated using \cref{Alg1}, let $H_{i-1}\in\mc{F}_{i-1}$ be chosen to be one of the graphs distance $1$ (a single edge flip) from $H_i$ arbitrarily.
    \item\label{Alg2'} Else if $\mc{F}_i\setminus\mc{F}_{i-1}$ was generated using \cref{Alg2} with parameter $j$, let $H_{i-1}\in\mc{F}_{i-1}$ be chosen to be one of the graphs at distance $1$ from $H_i$. If possible, we choose it with the additional property that if $e_i$ is the differing edge, adding $e_i$ to record graph $R_j$ does not increase the maximum degree of $R_j$. (If $e_i$ is already present in $R_j$, adding it in leaves $R_j$ the same.)
    \item\label{Alg3'} Else if $\mc{F}_i\setminus\mc{F}_{i-1}$ was generated using \cref{Alg3} with parameters $\ell,j$, let $H_{i-1}\in\mc{F}_{i-1}$ be chosen to be one of the graphs at distance $1$ from $H_i$. If possible, we choose it with the additional property that if $e_i$ is the differing edge, adding $e_i$ to record graph $R_j$ does not increase the maximum degree of $R_j$.
\end{enumerate}
This is clearly well-defined since every $\mc{F}_i\setminus\mc{F}_{i-1}$ is generated via one of \cref{Alg1,Alg2,Alg3} and also any graph in $\mc{F}_i\setminus\mc{F}_{i-1}$ is generated by flipping an edge of some graph in $\mc{F}_{i-1}$. Furthermore, we have $H_0\in\mc{F}_0=\mc{G}$.

For all $i\in[t]$, let $e_i$ be the differing edge between $H_{i-1},H_i$. Clearly $H_0\triangle H_t$ is a subgraph of the union of all these edges, call it $R$ ($H_0\triangle H_t$ may be a strict subgraph of $R$ if say at one step we flip an edge, and at another step flip it again). It suffices to bound the maximum degree of $R$. Let $R^\ast$ be the union of edges from \cref{Alg1}, and let $R_j^\ast,R_{\ell,j}^\ast$ be the final values of the record graphs $R_j,R_{\ell,j}$ for the various parameter choices.

The edges $e_i$ corresponding to \cref{Alg1} are at most $O(\sqrt{n/\log n})$ in number by the earlier analysis and hence $\Delta(R^\ast)=O(\sqrt{n/\log n})$.

For \cref{Alg2} with parameter $j$, let $k = 2^j$ and recall that there are at most $O(k/(\sqrt{n}(\log n)^5)+1)$ corresponding relevant steps $i$. For $j$ such that $k\le n(\log n)^2$, this amounts to $O(\sqrt{n}/(\log n)^3)$ total edges, so certainly $\Delta(R_j^\ast)=O(\sqrt{n}/(\log n)^3)$.

For $j$ such that $n(\log n)^2<k\le n^2$, consider any relevant step $i$. We have a graph $H_i\in\mc{F}_i$ and a current record graph $R_j$, and we choose $H_{i-1}\in\mc{F}_{i-1}$ at distance $1$ from $H_i$. Let $\Delta=\lfloor\sqrt{n}/(\log n)^2\rfloor$. If $\Delta(R_j)<\Delta$ before step $i$ then clearly $\Delta(R_j)\le\Delta$ after step $i$. Otherwise $\Delta(R_j)\ge\Delta$ before step $i$. By the definition of \cref{Alg2}, there are at least $k$ choices of $H_{i-1}\in\mc{F}_{i-1}$ at distance $1$ to $H_i$. Equivalently, there are at least $k$ choices of $e_i$. If all of these choices have at least one endpoint with degree at least $\Delta$ in $R_j$, then $R_j$ has at least $k/n$ vertices of degree at least $\Delta$, and hence at least $(k\Delta)/(2n)$ total edges. This contradicts our bound of $O(k/(\sqrt{n}(\log n)^5)+1)$ on the number of steps of \cref{Alg2} with parameter $j$. That is, there will always be a choice which maintains the maximum degree once if it ever hits $\Delta$. Thus we see that $\Delta(R_j^\ast)\le\Delta=O(\sqrt{n}/(\log n)^2)$.

Finally for \cref{Alg3} with parameters $\ell,j$, let $k=2^j$ and recall that earlier analysis showed that there are at most $O(\sqrt{k}(\log^{(\ell)}n)^2/\sqrt{\log n})$ total relevant steps $i$. We can perform a similar argument as the previous case. Let $\Delta_{\ell,j}=Kn(\log^{(\ell)}n)^4/\sqrt{2^j\log n}$ where $K$ is a sufficiently large absolute constant. By the definition of \cref{Alg3}, there are at least $k$ choices of $H_{i-1}$ or equivalently edges $e_i$. If all of these choices have at least one endpoint with degree at least $\Delta_{\ell,j}$ in $R_{\ell,j}$, then $R_{\ell,j}$ has at least $(k\Delta_{\ell,j})/(2n)$ total edges. This violates our bound on the number of relevant steps $i$ as long as $K$ was chosen sufficiently large. Hence a similar argument as in the analysis of \cref{Alg2} shows that $\Delta(R_{\ell,j}^\ast)\le\Delta_{\ell,j}$.

Finally,
\begin{align*}
\Delta(R)&\le\Delta(R^\ast)+\sum_{j=\lfloor(\log_2n)/2\rfloor}^{\lfloor\log_2\binom{n}{2}\rfloor}\Delta(R_j^\ast)+\sum_{\ell=2}^{\log^\ast n-1}\sum_{j=\lfloor\log_2n+30\log^{(\ell+1)}n\rfloor}^{\lfloor\log_2n+30\log^{(\ell)}n\rfloor}\Delta(R_{\ell,j}^\ast)\\
&\lesssim\sqrt{\frac{n}{\log n}}+(\log n)\cdot\frac{\sqrt{n}}{(\log n)^2}+\sum_{\ell=2}^{\log^\ast n-1}\frac{n(\log^{(\ell)}n)^4}{\sqrt{2^{\lfloor\log_2n+30\log^{(\ell+1)}n\rfloor}\log n}}\\
&\lesssim\sqrt{\frac{n}{\log n}}+\sqrt{\frac{n}{\log n}}\sum_{\ell=2}^{\log^\ast n-1}\frac{(\log^{(\ell)}n)^4}{2^{15\log^{(\ell+1)}n}}\lesssim\sqrt{\frac{n}{\log n}}\bigg(1+\sum_{\ell=2}^{\log^\ast n-1}\frac{1}{\log^{(\ell)}n}\bigg)\lesssim\sqrt{\frac{n}{\log n}}.
\end{align*}
We are done, since this establishes that every $H\in\mc{F}_t$ satisfies $\Delta(H\triangle H_0)=O(\sqrt{n/\log n})$ for some $H\in\mc{G}$, and since $\mu(\mc{F}_t)\ge 2\eps$.
\end{proof}

\subsection{Applications beyond the present work}\label{sub:boolean-application}
We finally elaborate on a remark made in \cref{sub:vertex-iso} stating that \cref{thm:boolean-function} may be used to derive the threshold for the capacity of the symmetric binary perceptron. Recall that the threshold for the symmetric binary perceptron with parameter $\kappa$ (with Bernoulli disorder) is defined via taking a series of random vectors $x_1,\ldots\in\{\pm 1\}^n$ and asking for the first time $\tau$ at which there does not exist $y\in\{\pm 1\}^n$ such that $|\sang{x_i, y}|\le \kappa\sqrt{n}$ for all $i\in[\tau]$. The work of Aubins, Perkins, and Zdeborov\'{a} \cite{APZ19} establish (based on the second moment method) the existence of a continuous function $\alpha(\kappa)$ such that with high probability the threshold is less than $(\alpha(\kappa)+o(1))n$ and with $\Omega(1)$ probability the threshold is at least $(\alpha(\kappa)-o(1))n$, if the disorder is Gaussian, and their technique applies in the case of Bernoulli disorder as well. Treating the vectors $x_1,\ldots,x_m$ as an $m\times n$ adjacency matrix of a bipartite graph, the remark following \cref{thm:boolean-function} regarding bipartite graphs immediately implies that at $(\alpha(\kappa)-o(1))n$ whp there exists a vector $y$ such that $|\sang{x_i,y}|\le(\kappa+o(1))\sqrt{n}$. Via continuity of $\alpha(\kappa)$, this establishes the sharp threshold for the symmetric perceptron model. We also note that an essentially identical argument proves that the number of solutions for a given $m$ concentrates on the exponenetial scale. (We note sharpness and counting solutions was shown without \cref{thm:boolean-function} by Abbe, Li, and Sly \cite{ALS21} for Bernoulli disorder and Perkins and Xu \cite{PX21} for Gaussian disorder.)

We note that in general arguments in the flavor of \cref{thm:boolean-function} in the context of perceptron model can prove sharp thresholds given ``monotonicity of the activation functions with respect to deformation'' in an appropriate sense. If one can prove a priori that the (not necessarily sharp) threshold is continuous with respect to an appropriate ``deformation parameter'', then \cref{thm:boolean-function} or similar arguments may apply to achieve sharpness. General results of this nature were established in work of Xu \cite{Xu21}, with a substantially simplified proof given in work of Nakajima and Sun \cite{NS22}. We note here that the techniques of Perkins and Xu \cite{Xu21} and  Nakajima and Sun \cite{NS22} rely on the ``row by row'' independence of the perceptron model and hence appear fundementally unsuited to give sharpness for our situation. The work of Xu \cite{Xu21} more similarly relies on Boolean functions and symmetry; however it requires as input deep work of Hatami \cite{Hat21}. Finally the work of Abbe, Li, and Sly \cite{ALS21} is based on an analogue of small subgraph conditioning for dense graphs and at the very least requires one to be able to compute the moments underlying \cref{thm:main} to a multiplicative $1+o(1)$ accuracy. This appears to be nigh computationally infeasible. Furthermore, one would need to show asymptotic normality for subgraph counts in various associated degree-constrained models (which has only recently seen progress \cite{SS22clt}).

\section{Setup for Second Moment Calculation}\label{sec:setup}
We now define the necessary definitions for the remainder of the paper. 
\begin{definition}\label{def:setup}
Let $G\sim \mb{G}(2n,1/2)$ and $V$ denote the vertex set of $G$. Furthermore, let $X_{\gamma}$ denote the number of equipartitions $V = Y_1\cup Y_2$ of $G$ such that for every vertex $y\in Y_1$ we have $\deg_{Y_1}(y)\ge\deg_{Y_2}(y)+\gamma\sqrt{n}$ and for every $y\in Y_2$ we have $\deg_{Y_2}(y)\ge\deg_{Y_1}(y)+\gamma\sqrt{n}$. We will always consider $\gamma\in[-1,1]$ with the property that $\gamma\sqrt{n}\in\mb{Z}$.
\end{definition}

The two crucial computations give a precise (up to constant) asymptotic for $\mb{E}[X_{\gamma}]$ and $\mb{E}[X_{\gamma}^2]$.  
\begin{lemma}\label{lem:first-moment}
Suppose that $|\gamma|\le 1$ and $\gamma\sqrt{n}\in\mb{Z}$. Then with notation as in \cref{def:setup},
\[\mb{E}X_\gamma\asymp\binom{2n}{n}\bigg(\sup_{\alpha\in \mb{R}} \exp(-\alpha^2)\mb{P}_{Z\sim \mc{N}(0,1)}[Z\ge(\gamma + \alpha)\sqrt{2}]\bigg)^{2n}=4^{-n}\binom{2n}{n}\exp(2n\sup_{\alpha\in\mb{R}}F_1(\alpha)).\]
\end{lemma}
\begin{lemma}\label{lem:second-moment}
Suppose that \cref{asm:comp} holds. Fix $\eta>0$ and suppose $\gamma \in [\gamma_{\mr{crit}}-\eps_{\ref{asm:comp}}, \gamma_{\mr{crit}} + \eps_{\ref{asm:comp}}]$ and $\gamma\sqrt{n}\in\mb{Z}$. Then with notation as in \cref{def:setup},
\[\mb{E}X_\gamma^2\lesssim_\eta(1+\eta)^n\mb{E}X_\gamma+16^{-n}\binom{2n}{n}^2\exp\bigg(n\sup_{\substack{\beta\in[0,1]\\\alpha_1,\alpha_2\in \mb{R}}} F_2(\beta,\alpha_1,\alpha_2)\bigg).\]
\end{lemma}
\begin{remark}
We certainly believe that \cref{lem:second-moment} holds for all $\gamma$; the restriction on the range of $\gamma$ in consideration comes precisely from \cref{asm:comp}.
\end{remark}

We now prove \cref{thm:main} given \cref{asm:comp,lem:first-moment,lem:second-moment}. In the following argument we will not be too sensitive about $\gamma\sqrt{n}$ being an integer as it is inconsequential here, but we note that in \cref{sec:moment} where \cref{lem:first-moment,lem:second-moment} are proven one must be careful about integrality concerns.
\begin{proof}
Choose $\eps>0$ sufficiently small and note direct numerical computation verifies $0\le\gamma_{\mr{crit}}\le 1/2$. We first prove that with high probability, $G\sim\mb{G}(2n,1/2)$ has no equipartitions where each vertex has at least $(\gamma_{\mr{crit}}/\sqrt{2}+\eps)\sqrt{2n}$ more vertices in its own part than in the other side. Note that $\gamma\mapsto\sup_{\alpha\in\mb{R}}(-\alpha^2 + \mb{P}_{Z\sim{\mc{N}(0,1)}}[Z\ge(\gamma + \alpha)\sqrt{2}])$ is trivially a strictly decreasing function. Therefore by \cref{lem:first-moment,def:const}, the number of $(\gamma_{\mr{crit}}/\sqrt{2}+\eps)\sqrt{2n}$-friendly equipartitions in expectation is exponentially small. By Markov's inequality, there are no such equipartitions with probability $1-\exp(-\Omega_\eps(n))$. The analogous result for odd numbers can be deduced as follows: if there exists a sufficiently friendly partition in $\mb{G}(2n-1,1/2)$ then adding a vertex with uniformly random neighbors to the smaller part, making the total number of vertices even, gives a partition which is $(\gamma_{\mr{crit}}/\sqrt{2}+\eps/2)\sqrt{2n}$-friendly with probability $\Omega(1)$ (namely, approximately the probability that the extra vertex is sufficiently friendly), a contradiction to the even case.

We next prove that with high probability, $G\sim\mb{G}(2n,1/2)$ has a $(\gamma_{\mr{crit}}/\sqrt{2}-\eps)\sqrt{2n}$-friendly equipartition whp. Let $\gamma=\gamma_{\mr{crit}}-\eps/2$. By $\sup_{\alpha\in\mb{R}}F_1(\alpha)$ decreasing as $\gamma$ increases and the definition of $\gamma_{\mr{crit}}$, along with \cref{asm:comp}, we know that $\sup_{\alpha\in\mb{R}}F_1(\alpha)>0$ hence $\sup_{\substack{\beta\in[0,1]\\\alpha_1,\alpha_2\in\mb{R}}}F_2(\beta,\alpha_1,\alpha_2)=4F_1(\alpha)$. Then \cref{lem:first-moment,lem:second-moment} with $\eta$ sufficiently small in terms of $\eps$ demonstrates
\[\mb{E}X_\gamma^2\lesssim_\eps(\mb{E}X_\gamma)^2.\]
Therefore, since $X_\gamma$ takes on nonnegative integer values, we have $\mb{P}[X_\gamma\neq 0]=\Omega_\eps(1)$. That is, we have at least a constant lower bound on the desired probability. The next step is to use the techniques of \cref{sec:boolean-functions} to boost this probability to $1-o(1)$.

Let $\mc{P}$ be the property of graphs on $2n$ vertices defined by having a $(\gamma_{\mr{crit}}-\eps/2)$-friendly bisection. We have $\mu(\mc{P})=\Omega_\eps(1)$ by the above analysis, borrowing the notation from \cref{sec:boolean-functions}. Then \cref{thm:boolean-function} implies $\mu(\mc{P}_{\sqrt{n}(\log n)^{-1/4}})=1-o(1)$. But notice that under the metric defined in \cref{def:dist}, at every vertex the set of neighbors differs by at most $\sqrt{n}(\log n)^{-1/4}$ additions and deletions when moving from a graph in $\mc{P}$ to $\mc{P}_{\sqrt{n}(\log n)^{-1/4}}$. Therefore a $(\gamma_{\mr{crit}}-\eps/2)$-friendly bisection guaranteed in $\mc{P}$ will certainly remain at least $(\gamma_{\mr{crit}}-\eps)$-friendly in $\mc{P}_{\sqrt{n}(\log n)^{-1/4}}$. The result follows for graphs with $2n$ vertices.

Finally, to prove this part of \cref{thm:main} for graphs with $2n-1$ vertices, simply add a vertex with uniformly random neighborhood to form $\mb{G}(2n,1/2)$, find a $(\gamma_{\mr{crit}}-\eps/2)$-friendly bisection, and then remove the extra vertex and check that the result is still $(\gamma_{\mr{crit}}-\eps)$-friendly. We are done.
\end{proof}

In the remainder of the body of the paper we focus on proving \cref{lem:first-moment,lem:second-moment}.

\section{Degree Models}\label{sec:degree-models}
\subsection{Degree enumeration}\label{sub:enumeration}
We will first require the following results regarding degree enumeration due to \cite{MW90} and \cite{CGM08}, respectively.
\begin{theorem}[{\cite{MW90}}]\label{thm:enum-graph}
There exists a fixed constant $\eps=\eps_{\ref{thm:enum-graph}}>0$ such that the following holds. Consider a sequence $\mbf{d} = (d_1, \dots, d_n)$ with even sum such that, writing $\ol{d}=(1/n)\sum_{i=1}^n d_i$, we have
\begin{itemize}
    \item $|d_i-\ol{d}|\le n^{1/2+\eps}$ for $1\le i\le n$, and
    \item $\ol{d}\ge n/\log n$.
\end{itemize}
Writing $m = \ol{d}n/2\in\mb{Z}$, $\mu = \ol{d}/(n-1)$, and $\gamma_2^2 = (1/(n-1)^2)\sum_{i=1}^n(d_i-\ol{d})^2$, the number of labeled graphs with degree sequence $\mbf{d}$ is
\begin{align*}
(1\pm O(n^{-1/4}&))  \exp\left(\frac{1}{4}-\frac{\gamma_2^2}{4\mu^2(1-\mu)^2}\right)\binom{n(n-1)/2}{m}\binom{n(n-1)}{2m}^{-1}\prod_{i=1}^n\binom{n-1}{d_i}. 
\end{align*}
\end{theorem}

\begin{theorem}[{\cite{CGM08}}]\label{thm:enum-bigraph}
There exists a fixed constant $\eps=\eps_{\ref{thm:enum-bigraph}}>0$ such that the following holds. 
Consider a pair of sequences $(\mbf{s} = (s_1, \dots, s_n), \mbf{t} = (t_1, \dots, t_m))$ with identical sums such that, writing $\ol{s}=(1/n)\sum_{i=1}^n s_i$ and $\ol{t}=(1/n)\sum_{i=1}^m t_i$, we have

\begin{itemize}
    \item $n/(\log n)^{1/2}\le m\le n(\log n)^{1/2}$,
    \item $|s_i-\ol{s}|\le n^{1/2+\eps}$ for $1\le i\le n$ and $|t_i-\ol{t}|\le m^{1/2+\eps}$ for $1\le i\le m$, and
    \item $\ol{s}\ge n/(\log n)^{1/2}$ and $\ol{t} \ge m/(\log m)^{1/2}$.
\end{itemize}

Writing $\mu = \sum_{i=1}^n s_i/(mn) = \sum_{i=1}^m t_i/(mn)$, $\gamma_2(\mbf s)^2 = (1/(mn))\sum_{i=1}^n(s_i-\ol{s})^2$ and $\gamma_2(\mbf t)^2 = (1/(mn))\sum_{i=1}^m(t_i-\ol{t})^2$, the number of labeled bipartite graphs whose partition classes have degree sequences $\mbf{s}$ and $\mbf{t}$ is
\begin{align*}
(1\pm O(n^{-1/8}&))\exp\left(-\frac{1}{2}\left(1-\frac{\gamma_2(\mbf{s})^2}{\mu(1-\mu)}\right)\left(1-\frac{\gamma_2(\mbf{t})^2}{\mu(1-\mu)}\right)\right)\binom{mn}{mn\mu}^{-1}\prod_{i=1}^n\binom{m}{s_i}\prod_{i=1}^m\binom{n}{t_i}.
\end{align*}
\end{theorem}

We now define a plethora of degree sequence models for random graphs that will be needed for the computations. At a high level, the work of McKay and Wormald \cite{MW97} and McKay and Skerman \cite{MS16} demonstrate that degrees of random graphs look independent conditional on, for example, total edge count. These models provide a way to encapsulate these facts quantitatively; however the precise result in \cite{MW97, MS16} are not sufficient for our work and we will require a number of suitable modifications. 
\begin{definition}[Degree sequence domains]\label{def:degree-sets}
Let $I_n = \{0,\ldots,n-1\}^n$, $E_n$ be the even sum sequences in this set, and $I_n^\ell$ be the sum $\ell$ sequences. We will typically denote elements of these sets by $\mbf{d}$. Let $I_{m,n} = \{0,\ldots,n\}^m\times\{0,\ldots,m\}^n$, $E_{m,n}$ be the sequences with equal sums on both sides, and $E_{m,n}^\ell$ be the sequences with equal sums $\ell$. We will typically denote elements of these sets by $\mbf{s}$ of length $m$ and $\mbf{t}$ of length $n$. We will denote random variable versions of these by capital boldface instead.
\end{definition}
\begin{definition}[True degree models]\label{def:true-model}
$\mc{D}_p^n$ is the degree sequence distribution of $\mb{G}(n,p)$, which is a random variable supported on $E_n\subseteq I_n$. $\mc{D}_p^{m,n}$ is the degree sequence distribution of a bipartite graph with $m$ vertices on one side and $n$ on the other, each edge included independently with probability $p$, which is a random variable supported on $E_{m,n}\subseteq I_{m,n}$.
\end{definition}
\begin{definition}[Independent degree models]\label{def:independent-model}
$\mc{B}_p^n$ is the distribution of $n$ independent $\mr{Bin}(n-1,p)$ random variables, supported on $I_n$. $\mc{B}_p^{m,n}$ is the distribution of $m$ independent $\mr{Bin}(n,p)$ and $n$ independent $\mr{Bin}(m,p)$ variables, supported on $I_{m,n}$.
\end{definition}
\begin{definition}[Conditioned degree models]\label{def:conditioned-model}
$\mc{E}_p^n$ is the distribution of $\mc{B}_p^n$ conditioned on having even sum, supported on $E_n$. $\mc{E}_p^{m,n}$ is the distribution of $\mc{B}_p^{m,n}$ conditioned on having equal sums on both sides, supported on $E_{m,n}$.
\end{definition}
\begin{definition}[Integrated degree models]\label{def:integrated-model}
$\mc{I}_p^n$ is the distribution sampled as follows. Sample $p'\sim\mc{N}(p,p(1-p)/(n^2-n))$, conditional on being in $(0,1)$. Then sample from $\mc{E}_{p'}^n$. $\mc{I}_p^{m,n}$ is the distribution sampled as follows. Sample $p'\sim\mc{N}(p,p(1-p)/(2mn))$, conditional on being in $(0,1)$. Then sample from $\mc{E}_{p'}^{m,n}$.
\end{definition}

The key reason these degree models will prove crucial in our analysis is that they allow, losing a constant, an estimate for the probability that a random graph or bipartite graph has a particular degree sequence:
\begin{lemma}\label{lem:integrate-p-graph}
Fix $\eps\in(0,\eps_{\ref{lem:integrate-p-graph}}]$ and $C>0$. Let a sequence $\vec{d}\in\mb{Z}^n$ be $(C,\eps)$-regular if:
\begin{enumerate}[{\bfseries{A\arabic{enumi}}}]
    \item\label{A1} $\sup_{i\in[n]}|d_i-(n-1)/2|\le n^{1/2+\eps}$,
    \item\label{A2} $\sum_{i\in[n]} |d_i - (n-1)/2|^2\le Cn^2$,
    \item\label{A3} $2|\sum_{i\in[n]}d_i$.
\end{enumerate}
Then there exists $C' = C_{\ref{lem:integrate-p-graph}}'(C,\eps)$ such that for all $(C,\eps)$-regular $\vec{d}$ we have
\[\frac{1}{C'}\le\frac{\mb{P}_{\mc{D}_{1/2}^n}[\mbf{D} = \vec{d}]}{\mb{P}_{\mc{I}_{1/2}^n}[\mbf{D} = \vec{d}]}\le C'.\]
\end{lemma}
\begin{lemma}\label{lem:integrate-p-bip-graph}
Fix $\eps\in(0,\eps_{\ref{lem:integrate-p-bip-graph}}]$, $\theta > 0$, and $C>0$. Let a pair of degree sequences $(\vec{s},\vec{t})\in\mb{Z}^{n_1}\times\mb{Z}^{n_2}$ be $(C,\theta, \eps)$-regular if:
\begin{enumerate}[{\bfseries{B\arabic{enumi}}}]
    \item\label{B1} $\theta\le n_1/n_2\le\theta^{-1}$,
    \item\label{B2} $\sum_{i\in[n_1]}s_i = \sum_{i\in[n_2]}t_i$,
    \item\label{B3} $|s_i-n_2/2|\le n_2^{1/2+\eps}$ for $i\in[n_1]$,
    \item\label{B4} $|t_i-n_1/2|\le n_1^{1/2+\eps}$ for $i\in[n_2]$,
    \item\label{B5} $\sum_{i\in[n_1]}|s_i-n_2/2|^2+\sum_{i\in[n_2]}|t_i-n_1/2|^2\le Cn_1^2$.
\end{enumerate}
Then there exists $C' = C_{\ref{lem:integrate-p-bip-graph}}'(C,\theta, \eps)$ such that for all $(C,\theta,\eps)$-regular $(\vec{s},\vec{t})$ we have
\[\frac{1}{C'}\le\frac{\mb{P}_{\mc{D}_{1/2}^{n_1,n_2}}[\mbf{S} = \vec{s}\wedge\mbf{T} = \vec{t}]}{\mb{P}_{\mc{I}_{1/2}^{n_1,n_2}}[\mbf{S} = \vec{s}\wedge\mbf{T} = \vec{t}]}\le C'.\]
\end{lemma}

These lemmas follow from \cref{thm:enum-graph,thm:enum-bigraph} and some basic computations, which we now provide. We will need the following technical binomial coefficient estimate.
\begin{lemma}\label{lem:binom-middle}
For $|t|\le n^{4/5}$, 
\[\binom{n}{n/2+t}\binom{n}{n/2}^{-1} = \bigg(1+O\bigg(\frac{1}{n} +\frac{t^2}{n^2}+\frac{ t^6}{n^5}\bigg)\bigg)\exp\bigg(-\frac{2t^2}{n}-\frac{4t^4}{3n^3}\bigg).\]
\end{lemma}
\begin{proof}
Note that 
\begin{align*}
\binom{n}{n/2+t}\binom{n}{n/2}^{-1} &= \frac{(n/2)!^{2}}{(n/2+t)!(n/2-t)!}\\
&=\bigg(1+O\bigg(\frac{1}{n}\bigg)\bigg)\sqrt{\frac{n^2}{n^2-4t^2}}\bigg(1+\frac{2t}{n}\bigg)^{-(n/2+t)}\bigg(1-\frac{2t}{n}\bigg)^{-(n/2-t)}\\
&=\bigg(1+O\bigg(\frac{1}{n} +\frac{ t^2}{n^2}\bigg)\bigg)\exp\bigg(-\frac{n}{2}\bigg(\frac{4t^2}{n^2}+\frac{8t^4}{3n^4} + O\bigg(\frac{t^6}{n^6}\bigg)\bigg)\\
&=\bigg(1+O\bigg(\frac{1}{n} +\frac{t^2}{n^2}+\frac{ t^6}{n^5}\bigg)\bigg)\exp\bigg(-\frac{2t^2}{n}-\frac{4t^4}{3n^3}\bigg)
\end{align*}
using Stirling's formula. 
\end{proof}

Now we prove \cref{lem:integrate-p-graph,lem:integrate-p-bip-graph}.

\begin{proof}[Proof of \cref{lem:integrate-p-graph}]
Let $N = n(n-1)/2$ and $2m = \sum_{i\in[n]}d_i$. Applying \cref{thm:enum-graph} and noting that the associated statistics satisfy $\gamma_2^2,\mu = \Theta_C(1)$ under the given conditions \cref{A1,A2}, we have
\begin{equation}\label{eq:D-model-graph}
\mb{P}_{\mc{D}_{1/2}^n}[\mbf{D}=\vec{d}]\asymp_{C,\eps}2^{-N}\binom{2N}{2m}^{-1}\binom{N}{m}\prod_{i=1}^{n}\binom{n-1}{d_i}\asymp_{C,\eps}2^{-2N}\exp\bigg(\frac{(2m-N)^2}{2N}\bigg)\prod_{i=1}^{n}\binom{n-1}{d_i}.
\end{equation}
The second part comes from
\begin{align*}
\binom{2N}{2m}^{-1}\binom{N}{m}(1/2)^{N} &\asymp \frac{\binom{2N}{N}}{\binom{2N}{2m}}\cdot \frac{\binom{N}{m}}{\binom{N}{N/2}}\cdot (1/2)^{2N}\asymp_C(1/2)^{2N}\exp\bigg(\frac{(2m-N)^2}{2N}\bigg)
\end{align*}
by \cref{lem:binom-middle}, since $m = N/2+O_C(N^{3/4})$ follows from \cref{A2} and Cauchy--Schwarz.

On the other hand, the definition of $\mc{I}_{1/2}^n$ and \cite[Lemma~2.2]{MW97} demonstrate
\begin{align}
\mb{P}_{\mc{I}_{1/2}^n}[\mbf{D}=\vec{d}]&\asymp\prod_{i=1}^{n}\binom{n-1}{d_i}\int_0^1\frac{2}{1+(2p'-1)^{2N}}(p')^{2m}(1-p')^{2N-2m}\cdot\sqrt{\frac{4N}{\pi}}\exp(-4N(p'-1/2)^2)dp'\notag\\
&\asymp\prod_{i=1}^{n}\binom{n-1}{d_i}\int_0^1 2\sqrt{\frac{4N}{\pi}}\exp(-4N(p'-1/2)^2)(p')^{2m}(1-p')^{2N-2m}dp'.\label{eq:I-model-graph}
\end{align}
It suffices to prove that the expressions \cref{eq:D-model-graph,eq:I-model-graph} are within multiplicative constants.

To compute the desired integral we set $2m = N + y\sqrt{N}$ (recall that $|y|\lesssim_C N^{1/4}$) and let $p' = 1/2 + z/(2\sqrt{N})$, $z_0=-\sqrt{N}$, $z_1=\sqrt{N}$. Then the integral under question can be reparametrized as 
\[T(m) = \frac{2}{\sqrt{\pi}}\int_{z_0}^{z_1}t(z)dz\]
where 
\begin{align*}
t(z) = \exp\big(-z^2 + (N + y \sqrt{N})\log(1/2+z/(2\sqrt{N})) + (N-y\sqrt{N})\log(1/2-z/(2\sqrt{N})\big).
\end{align*}

We break into a number of cases in order to show this integral is, up to constants, of size $2^{-2N}\exp((2m-N)^2/(2N))$. We first bound the contribution to the integral from $z$ such that $|z|\ge 2|y|+N^{1/4}$. Note that
\begin{align*}
(N + y \sqrt{N})&\log\Big(\frac{1}{2}+\frac{z}{2\sqrt{N}}\Big) + (N-y\sqrt{N})\log\Big(\frac{1}{2}-\frac{z}{2\sqrt{N}}\Big)\\
&\le(N + y \sqrt{N})\log\Big(\frac{1}{2}+\frac{y}{2\sqrt{N}}\Big) + (N-y\sqrt{N})\log\Big(\frac{1}{2}-\frac{y}{2\sqrt{N}}\Big)\le -2N\log 2 + 2y^2
\end{align*}
hence we have
\[t(z)\le 2^{-2N}\exp(-z^2 + 2y^2)\]
and thus the contribution from $|z|\ge 2|y|+N^{1/4}$ is trivially seen to be negligible compared to the target value. Now when $|z|\le 2|y|+N^{1/4}$ we find from Taylor series that
\[t(z) = 2^{-2N}\exp\bigg(\frac{y^2-4(z-y/2)^2}{2} + O\bigg(\frac{z^4}{N} + \frac{z^2y}{N}\bigg)\bigg).\]
Noting that $z = O_C(N^{1/4})$ in this range, up to a constant multiplicative factor depending on $C$ we see this trivially integrates to give the desired result noting that $\exp((2m-N)^2/(2N)) = \exp(y^2/2)$.
\end{proof}
\begin{proof}[Proof of \cref{lem:integrate-p-bip-graph}]
Let $N = n_1n_2$ and $m = \sum_{i\in[n_1]}s_i$. Applying \cref{thm:enum-bigraph} and \cref{lem:binom-middle}, similar to the proof of \cref{lem:integrate-p-graph} we have $m = N/2+O_C(N^{3/4})$ and find
\begin{align}
\mb{P}_{\mc{D}_{1/2}^{n_1,n_2}}[\mbf{S}=\vec{s}\wedge\mbf{T}=\vec{t}]&\asymp_{C,\theta,\eps}2^{-N}\binom{N}{m}^{-1}\prod_{i=1}^{n_1}\binom{n_2}{s_i}\prod_{i=1}^{n_2}\binom{n_1}{t_i}\notag\\
&\asymp_{C,\theta,\eps}2^{-2N}N^{1/2}\exp\bigg(\frac{(2m-N)^2}{2N}\bigg)\prod_{i=1}^{n_1}\binom{n_2}{s_i}\prod_{i=1}^{n_2}\binom{n_1}{t_i}.\label{eq:D-model-bip-graph}
\end{align}
On the other hand, the definition of $\mc{I}_{1/2}^{n_1,n_2}$ demonstrates
\begin{align}
\frac{\mb{P}_{\mc{I}_{1/2}^{n_1,n_2}}[\mbf{S}=\vec{s}\wedge\mbf{T}=\vec{t}]}{\prod_{i=1}^{n_1}\binom{n_2}{s_i}\prod_{i=1}^{n_2}\binom{n_1}{t_i}}&\asymp\int_0^1\bigg(\frac{N}{\pi(1/2)^2}\bigg)^{1/2}\exp\bigg(-\frac{N}{(1/2)^2}(p'-1/2)^2\bigg)\frac{(p')^{2m}(1-p')^{2(N-m)}}{\mb{P}[\mr{Bin}(N,p') = \mr{Bin}(N,p')]}dp'\notag\\
&\asymp N\int_0^1\exp\bigg(-\frac{N}{(1/2)^2}(p'-1/2)^2\bigg)(p')^{2m}(1-p')^{2(N-m)}dp',\label{eq:I-model-bip-graph}
\end{align}
where the second line uses that $\mb{P}[\mr{Bin}(N,p')=\mr{Bin}(N,p')] = \Theta(1/\sqrt{N})$ for $p'\in[1/3,2/3]$ and that the integrand is so small when $p\notin[1/3,2/3]$ that it contributes only lower order terms.

Therefore it suffices to prove that $4^{-N}\exp((2m-N)^2/(2N))$ and 
\[T(m) = \int_{0}^{1}\bigg(\frac{N}{\pi(1/2)^2}\bigg)^{1/2}\exp\bigg(-\frac{N}{(1/2)^2}(p'-1/2)^2\bigg)(p')^{2m}(1-p')^{2(N-m)}dp'\]
are within constants. We now proceed in an identical manner to the proof of \cref{lem:integrate-p-graph}. Let $2m = N + y \sqrt{N}$, $p' = 1/2 + z/(2\sqrt{N})$ and note that 
\[T(m) = \frac{1}{\sqrt{\pi}}\int_{-\sqrt{N}}^{\sqrt{N}}t(z)dz\]
where
\[t(z) = \exp\big(-z^2 + (N+y\sqrt{N})\log(1/2+z/(2\sqrt{N})) + (N-y\sqrt{N})\log(1/2-z/(2\sqrt{N}))\big).\]
This is the same situation as the end of the proof of \cref{lem:integrate-p-graph}, and we finish in the same way.
\end{proof}

\section{Eliminating Atypical Degree Sequences}\label{sec:atypical-switching}
In order to apply \cref{lem:integrate-p-graph,lem:integrate-p-bip-graph} we need to eliminate degree sequences which are too atypical. This involves eliminating graphs where the largest degrees are $n^{1/2+\eta}$ bigger than the smallest and graphs for which there exists a subset $S$ such that $\sum_{v\in V(G)}|\deg(v,S) - |S|/2|^2 \ge Cn^2$ for $C$ sufficiently large. 

To eliminate high and low degrees we use switchings in the style of \cite[Section~4]{MWW02}; however the work of \cite[Section~4]{MWW02} uses that the class of graphs considered is downwards-closed. In our model, however, there are constraints which are ``two-sided'' for various vertices and therefore our analysis is more delicate. The crucial observation is that conditional on the graph not containing a very dense subgraph, only an $n^{-1/2+\eta}$ fraction of vertices can be critical with respect to the constraints (see \cref{lem:critical-vtx}). This itself is proved via switchings, but under this assumption a straightforward argument proves the necessary degree bounds.

The second condition is substantially easier to guarantee and only requires noting that for sufficiently large $C$ the above occurs with (very) exponentially small probability. The necessary statement essentially appears in \cite[Lemma~4.11]{FKNSS22} but with an extra disjointness condition that we easily remove.

\subsection{Eliminating highly irregular subgraphs}
We first show that it is typical for the degree sequence of $\mb{G}(n,1/2)$ and all its dense subgraphs in an appropriate sense to have variance of order $n$. This will hold with sufficiently high probability that it will be robust to conditioning even on very unlikely events.

We will require a version of the classical Bernstein inequality.
\begin{theorem}[{\cite[Theorem~2.8.1]{Ver18}}]\label{thm:bernstein}
For a random variable $X$ define the $\psi_1$-norm
\[\snorm{X}_{\psi_1}=\inf\{t>0\colon\mb{E}[\exp(|X|/t)]\le 2\}.\]
There is an absolute constant $c = c_{\ref{thm:bernstein}}> 0$ such that the following holds. If $X_1,\ldots,X_N$ are independent random variables then 
\[\mb{P}\bigg[\bigg|\sum_{i=1}^NX_i\bigg|\ge t\bigg]\le 2\exp\bigg(-c\min\bigg(\frac{t^2}{\sum_{i=1}^N\snorm{X_i}_{\psi_1}^2},\frac{t}{\max_i\snorm{X_i}_{\psi_1}}\bigg)\bigg)\]
for all $t\ge 0$.
\end{theorem}
\begin{lemma}\label{lem:large-var-deg}
Given $c,K>0$, there is $C=C_{\ref{lem:large-var-deg}}(K)>0$ so the following holds. Let $G\sim\mb{G}(2n,1/2)$ and let $\mc{E}_{\mr{big}}$ be the event that for some $S\subseteq[2n]$, we have
\[\sum_{v\in[2n]}(\deg(v,S)-|S|/2)^2\ge Cn^2.\]
Then $\mb{P}[\mc{E}_{\mr{big}}]\le\exp(-Kn)$ for $n$ large.
\end{lemma}
\begin{proof}
Suppose there is some $S$ with
\[\sum_{v\in[2n]}(\deg(v,S)-|S|/2)^2\ge Cn^2,\]
with $C$ large to be chosen later. Let $T$ be a uniformly random subset of $S$ and $U=S\setminus T$. Let us denote by $\mb{E}_T$ the result of averaging only over the randomness of $T$. Then given some $v\in[2n]$,
\begin{align*}
\mb{E}_T[\mbm{1}_{v\notin T}(\deg(v,T)-|T|/2)^2&+\mbm{1}_{v\notin U}(\deg(v,U)-|U|/2)^2]\\
&\ge\frac{1}{2}\mb{E}_T[(\deg(v,T)-|T|/2)^2+(\deg(v,U)-|U|/2)^2]\\
&\ge\frac{1}{4}\mb{E}_T[(\deg(v,S)-|S|/2)^2]
\end{align*}
where the first line comes from the fact that $\deg(v,U),\deg(v,T)$ are determined purely by the information of $T\cap(S\setminus\{v\})$, and that $v\notin T$ and $v\notin U$ then conditionally occur with probability at least $1/2$, and the second line follows from Cauchy--Schwarz and $S=T\cup U$ with $T,U$ disjoint. Summing over $v$, we find that
\[\mb{E}_T\bigg[\sum_{v\in[2n]}\bigg(\mbm{1}_{v\notin T}(\deg(v,T)-|T|/2)^2+\mbm{1}_{v\notin U}(\deg(v,U)-|U|/2)^2\bigg)\bigg]\ge Cn^2/4\]
hence there is some $T\subseteq S$ such that $\sum_{v\in[2n]\setminus T}(\deg(v,T)-|T|/2)^2\ge Cn^2/8$. That is, if $\mc{E}_{\mr{big}}$ holds then one of $2^{2n}$ possible $T$ satisfy this inequality. Therefore, by the union bound and adjusting the value of $K$ appropriately, it suffices to understand for a fixed $T$ the chances of $\sum_{v\in[2n]\setminus T}(2\deg(v,T)-|T|)^2\ge Cn^2/2$ occurring.

Now note that over the randomness of $\mb{G}(n,1/2)$, each $(2\deg(v,T)-|T|)^2$ is distributed as $X^2$ where $X$ is the sum of $|T|$ independent Rademacher 
variables (uniform on $\{\pm1\}$). Furthermore, as we vary $v\in[2n]\setminus T$, these variables are independent. We easily check that $\snorm{X^2}_{\psi_1}=O(|T|)=O(n)$ and hence by \cref{thm:bernstein} we have
\[\mb{P}\bigg[\sum_{v\in[2n]\setminus T}(2\deg(v,T)-|T|)^2\ge t\bigg]\le2\exp\big(-\Omega(\min\{t^2/n^3,t/n\})\big).\]
Plugging in $t=Cn^2/2$ and choosing $C$ sufficiently large compared to $K$, the result follows.
\end{proof}

Next we show friendly partitions are not super-exponentially unlikely (in particular, although conditioning on having two fixed partitions be friendly will tilt the probability space, the event $\mc{E}_{\mr{var}}$ from \cref{lem:large-var-deg} with appropriately chosen parameters will still hold).
\begin{lemma}\label{lem:friendly-exp-lower}
Given $c>0$ there is $C=C_{\ref{lem:friendly-exp-lower}}(c)>0$ such that the following holds for all $|\gamma|\le 1$. Fix a pair of equipartitions of $[2n] = A_1 \cup A_2 = B_1 \cup B_2$ such that $cn\le |A_1 \cap B_1|\le (1-c)n$. Sample $G\sim\mb{G}(2n,1/2)$ and let $\mc{E}_2$ be the event that $\deg(v,A_i)\ge\deg(v,A_{i+1})+\gamma\sqrt{n}$ for all $v\in A_i$ and $\deg(v,B_i)\ge\deg(v,B_{i+1})+\gamma\sqrt{n}$ for all $v\in B_i$. Then for $n$ large,
\[\mb{P}[\mc{E}_2]\ge\exp(-Cn).\]
\end{lemma}
\begin{proof}
Let us consider the event $\mc{E}$ that for every $i,j\in\{1,2\}$ and $v\in A_i\cap B_j$, we have $\deg(v,A_i\cap B_j)-|A_i\cap B_j|/2\in[4\sqrt{n},8\sqrt{n}]$ and further for every $(i',j')\neq(i,j)$ we have $\deg(v,A_{i'}\cap B_{j'})-|A_{i'}\cap B_{j'}|/2\in[-\sqrt{n},\sqrt{n}]$. It is easy to check that $\mc{E}_2$ is satisfied in such a circumstance. We now compute a lower bound for $\mb{P}[\mc{E}]$ using \cref{thm:enum-graph,thm:enum-bigraph}.

Notice that there are $4$ degree sequences for each $G[A_i\cap B_j]$, call them $(d_v^{(i,j)})_{v\in A_i\cap B_j}$ and $12$ degree sequences for each $G[A_i\cap B_j,A_{i'}\cap B_{j'}]$, call them $(d_v^{(i,j,i',j')})_{v\in A_i\cap B_j}$ (switching the roles of $(i,j),(i',j')$ gives the degrees on both sides of a given pair of parts). The event $\mc{E}$ is merely a system of constraints on the elements of these degree sequences. Let us consider any sequence of values satisfying the constraints given by $\mc{E}$ and such that the two following additional properties hold:
\begin{itemize}
    \item $2|\sum_{v\in A_i\cap B_j}d_v^{(i,j)}$ for all $i,j\in\{1,2\}$;
    \item $\sum_{v\in A_i\cap B_j}d_v^{(i,j,i',j')}=\sum_{v\in A_{i'}\cap B_{j'}}d_v^{(i',j',i,j)}$ for all $i,j,i',j'\in\{1,2\}$ with $(i,j)\neq(i',j')$.
\end{itemize}
We easily see that there are at least $(\Omega(\sqrt{n}))^{8n}$ many choices of degree sequences with this property: each individual value $\deg(v,A_i\cap B_j)$ for $v\in[2n]$ and $i,j\in\{1,2\}$ has at least $\sqrt{n}$ choices, and it is easy to see that at least a $\exp(-n)$ fraction of all these choices satisfy the above constraints. To solve the parity constraints, it is easy to see around half of the sequences $d_v^{(i,j)}$ work for each $i,j$. To solve the equality constraints one can check that most pairs of sequences $(d_v^{(i,j,i',j')})_{v\in A_i\cap B_j}$ and $(d_v^{(i',j',i,j)})_{v\in A_{i'}\cap B_{j'}}$ have sums differing by say $O(n\log n)$ due to variance considerations, and then modifying $O(\sqrt{n}(\log n)^2)$ values in both sequences by approximately $\sqrt{n}/\log n$ each will allow for the appropriate balancing. The fraction of possible sequences attained from this procedure is at least $(1/\sqrt{n})^{O(\sqrt{n}(\log n)^2)}\ge\exp(-n)$.

Furthermore, \cref{thm:enum-graph,thm:enum-bigraph} apply to these degree sequences. If we let $n_{ij}=|A_i\cap B_j|$ and $2m_{ij}=\sum_{v\in A_i\cap B_j}d_v^{(i,j)}$, then \cref{thm:enum-graph} applied to $(d_v^{(i,j)})_{v\in A_i\cap B_j}$ leads to $\mu=1/2+O(1/\sqrt{n})$, $\gamma_2^2=\Theta(1)$, and therefore on the order of at least
\begin{align*}
\gtrsim\binom{n_{ij}(n_{ij}-1)/2}{m_{ij}}\binom{n_{ij}(n_{ij}-1)}{2m_{ij}}^{-1}\prod_{v\in A_i\cap B_j}\binom{n_{ij}-1}{d_v^{(i,j)}}.
\end{align*}
One can check that each binomial coefficient in the product is $\Omega(2^{n_{ij}-1}/\sqrt{n})$, while the two initial binomial coefficients when divided contribute at least $\exp(-Kn)2^{-n_{ij}(n_{ij}-1)/2}$ for some constant $K$ depending only on $c$. Multiplying we have a contribution of at least say
\[\exp(-2Kn)\cdot2^{\binom{n_{ij}}{2}}(1/\sqrt{n})^{n_{ij}}\] many possibilities for the graph $G[A_i\cap B_j]$ given this degree sequence.

Similarly, the number of choices for $G[A_i\cap B_j,A_{i'}\cap B_{j'}]$ can be seen to be at least say
\[\exp(-2Kn)\cdot 2^{n_{ij}n_{i'j'}}(1/\sqrt{n})^{n_{ij}+n_{i'j'}}.\]

Multiplying over $4$ graphs $G[A_i\cap B_j]$ and $6$ graphs $G[A_i\cap B_j,A_{i'}\cap B_{j'}]$, this becomes at least $\exp(-20Kn)2^{\binom{2n}{2}}(1/\sqrt{n})^{8n}$ total. Multiplying by the number of choices of degree sequences from earlier, we have at least say $\exp(-40Kn)2^{\binom{2n}{2}}$ total choices of realizations $G\sim\mb{G}(2n,1/2)$ satisfying the desired property $\mc{E}$ (and hence $\mc{E}_2$), if $K$ was chosen sufficiently large. The desired result immediately follows taking $C=40K$, which depends only on $c$.
\end{proof}

As an immediate consequence we have the following proposition.
\begin{proposition}\label{prop:friendly-var-deg}
Given the setup and notation of \cref{lem:friendly-exp-lower}, there is $K=K_{\ref{prop:friendly-var-deg}}(c)>0$ such that the following holds. Let $G\sim\mb{G}(2n,1/2)$ and let $\mc{E}_{\mr{var}}$ be the event that
\[\sum_{v\in[2n]}(\deg(v,S)-|S|/2)^2\le Kn^2\]
for all $S\subseteq[2n]$. Then $\mb{P}[\mc{E}_{\mr{var}}|\mc{E}_2]\ge1-\exp(-n)$ for $n$ large.
\end{proposition}
\begin{proof}
Let $C=C_{\ref{lem:friendly-exp-lower}}(c)$ and apply \cref{lem:large-var-deg} with $K$ replaced by $1+2C_{\ref{lem:friendly-exp-lower}}(c)$. We have $\mb{P}[\mc{E}_{\mr{big}}|\mc{E}_2]\le\mb{P}[\mc{E}_{\mr{big}}]/\mb{P}[\mc{E}_2]\le\exp(-(2C+1)n)/\exp(-Cn)\le\exp(-n)$, and the result follows.
\end{proof}

\subsection{Eliminating high degree vertices}
We require a tail bound of the binomial random model based on the first value in the tail, and a generalization of this estimate to arbitrary upwards-closed families of sets.
\begin{lemma}\label{lem:binomial-tail}
For $0\le\ell\le n$ we have that
\[\frac{\binom{n}{\ell}}{\sum_{k=\ell}^n\binom{n}{k}}\lesssim \frac{|\ell-n/2|}{n}+\frac{1}{\sqrt{n}}.\]
\end{lemma}
\begin{proof}
For $\ell\le n/2+\sqrt{n}$, we have that $\binom{n}{\ell}\lesssim 2^n/\sqrt{n}$ and $\sum_{k=\ell}^n\binom{n}{k}\gtrsim 2^n$ and hence the result follows. For $|\ell-n/2|\ge n/5$ the result is trivial. Finally when $n/2+\sqrt{n}\le\ell\le n/2+n/5$ we have that 
\begin{align*}
\frac{\sum_{k=\ell}^n\binom{n}{k}}{\binom{n}{\ell}} &\ge\sum_{t=0}^{\sqrt{n}}\binom{n}{\ell+t}\binom{n}{\ell}^{-1}\\
&\ge\sum_{t=0}^{\sqrt{n}}((\ell+t)/(n-\ell-t+1))^{-t}\\
&\gtrsim \sum_{t=0}^{\sqrt{n}}(1+16|\ell-n/2|/n)^{-t}\gtrsim n/|\ell-n/2|.\qedhere
\end{align*}
\end{proof}
\begin{lemma}\label{lem:binomial-family}
Let $\mc{F}$ be an upwards-closed family of subsets of $[n]$, i.e., if $F_1\in\mc{F}$ and $F_2\supseteq F_1$ then $F_2\in\mc{F}$. Let $\mc{F}_\ell$ be the elements of $\mc{F}$ of size $\ell$. Then
\[\frac{|\mc{F}_\ell|}{|\mc{F}_\ell|+|\mc{F}_{\ell+1}|+\cdots+|\mc{F}_n|}\le\frac{\binom{n}{\ell}}{\sum_{i=\ell}^n\binom{n}{i}}\lesssim\frac{|\ell-n/2|}{n}+\frac{1}{\sqrt{n}}.\]
\end{lemma}
\begin{proof}
Every size $\ell$ set within $\mc{F}$ has exactly $\binom{n-\ell}{s-\ell}$ size $s$ sets in $\mc{F}$ containing it by the upwards-closed condition. Furthermore, every size $s$ set within $\mc{F}$ has at most $\binom{s}{\ell}$ size $\ell$ sets in $\mc{F}$ contained within it. Therefore, $\binom{n-\ell}{s-\ell}|\mc{F}_\ell|\le\binom{s}{\ell}|\mc{F}_s|$ for all $\ell\le s\le n$. The result easily follows noting that $\binom{n-\ell}{s-\ell}/\binom{s}{\ell}=\binom{n}{s}/\binom{n}{\ell}$ and then using \cref{lem:binomial-tail}.
\end{proof}

Now we prove various structural properties hold whp for friendly partitions in $\mb{G}(2n,1/2)$.
\begin{proposition}\label{prop:large-degree-1}
Fix $\eta>0$ and consider $|\gamma|\le 1$. Fix an equipartitions $[2n] = A_1 \cup A_2$. Sample $G\sim\mb{G}(2n,1/2)$ and let $\mc{E}_1$ be the event that $\deg(v,A_i)\ge\deg(v,A_{i+1})+\gamma\sqrt{n}$ for all $v\in A_i$. Let $\mc{E}_{\mr{irreg},1}$ denote the event that there exists a vertex $v\in[2n]$ and $i\in\{1,2\}$ such that $|\deg(v,A_i)-n|\ge n^{1/2+\eta}$. Then for $n$ large,
\[\mb{P}[\mc{E}_{\mr{irreg},1}|\mc{E}_1]\le n^{-\omega(1)}.\]
\end{proposition}

\begin{proposition}\label{prop:large-degree-2}
Fix $c,\eta>0$ and consider $|\gamma|\le 1$. Fix a pair of equipartitions of $[2n] = A_1 \cup A_2 = B_1 \cup B_2$ such that $cn\le |A_1 \cap B_1|\le (1-c)n$. Sample $G\sim\mb{G}(2n,1/2)$ and let $\mc{E}_2$ be the event that $\deg(v,A_i)\ge\deg(v,A_{i+1})+\gamma\sqrt{n}$ for all $v\in A_i$ and $\deg(v,B_i)\ge\deg(v,B_{i+1})+\gamma\sqrt{n}$ for all $v\in B_i$. Let $\mc{E}_{\mr{irreg},2}$ denote the event that there exists a vertex $v\in[2n]$ and $i,j\in\{1,2\}$ such that $|\deg(v,A_i\cap B_j)-|A_i\cap B_j|/2|\ge n^{1/2+\eta}$. Then for $n$ large,
\[\mb{P}[\mc{E}_{\mr{irreg},2}|\mc{E}_2]\le n^{-\omega(1)}.\]
\end{proposition}

To prove these propositions, we need the following estimate on ``critical'' vertices in friendly partitions.

\begin{lemma}\label{lem:critical-vtx}
Fix $c>0$ and a positive integer $k$, and consider $|\gamma|\le 1$ with $\gamma\sqrt{n}\in\mb{Z}$. Fix a pair of equipartitions of $[2n] = A_1 \cup A_2 = B_1 \cup B_2$ such that $cn\le |A_1 \cap B_1|\le (1-c)n$. Sample $G\sim\mb{G}(2n,1/2)$ and let $\mc{E}_2$ be the event that $\deg(v,A_i)\ge\deg(v,A_{i+1})+\gamma\sqrt{n}$ for all $v\in A_i$ and $\deg(v,B_i)\ge\deg(v,B_{i+1})+\gamma\sqrt{n}$ for all $v\in B_i$. Let $\mc{E}_1$ be the event that this holds just for $A$. Define a vertex to be \emph{critical for $A$} if $v\in A_i$ and $\deg(v,A_i) = \deg(v,A_{i+1})+\gamma\sqrt{n}$ and let the number of critical for $A$ vertices be $X$. Then
\[\mb{E}[X^k|\mc{E}_b]\lesssim_{c,k}n^{k/2}\quad\emph{for}~b\in\{1,2\}.\]
\end{lemma}
\begin{proof}
Let $Y$ be the number of critical for $A$ vertices inside $A_1\cap B_1$. We will show $\mb{E}[Y^k|\mc{E}_2]\lesssim_{c,k}n^{k/2}$. By symmetry we will find the same holds for each $A_i\cap B_j$, which will imply the result for $b=2$. The case for $b=1$ is similar so we forgo the details.

Let us further reveal the values $\deg(v,A_i\cap B_j)$ for all $i,j\in\{1,2\}$ and $v\in V(G)\setminus (A_i\cap B_j)$. Thus the remaining randomness is over $\mb{G}(A_i\cap B_j,1/2)$ for each $i,j\in\{1,2\}$, and we condition these graphs on certain degree inequalities. Notice that the four parts are independent since we have revealed the information of the degrees between the parts. Whether a vertex is critical thus also only depends on its degree within a part. Let us focus on the part $A_1\cap B_1$, which will allow us to control $Y$. The graph $G[A_i\cap B_j]$ is a uniform random graph satisfying various lower bounds on the degrees of its vertices. For each vertex, given the revealed information of the degrees to the outside, the conditions $\deg(v,A_1)\ge\deg(v,A_2)+\gamma\sqrt{n}$ and $\deg(v,B_1)\ge\deg(v,B_2)+\gamma\sqrt{n}$ lead to precisely these conditions on $\deg(v,A_i\cap B_j)$. If the latter inequality provides a sharper inequality than the former, then $v$ will not be critical for $A$. If the former inequality is sharper, then it can be critical for $A$ precisely when it equals the minimum allowed value.

Let $\mc{G}$ be the information of all edges in $G$ outside of $A_1\cap B_1$. We have by linearity of expectation that
\begin{equation}\label{eq:critical-vtx-moment}
\mb{E}[Y(Y-1)\cdots(Y-k+1)|\mc{E}_2,\mc{G}]\le\sum_{\substack{v_1,\ldots,v_k\in A_1\cap B_1\\\text{distinct}}}\mb{P}[v_1,\ldots,v_k\text{ are critical for }A|\mc{E}_2,\mc{G}].
\end{equation}
Now given a choice of $v_1,\ldots,v_k$, let us further reveal everything in the graph $G[A_1\cap B_1]$ except the edges between $\{v_1,\ldots,v_k\}$ and $A_1\cap B_1$ and furthermore reveal the edges internal to $\{v_1,\ldots,v_k\}$. The remaining random portion is a uniformly random bipartite graph between $\{v_1,\ldots,v_k\}$ and $(A_1\cap B_1)\setminus\{v_1,\ldots,v_k\}$ subject to certain lower bounds for degrees on both sides (coming from our original inequalities). We now consider the maximum probability that $v_1$ is critical given the neighborhoods of $v_2,\ldots,v_k$. Let us think of the possible neighborhoods of $v_1$ (disregarding the revealed part among $\{v_2,\ldots,v_k\}$) as subsets of $(A_1\cap B_1)\setminus\{v_1,\ldots,v_k\}$. Since the conditioned information only gives lower bounds on certain degrees (which may or may not force certain vertices to connect to $v_1$ to meet their ``quota''), we see that the family $\mc{F}$ of possibilities is upwards-closed. By \cref{lem:binomial-family} we find that the resulting probability that $v$ is critical is at most
\[O\bigg(\frac{|\ell-(|A_1\cap B_1|-k)/2|}{cn}+\frac{1}{\sqrt{cn}}\bigg)\]
where $\ell$ is the lower bound on the size of $N(v)\cap((A_1\cap B_1)\setminus\{v_1,\ldots,v_k\})$ coming from the condition $\deg(v,A_1)\ge\deg(v,A_2)+\lceil\gamma\sqrt{n}\rceil$. That is,
\[\ell=\deg(v_1,A_2)-\deg(v_1,A_1\cap B_2)-\deg(v_1,\{v_2,\ldots,v_k\})+\lceil\gamma\sqrt{n}\rceil.\]
Since $k$ is constant, the bound becomes
\[O\bigg(\frac{|\deg(v_1,A_2)-\deg(v_1,A_1\cap B_2)-|A_1\cap B_1|/2|}{n}+\frac{1}{\sqrt{n}}\bigg).\]
Multiply over $v_1,\ldots,v_k$ to obtain a bound on the probability that all of $v_1,\ldots,v_k$ are critical (note that the above is a bound on $v_1$ even conditioned on any outcomes for the other vertices, so this is valid). Plugging into \cref{eq:critical-vtx-moment}, we find
\begin{align*}
\mb{E}[Y(Y-1)\cdots&(Y-k+1)|\mc{E}_2,\mc{G}]\lesssim\bigg(\sum_{v\in A_1\cap B_1}\frac{|\deg(v,A_2)-\deg(v,A_1\cap B_2)-|A_1\cap B_1|/2|}{n}+\frac{1}{\sqrt{n}}\bigg)^k\\
&\lesssim n^{k/2}+n^{-k}\bigg(\sum_{v\in A_1\cap B_1}|\deg(v,A_2)-\deg(v,A_1\cap B_2)-|A_1\cap B_1|/2|\bigg)^k\\
&\lesssim n^{k/2}+n^{-k/2}\bigg(\sum_{v\in A_1\cap B_1}|\deg(v,A_2)-\deg(v,A_1\cap B_2)-|A_1\cap B_1|/2|^2\bigg)^{k/2}.
\end{align*}

Now by \cref{prop:friendly-var-deg}, there is some appropriately large constant $A$ so that conditional on $\mc{E}_2$, with probability at least $1-\exp(-n)$ we have $\sum_{v\in[2n]}(\deg(v,S)-|S|/2)^2\le An^2$ for all $S\subseteq[2n]$ (event $\mc{E}_{\mr{var}}$). In particular, note that
\begin{equation}\label{eq:switching-variance-check}
\sum_{v\in A_1\cap B_1}|\deg(v,A_2)-\deg(v,A_1\cap B_2)-|A_1\cap B_1|/2|^2\le 4An^2
\end{equation}
holds under the event $\mc{E}_{\mr{var}}$. Furthermore, the truth of this particular inequality is a function only of the information of $\mc{G}$. Therefore, we easily see
\[\mb{E}[Y(Y-1)\cdots(Y-k+1)|\mc{E}_2]\lesssim(n^{k/2}+n^{-k/2}(4An^2)^{k/2})+\exp(-n)\cdot n^k\lesssim n^{k/2}\]
by considering whether $\mc{G}$ satisfies \cref{eq:switching-variance-check} or not. The desired estimate follows immediately.
\end{proof}

Now we use switchings to prove \cref{prop:large-degree-1,prop:large-degree-2}.
\begin{proof}[Proof of \cref{prop:large-degree-1,prop:large-degree-2}]
We focus on \cref{prop:large-degree-2}; the case for \cref{prop:large-degree-1} is analogous but simpler, so we truncate the details. Furthermore, it suffices to prove the estimate for the probability for a single choice of $i^\ast,j^\ast\in\{1,2\}$ and $v\in V(G)$. Let us assume $v\in A_1\cap B_1$ without loss of generality. We wish to understand $\mb{P}[|\deg(v,A_{i^\ast}\cap B_{j^\ast})-|A_{i^\ast}\cap B_{j^\ast}|/2|\ge n^{1/2+\eta}|\mc{E}_2]$.

Let us reveal the identities of all edges in $G$ except the potential edges between $v$ and $V(G)\setminus\{v\}$. Call this information $\mc{G}$. Conditional on $\mc{E}_2,\mc{G}$, our distribution on the neighborhood of $v$ is uniform over choices such that $v$ satisfies its two degree inequalities and such that every other vertex satisfies its inequalities as well. But note that for the typical other vertex, whether $v$ connects to it will not influence whether the inequalities are satisfied. In fact, this will only be the case for vertices that are ultimately critical (for $A$ or $B$) in the graph $G$. By \cref{lem:critical-vtx} with $k$ growing sufficiently slowly as a function of $n$, we see that there are at most $n^{1/2+\eta/4}$ critical vertices for $A$ with probability $1-n^{-\omega(1)}$, and similar for critical vertices for $B$ by symmetry. Thus, we may assume $\mc{G}$ is such that the number of vertices with ``forced'' edges and non-edges to $v$ is at most $n^{1/2+\eta/2}$ in total.

Now the neighborhood of $v$ has the following distribution: there are at most $n^{1/2+\eta/2}$ total designated forced edges and non-edges to the four parts $A_i\cap B_j$, and then the remaining edges are uniform subject to the constraints $\deg(v,A_1)-\deg(v,A_2)\ge\gamma\sqrt{n}$ and $\deg(v,B_1)-\deg(v,B_2)\ge\gamma\sqrt{n}$. We wish to understand the chance that $|\deg(v,A_{i^\ast}\cap B_{j^\ast})-|A_{i^\ast}\cap B_{j^\ast}|/2|\ge n^{1/2+\eta}$, call this event $\mc{E}_{\mr{bad}}$, in this model. Let $C$ be the set of vertices which are forced to either connect or be disconnected from $v$, with $C_1$ connected and $C_0$ disconnected. Heuristically, we are close to a binomial random situation with some ``planted'' deviations coming from $C$ on the order of $n^{1/2+\eta/2}$, so the chance that some degree deviates by $n^{1/2+\eta}$ ought to be extremely unlikely.

We make this heuristic precise in order to complete the proof. Having revealed the above information, let us consider the distribution $\mbf{\Delta}$ of sampling the neighborhood of $v$ within $V(G)\setminus(C\cup\{v\})$ uniformly at random. The above analysis shows it is enough to give an upper bound on $\mb{P}[\mc{E}_{\mr{bad}}\cap\mc{E}_2]/\mb{P}[\mc{E}_2]$ in this model (given the revealed information including $C_0,C_1$).

First, we claim $\mb{P}[\mc{E}_2]\ge\exp(-n^{3\eta/2})$. Indeed, choose positive integers $n_{ij}$ for $i,j\in\{1,2\}$ such that $n_{11}+n_{12}-n_{21}-n_{22}\ge\gamma\sqrt{n}+n^{1/2+\eta/2}$ and $n_{11}+n_{21}-n_{12}-n_{22}\ge\gamma\sqrt{n}+n^{1/2+\eta/2}$ and $|n_{ij}-|A_i\cap B_j|/2|\le 4n^{1/2+\eta/2}$ for $i,j\in\{1,2\}$. This is easily seen to be possible since $|A_i\cap B_j|\in[cn,(1-c)n]$ and $|A_i|=|B_j|=n$. Then the probability that $v$ has $n_{ij}$ neighbors to $(A_i\cap B_j)\setminus C$ for each $i,j\in\{1,2\}$ in the distribution $\mbf{\Delta}$ is easily seen to be at least $\exp(-O((n^{1/2+\eta/2})^2/n))\ge\exp(-n^{3\eta/2})$ for $n$ large. Here we are using the easy estimate $\binom{x}{x/2+t}/2^x=\exp(-\Theta(t^2/x))$ for $t\in[-x/2,x/2]$ satisfying $|t|\ge x^{1/2+\eta/4}$.

Similarly, we see that $\mb{P}[\mc{E}_{\mr{bad}}]\le\exp(-\Omega(n^{2\eta}))$ since if $\mc{E}_{\mr{bad}}$ holds then $\deg(v,(A_i\cap B_j)\setminus C)\ge|(A_i\cap B_j)\setminus C|/2+n^{1/2+\eta}/2$ is easily verified, and under the independent distribution $\mbf{\Delta}$ this can be bounded as above. Dividing, we obtain
\[\mb{P}[\mc{E}_{\mr{bad}}\cap\mc{E}_2]/\mb{P}[\mc{E}_2]\le n^{-\omega(1)},\]
and combining with the earlier analysis we are done.
\end{proof}

\section{Convexity via the Pr\'ekopa--Leindler Inequality}\label{sec:convex}
In this section we collect a pair of log-convexity claims which are proved via observing that the marginals of a log-concave function are log-concave. This fact is a well known consequence of the Pr\'ekopa--Leindler inequality \cite{Lei72,Pre71}. We note here that the use of log-concavity of such Gaussian functionals also appears in the work of Gamarnik and Li \cite[Lemma~4.6]{GL18} (and for the similar purpose of simplifying an associated maximization problem arising in the second moment computation); we provide a short proof of the necessary results in this section.

\begin{theorem}\label{thm:prekopa-leindler}
If $f(x,y)$ is a log-concave function over $(x,y)\in\mb{R}^m\times\mb{R}^n$, then $g(x) = \int_{\mb{R}^n}f(x,y)dy$ is a log-concave function on $\mb{R}^m$.
\end{theorem}

For the first moment we will require the following log-concavity.
\begin{lemma}\label{lem:log-concave-1}
The following function is log-concave:
\[\alpha\mapsto\mb{P}_{Z\sim \mc{N}(0,1)}[Z\ge(\gamma + \alpha)\sqrt{2}].\]
\end{lemma}
\begin{proof}
Via shifting and rescaling it suffices to prove that 
\[\alpha\mapsto\mb{P}_{Z\sim\mc{N}(0,1)}[Z\ge\alpha]\]
is log-concave. Note that 
\[\mb{P}_{Z\sim \mc{N}(0,1)}[Z\ge  \alpha] = \int_{0}^{\infty}\frac{1}{\sqrt{2\pi}}\exp\bigg(\frac{-(x+\alpha)^2}{2}\bigg)dx\]
and that 
\[\mbm{1}_{x\ge 0}\exp\bigg(\frac{-(x+\alpha)^2}{2}\bigg)\] is a log-concave function of $x$ and $\alpha$. \cref{thm:prekopa-leindler} finishes the proof, taking the marginal in $x$.
\end{proof}

For the second moment we will repeatedly used the following generalization of \cref{lem:log-concave-1} (which can be seen as the special case $\beta=1/2$).
\begin{lemma}\label{lem:log-concave-2}
Fix $\beta\in[0,1]$. The function 
\[(\alpha_1,\alpha_2)\mapsto\mb{P}_{Z_1,Z_2\sim \mc{N}(0,1)}\bigg[\sqrt{\frac{\beta}{2}}Z_1+\sqrt{\frac{1-\beta}{2}}Z_2\ge(\gamma + \alpha_1)\sqrt{2}\wedge \sqrt{\frac{\beta}{2}}Z_1-\sqrt{\frac{1-\beta}{2}}Z_2\ge(\gamma + \alpha_2)\sqrt{2}\bigg]\]
is log-concave.
\end{lemma}
\begin{proof}
Via shifting and applying a linear transformation it suffices to prove that 
\[(\alpha_1,\alpha_2)\mapsto\mb{P}_{Z_1,Z_2\sim \mc{N}(0,1)}\bigg[\sqrt{\frac{\beta}{2}}(Z_1-\alpha_1)+\sqrt{\frac{1-\beta}{2}}(Z_2-\alpha_2)\ge 0\wedge \sqrt{\frac{\beta}{2}}(Z_1-\alpha_1)-\sqrt{\frac{1-\beta}{2}}(Z_2-\alpha_2)\ge 0\bigg]\]
is log-concave. Note that 
\begin{align*}
&\mb{P}_{Z_1,Z_2\sim \mc{N}(0,1)}\bigg[\sqrt{\frac{\beta}{2}}(Z_1-\alpha_1)+\sqrt{\frac{1-\beta}{2}}(Z_2-\alpha_2)\ge 0\wedge \sqrt{\frac{\beta}{2}}(Z_1-\alpha_1)-\sqrt{\frac{1-\beta}{2}}(Z_2-\alpha_2)\ge 0\bigg]\\
&=\int_{\substack{\sqrt{\frac{\beta}{2}}z_1+\sqrt{\frac{1-\beta}{2}}z_2\ge 0\\\sqrt{\frac{\beta}{2}}z_1-\sqrt{\frac{1-\beta}{2}}z_2\ge 0}}\frac{1}{2\pi}\exp\bigg(\frac{-(z_1+\alpha_1)^2-(z_2+\alpha_2)^2}{2}\bigg)dz_1dz_2.
\end{align*}
Applying \cref{thm:prekopa-leindler} to the log-concave function of $z_1,z_2,\alpha_1,\alpha_2$ given by
\[\mbm{1}_{\sqrt{\frac{\beta}{2}}z_1+\sqrt{\frac{1-\beta}{2}}z_2\ge 0}\mbm{1}_{\sqrt{\frac{\beta}{2}}z_1-\sqrt{\frac{1-\beta}{2}}z_2\ge 0}\exp\bigg(\frac{-(z_1+\alpha_1)^2-(z_2+\alpha_2)^2}{2}\bigg)\]
and taking the marginal in $z_1,z_2$, the result follows.
\end{proof}

\section{Moment Computations}\label{sec:moment}
Now we are ready to attack the moment computation. Recall $|\gamma|\le 1$ is such that $\gamma\sqrt{n}\in\mb{Z}$.

\subsection{The second moment of \texorpdfstring{$X_\gamma$}{Xgamma}}\label{sub:second-moment}
We first perform the second moment in detail, proving \cref{lem:second-moment}. We defer the much simpler first moment to \cref{sub:first-moment}, in which we truncate various repeated or similar details. Let $c\in(0,1/2)$ be a small constant to be chosen later.

\subsubsection{Setup and initial truncation}\label{sub:second-setup}
Consider $n_1'\in\{0,\ldots,n\}$ and $n_2'=n-n_1'$, and let $A_1=[n]$ and $A_2=\{n+1,\ldots,2n\}$ and $B_1=[n_1']\cup\{n+n_1'+1,\ldots,2n\}$ and $B_2=\{n_1'+1,n_1'+2,\ldots,n+n_1'\}$. Let $n_{ij}=|A_i\cap B_j|$, so that $n_{ij}=n_{i+j-1}'$ where we take indices $\imod{2}$. Also, define $\beta=n_{11}/n$. Given this setup let $Y_A(n_1')=1$ if $G\sim\mb{G}(2n,1/2)$ is such that $A_1\cup A_2$ is a $\gamma\sqrt{n}$-friendly partition, and $0$ otherwise; let $Y_B(n_1')$ be defined similarly. By linearity of expectation we have
\begin{align}\label{eq:first-moment}
\mb{E}X_\gamma&=\binom{2n}{n}\mb{P}[Y_A(n_1')=1],\\
\mb{E}X_\gamma^2&=\binom{2n}{n}\sum_{n_1'=0}^n\binom{n}{n_1'}^2\mb{P}[Y_A(n_1')=1\wedge Y_B(n_1')=1].\label{eq:second-moment}
\end{align}
This means that as long as $c$ is sufficiently smaller than $\eta$, we have
\begin{align}\notag
\binom{2n}{n}\sum_{n_1'<cn}\binom{n}{n_1'}^2\mb{P}[Y_A(n_1')=1\wedge Y_B(n_1')=1]&+\binom{2n}{n}\sum_{n_1'>(1-c)n}\binom{n}{n_1'}^2\mb{P}[Y_A(n_1')=1\wedge Y_B(n_1')=1]\\
&\le 2\sum_{n_1'<cn}\binom{n}{n_1'}^2\binom{2n}{n}\mb{P}[Y_A(n_1')=1]\notag\\
&\le(1+\eta)^n\mb{E}X_\gamma\label{eq:second-moment-small-overlap}
\end{align}
for $n$ large. Thus it suffices to consider the terms for which $n_1'\in[cn,(1-c)n]$ (equivalently, $\beta\in[c,1-c]$), and we wish to compute $\mb{P}[Y_A(n_1')=Y_B(n_1')=1]$. Note that whether the desired event holds or not is purely a function of the degree sequences of the $4$ graphs $G[A_i\cap B_j]$ for $i,j\in\{1,2\}$ and the $6$ bipartite graphs $G[A_i\cap B_j,A_{i'}\cap B_{j'}]$ for $i,j,i',j'\in\{1,2\}$ such that either $i<i'$ or $i=i'$ and $j<j'$. Thus, we wish to apply \cref{lem:integrate-p-graph,lem:integrate-p-bip-graph} to transfer to a simple model for such degree sequences in which it is easier to compute the desired probabilities.

\subsubsection{Transference and defining the integrated degree model}\label{sub:second-transfer}
To do this, we use the results of \cref{sec:atypical-switching}. Let $\mc{E}_2$ be the event that $Y_A(n_1')=Y_B(n_1')=1$ (which is in line with usage in \cref{sec:atypical-switching}). Let $K$ be some large constant in terms of $c$ to be determined. First apply \cref{prop:friendly-var-deg}. Let $\mc{E}_{\mr{var}}$ be the event that $\sum_{v\in[2n]}(\deg(v,S)-|S|/2)^2\ge Kn^2$ for all $S\subseteq[2n]$, where $K$ is some appropriate constant coming from the lemma. We have $\mb{P}[\mc{E}_{\mr{var}}|\mc{E}_2]\ge1-\exp(-n)$. By \cref{prop:large-degree-2} with $\eta$ replaced by $\min(\eps_{\ref{lem:integrate-p-graph}},\eps_{\ref{lem:integrate-p-bip-graph}})/2$, we have $\mb{P}[\mc{E}_{\mr{irreg},2}|\mc{E}_2]\le n^{-\omega(1)}$. Thus
\begin{equation}\label{eq:second-moment-atypical}
\mb{P}[\mc{E}_2]=(1\pm n^{-\omega(1)})\mb{P}[\mc{E}_2\wedge\mc{E}_{\mr{irreg},2}^c\wedge\mc{E}_{\mr{var}}].
\end{equation}
Note that $\mc{E}_{\mr{irreg},2}^c,\mc{E}_{\mr{var}}$ ensure that for the $4+6$ graphs defined above, they have degree sequences which satisfy \cref{A1,A2,A3} of \cref{lem:integrate-p-graph} with appropriate parameters (namely, the degree sequences of $G[A_i\cap B_j]$ are $(O(1),\eps_{\ref{lem:integrate-p-graph}})$-regular) or satisfy \cref{B1,B2,B3,B4,B5} of \cref{lem:integrate-p-bip-graph} with appropriate parameters (namely, the degree sequences of the $6$ bipartite portions of $G$ are $(O(1),\Omega(1),\eps_{\ref{lem:integrate-p-bip-graph}})$-regular). Here we are also using that $cn\le n_1',n_2'\le(1-c)n$, e.g., to establish \cref{B1}.

We deduce
\begin{equation}\label{eq:second-moment-transference}
\mb{P}[\mc{E}_2\wedge\mc{E}_{\mr{irreg},2}^c\wedge\mc{E}_{\mr{var}}]\lesssim\rho,
\end{equation}
where $\rho$ is the following probability:
\begin{itemize}
    \item For $i,j\in\{1,2\}$ sample the ``degrees'' $d_v^{(i,j)}$ for $v\in A_i\cap B_j$ (which should be thought of as corresponding to the internal degrees $\deg(v,A_i\cap B_j)$ for $v\in A_i\cap B_j$) from $\mc{I}_{1/2}^{|A_i\cap B_j|}$;
    \item For $i,j,i',j'\in\{1,2\}$ with $i<i'$ or $i=i'$ and $j<j'$, sample the ``degrees'' $d_v^{(i,j,i',j')}$ for $v\in A_i\cap B_j$ and $d_v^{(i',j',i,j)}$ for $v\in A_{i'}\cap B_{j'}$ (which correspond to the bipartite degrees $\deg(v,A_{i'}\cap B_{j'})$ for $v\in A_i\cap B_j$ and vice versa) jointly from $\mc{I}_{1/2}^{|A_i\cap B_j|,|A_{i'}\cap B_{j'}|}$;
    \item Compute the probability that for each $i,j\in\{1,2\}$ and $v\in A_i\cap B_j$ we have
    \begin{align}\label{eq:degree-A-ineq}
    d_v^{(i,j)}-d_v^{(i,j,i+1,j+1)}+d_v^{(i,j,i,j+1)}-d_v^{(i,j,i+1,j)}&\ge\gamma\sqrt{n},\\
    d_v^{(i,j)}-d_v^{(i,j,i+1,j+1)}-d_v^{(i,j,i,j+1)}+d_v^{(i,j,i+1,j)}&\ge\gamma\sqrt{n},\label{eq:degree-B-ineq}
    \end{align}
    where indices are taken $\imod{2}$.
\end{itemize}
This can be interpreted as the probability that if the degrees of our $4+6$ graphs are sampled as in \cref{lem:integrate-p-graph,lem:integrate-p-bip-graph}, we have that $A_1\cup A_2$ and $B_1\cup B_2$ are both $\gamma\sqrt{n}$-friendly.

Understanding the ``degree'' distribution is relatively simple:
\begin{itemize}
    \item Sample $4+6$ parameters $p_{ij}\sim\mc{N}(1/2,1/(4n_{ij}^2-4n_{ij}))$ for $i,j\in\{1,2\}$ and $p_{ij,i'j'}\sim\mc{N}(1/2,1/(8n_{ij}n_{i'j'}))$ for $i,j,i',j'\in\{1,2\}$ with $i<i'$ or $i=i'$ and $j<j'$ (all conditional on being in $(0,1)$), and let $p_{i'j',ij}=p_{ij,i'j'}$;
    \item Sample $d_v^{(i,j)}\sim\mr{Bin}(n_{ij}-1,p_{ij})$ and $d_v^{(i,j,i',j')}\sim\mr{Bin}(n_{i'j'},p_{ij,i'j'})$ completely independently;
    \item Condition on the $4$ divisibility conditions and $6$ equality conditions
    \begin{align}\label{eq:divisibility-1}
    2&\mid\sum_{v\in A_i\cap B_j}d_v^{(i,j)},\\
    \sum_{v\in A_i\cap B_j}d_v^{(i,j,i',j')}&=\sum_{v\in A_{i'}\cap B_{j'}}d_v^{(i',j',i,j)}.\label{eq:divisibility-2}
    \end{align}
\end{itemize}
The first step can be handled by taking a Gaussian integral over the $4+6$ parameters, and the second step is independent. Thus the major difficulty lies in the conditioning introduced by the third step. We will handle this by use of Bayes' theorem, local central limit theorems, and anticoncentration inequalities (see e.g.~\cite[Sections~2.3.2,~2.3.3]{SS21maj}), but one should think of this more as a technical detail: the probability $\rho$ ought to be mostly unchanged if this conditioning were not present, although one cannot actually directly drop the equal sum conditions.

\subsubsection{Bayes' theorem in the transferred model}\label{sub:second-bayes-theorem}
For simplicity, let us assume that any event involving $\mr{Bin}(n,p)$ for $p\notin[0,1]$ does not hold. Let $p_{ij}=1/2+\alpha_{ij}/\sqrt{4n_{ij}^2-4n_{ij}}$ and $p_{ij,i'j'}=1/2+\alpha_{ij,i'j'}/\sqrt{8n_{ij}n_{i'j'}}$. Additionally, we will use the subscript $\vec{d}$ to denote the randomness over the joint ``degree'' distribution, and the subscript $\vec{d^\ast}$ for this distribution prior to conditioning on \cref{eq:divisibility-1,eq:divisibility-2} (which is therefore a collection of independent binomials if we reveal $\vec{p}$). Now we have
\begin{align*}
\rho&=\mb{P}_{\vec{p},\vec{d}}[\text{\cref{eq:degree-A-ineq,eq:degree-B-ineq}}]=(1+o(1))\int_{\alpha_{ij},\alpha_{ij,i'j'}}\mb{P}_{\vec{d}}[\text{\cref{eq:degree-A-ineq,eq:degree-B-ineq}}]\cdot\frac{e^{-\frac{1}{2}\sum\alpha_{ij}^2-\frac{1}{2}\sum\alpha_{ij,i'j'}^2}}{(2\pi)^5}d\alpha\\
&\lesssim e^{-32n}+\int_{\snorm{\alpha}_\infty\le 8\sqrt{n}}\mb{P}_{\vec{d}}[\text{\cref{eq:degree-A-ineq,eq:degree-B-ineq}}]e^{-\frac{1}{2}\sum\alpha_{ij}^2-\frac{1}{2}\sum\alpha_{ij,i'j'}^2}d\alpha\\
&\lesssim e^{-32n}+\int_{\snorm{\alpha}_\infty\le 8\sqrt{n}}\mb{P}_{\vec{d^\ast}}[\text{\cref{eq:degree-A-ineq,eq:degree-B-ineq}}|\text{\cref{eq:divisibility-1,eq:divisibility-2}}]e^{-\frac{1}{2}\sum\alpha_{ij}^2-\frac{1}{2}\sum\alpha_{ij,i'j'}^2}d\alpha\\
&\lesssim e^{-32n}+\int_{\snorm{\alpha}_\infty\le 8\sqrt{n}}\mb{P}_{\vec{d^\ast}}[\text{\cref{eq:degree-A-ineq,eq:degree-B-ineq}}]\frac{\mb{P}_{\vec{d^\ast}}[\text{\cref{eq:divisibility-1,eq:divisibility-2}}|\text{\cref{eq:degree-A-ineq,eq:degree-B-ineq}}]}{\mb{P}_{\vec{d^\ast}}[\text{\cref{eq:divisibility-1,eq:divisibility-2}}]}e^{-\frac{1}{2}\sum\alpha_{ij}^2-\frac{1}{2}\sum\alpha_{ij,i'j'}^2}d\alpha.
\end{align*}
The $1+o(1)$ comes from the probability that the $\alpha$ are such that each $p_{ij}\in(0,1)$, etc., and the $e^{-32n}$ comes from the fact that there is a very low contribution to the integral when any coordinate of $\alpha$ is bigger than $8\sqrt{n}$ in magnitude. We remark that the sum $\sum\alpha_{ij,i'j'}^2$ is over pairs with $i<i'$ or $i=i'$ and $j<j'$, and that furthermore the integration over $\alpha_{ij,i'j'}$ is only done over such pairs; recall we are letting $\alpha_{ij,i'j'}=\alpha_{i'j',ij}$.

\subsubsection{Estimating the Bayes ratio}\label{sub:second-bayes-ratio}
Now we estimate the ratio of probabilities in the final integrand using local central limit theorems. As we only care about $\rho$ up to constant factors, we will be somewhat loose with details (see e.g.~\cite[Section~2]{SS21maj} for a similar computation carried out with precision). The denominator is the chance that in the purely independent distribution (which we labeled by $\vec{d^\ast}$), we have the given divisibilities and equalities. For the denominator, note that each condition in \cref{eq:divisibility-1,eq:divisibility-2} is independent of each other. Furthermore, each equality in \cref{eq:divisibility-2} is tantamount to a sum of independent binomials being equal to another sum of independent binomials (and each divisibility is making sure a sum of independent binomials is even). We easily deduce that the denominator has order of magnitude $\Theta((1/2)^4(1/n)^6)=\Theta(n^{-6})$ by e.g.~using a local central limit theorem (or explicit binomial calculations). Here we are implicitly using that $\snorm{\alpha}_\infty$ is not too large.

Now we focus on the numerator. It is similar to the denominator, but we first condition on the inequalities \cref{eq:degree-A-ineq,eq:degree-B-ineq}. At this stage it suffices to give an upper bound, so we focus just on bounding $\mb{P}[\cref{eq:divisibility-2}|\text{\cref{eq:degree-A-ineq,eq:degree-B-ineq}}]$.

Note that even after conditioning, the ``degrees'' $(d_v^{(i,j)},d_v^{(i,j,i,j+1)},d_v^{(i,j,i+1,j+1)},d_v^{(i,j,i+1,j)})$ are independent over all $v\in A_i\cap B_j$ (and choices of $i,j\in\{1,2\}$). Therefore, the resulting upper bound will follow from an Erd\H{o}s--Littlewood--Offord-style argument. Explicitly, first reveal the ``degrees'' for all $v\in A_1\cap B_1$. Then, reveal them for $v\in A_1\cap B_2$ and consider the probability that \cref{eq:divisibility-2} holds for $i=1,j=1,i'=1,j'=2$. Then, reveal them for $v\in A_2\cap B_1$ and consider the probability that \cref{eq:divisibility-2} holds for $i'=2,j'=1$ and $(i,j)\in\{(1,1),(1,2)\}$. Finally, reveal them for $v\in A_2\cap B_2$ and consider the probability that \cref{eq:divisibility-2} holds for $i'=2,j'=2$ and $(i,j)\in\{(1,1),(1,2),(2,1)\}$. We claim that the probabilities are, conditional on the prior revelations, bounded by $O(n^{-1})$, $O(n^{-2})$, and $O(n^{-3})$, respectively. We focus on the last claim: note that for each $v\in A_2\cap B_2$ the distribution of $(d_v^{(2,2)},d_v^{(2,2,2,1)},d_v^{(2,2,1,1)},d_v^{(2,2,1,2)})$ is that of
\[\big(\mr{Bin}(n_{22}-1,p_{22}),\mr{Bin}(n_{21},p_{22,21}),\mr{Bin}(n_{11},p_{22,11}),\mr{Bin}(n_{12},p_{22,12})\big)\]
conditional on inequalities coming from \cref{eq:degree-A-ineq,eq:degree-B-ineq}. It is not hard to see that there is a $4$-dimensional axis-aligned lattice cube of dimensions $\Omega(\sqrt{n})$ such that every lattice point in the cube is obtained by this conditional distribution with probability $\Omega(n^{-4/2})$ (this follows almost immediately from the same property for the unconditioned binomial distribution). Projecting, this means that there is a $3$-dimensional axis-aligned cube of dimensions $\Omega(\sqrt{n})$ such that every lattice point in the cube is obtained by the distribution $(d_v^{(2,2,2,1)},d_v^{(2,2,1,1)},d_v^{(2,2,1,2)})$ with probability $\Omega(n^{-3/2})$.

Now it suffices to show that the sum of $n_{22}=\Omega(n)$ independent copies of this distribution has the property that every point in $\mb{R}^3$ is obtained with probability $O(n^{-3})$. If we knew the distribution was in fact uniform on a $3$-dimensional cube of dimensions $\Omega(\sqrt{n})$, the result would be immediate by considering the dimensions independently and some explicit computations (or the L\'evy--Kolmogorov--Rogozin inequality \cite{Rog61}). We can reduce to this situation by the following process: after sampling $X=x\in\mb{R}^3$ to add to our sum, we sample a weighted coin which turns heads with probability $p_x\in[0,1]$, choosing the weights $p_x$ to ensure that conditional on heads appearing, the outcome of $X$ is uniform on the aforementioned lattice cube. It is easy to check that we can further do this so that with probability $1-\exp(-\Omega(n))$, there will be $\Omega(n)$ many heads appearing. Reveal which indices are heads and tails, and then notice that the sum corresponding to indices where heads appears is an independent sum of random variables uniform on lattice cubes.

With the claims in hand, the ultimate upshot is that
\[\frac{\mb{P}_{\vec{d^\ast}}[\text{\cref{eq:divisibility-1,eq:divisibility-2}}|\text{\cref{eq:degree-A-ineq,eq:degree-B-ineq}}]}{\mb{P}_{\vec{d^\ast}}[\text{\cref{eq:divisibility-1,eq:divisibility-2}}]}\le O(1).\]

\subsubsection{The independent model}\label{sub:second-independent}
We now deduce
\[\rho\lesssim e^{-32n}+\int_{\snorm{\alpha}_\infty\le 8\sqrt{n}}\mb{P}_{\vec{d^\ast}}[\text{\cref{eq:degree-A-ineq,eq:degree-B-ineq}}]e^{-\frac{1}{2}\sum\alpha_{ij}^2-\frac{1}{2}\sum\alpha_{ij,i'j'}^2}d\alpha,\]
and now we are integrating over probabilities coming from a purely independent model (which depends on the parameters $\alpha$). Now at this stage let us define $p_{ij}=1/2+\alpha_{ij}^\ast/\sqrt{4n_{ij}^2}$ and $p_{ij,i'j'}=1/2+\alpha_{ij,i'j'}^\ast/\sqrt{8n_{ij}n_{i'j'}}$.

Let
\[\rho^\ast:=\int_{\snorm{\alpha^\ast}_\infty\le8\sqrt{n}}\prod_{v\in V(G)}\mb{P}_{\vec{d^\ast}}[\text{\cref{eq:degree-A-ineq,eq:degree-B-ineq} for }v]e^{-\frac{1}{2}\sum\alpha_{ij}^{\ast2}-\frac{1}{2}\sum\alpha_{ij,i'j'}^{\ast2}}d\alpha^\ast,\]
and note that $\alpha^\ast=(1+O(1/n))\alpha$ and $\snorm{\alpha}_\infty\le8\sqrt{n}$ in the integrand easily imply that
\begin{equation}\label{eq:second-moment-rho-rho-star}
\rho\lesssim e^{-32n}+\rho^\ast.
\end{equation}
(Note that the shift from $\alpha$ to $\alpha^\ast$ slightly changes the region of integration, but this error can be folded into the $e^{-32n}$ error term already present.)

Now we focus on $\rho^\ast$. Recall that $n_{11}=n_{22}=\beta n$ and $n_{12}=n_{21}=(1-\beta)n$, and note that each condition \cref{eq:degree-A-ineq,eq:degree-B-ineq} for $v$ corresponds to some binomial sum and difference inequalities. These probabilities can be estimated to within $O(1/n)$ error by the results in \cref{app:binomial}. For the second moment we use \cref{lem:second-moment-estimate} to find
\begin{align*}
\rho^\ast&=\int_{\snorm{\alpha}_\infty\le 8\sqrt{n}}\prod_{v\in V(G)}\mb{P}_{\vec{d^\ast}}[\text{\cref{eq:degree-A-ineq,eq:degree-B-ineq} for }v]e^{-\frac{1}{2}\sum\alpha_{ij}^{\ast2}-\frac{1}{2}\sum\alpha_{ij,i'j'}^{\ast2}}d\alpha^\ast\\
&=\int_{\snorm{\alpha^\ast}_\infty\le8\sqrt{n}}(q_{11}(\alpha)q_{22}(\alpha))^{\beta n}(q_{12}(\alpha)q_{21}(\alpha))^{(1-\beta)n}e^{-\frac{1}{2}\sum\alpha_{ij}^{\ast2}-\frac{1}{2}\sum\alpha_{ij,i'j'}^{\ast2}}d\alpha^\ast
\end{align*}
where
\[q_{ij}(\alpha):=\mb{P}\bigg[\genfrac{}{}{0pt}{}{X^{(i,j)}-X^{(i,j,i+1,j+1)}+X^{(i,j,i,j+1)}-X^{(i,j,i+1,j)}\ge\gamma\sqrt{n}}{X^{(i,j)}-X^{(i,j,i+1,j+1)}-X^{(i,j,i,j+1)}+X^{(i,j,i+1,j)}\ge\gamma\sqrt{n}}\bigg]\]
with $X^{(i,j)}\sim\mr{Bin}(n_{ij}-1,p_{ij})$ and $X^{(i,j,i',j')}\sim\mr{Bin}(n_{i'j'},p_{ij,i'j'})$.

We respectively apply \cref{lem:second-moment-estimate} with the following parameters:
\begin{align*}
k&=\beta n,&(a_1,a_2,a_3,a_4)&=\Big(\frac{\alpha_{11}^\ast}{\sqrt{\beta^2n}},\frac{\alpha_{11,12}^\ast}{\sqrt{2\beta(1-\beta)n}},\frac{\alpha_{11,22}^\ast}{\sqrt{2\beta^2n}},\frac{\alpha_{11,21}^\ast}{\sqrt{2\beta(1-\beta)n}}\Big);\\
k&=\beta n,&(a_1,a_2,a_3,a_4)&=\Big(\frac{\alpha_{22}^\ast}{\sqrt{\beta^2n}},\frac{\alpha_{22,12}^\ast}{\sqrt{2\beta(1-\beta)n}},\frac{\alpha_{22,11}^\ast}{\sqrt{2\beta^2n}},\frac{\alpha_{22,21}^\ast}{\sqrt{2\beta(1-\beta)n}}\Big);\\
k&=(1-\beta)n,&(a_1,a_2,a_3,a_4)&=\Big(\frac{\alpha_{12}^\ast}{\sqrt{(1-\beta)^2n}},\frac{\alpha_{12,11}^\ast}{\sqrt{2\beta(1-\beta)n}},\frac{\alpha_{12,21}^\ast}{\sqrt{2(1-\beta)^2n}},\frac{\alpha_{12,22}^\ast}{\sqrt{2\beta(1-\beta)n}}\Big);\\
k&=(1-\beta)n,&(a_1,a_2,a_3,a_4)&=\Big(\frac{\alpha_{21}^\ast}{\sqrt{(1-\beta)^2n}},\frac{\alpha_{21,11}^\ast}{\sqrt{2\beta(1-\beta)n}},\frac{\alpha_{21,12}^\ast}{\sqrt{2(1-\beta)^2n}},\frac{\alpha_{21,22}^\ast}{\sqrt{2\beta(1-\beta)n}}\Big).\\
\end{align*}
These provide expressions of the form $g(\tau,a_1,a_2,a_3,a_4)$ (\cref{def:binomial-gaussian-probability-2}) where $\tau\in\{\beta,1-\beta\}$ for the functions $q_{11}(\alpha^\ast),q_{22}(\alpha^\ast),q_{12}(\alpha^\ast),q_{21}(\alpha^\ast)$, respectively, within an additive error of $O(1/n)$; call these expressions $q_{ij}^\dagger(\alpha^\ast)$, which are continuous functions (suppressing the dependence on $\gamma,\beta$). In fact, each $q_{ij}^\dagger(\alpha^\ast)$ is log-concave in $\mb{R}^{10}$ (recall $\alpha_{ij,i'j'}^\ast=\alpha_{i'j',ij}^\ast$) by \cref{lem:log-concave-2}. This is because $g(\tau,a_1,a_2,a_3,a_4)$ is seen to be log-concave in $(a_1,a_2,a_3,a_4)$. Additionally, they are lower-bounded by an absolute constant in the range $\snorm{\alpha^\ast}_\infty\le8\sqrt{n}$, which means
\begin{equation}\label{eq:q-dagger-second-moment}
\rho^\ast\lesssim\int_{\snorm{\alpha^\ast}_\infty\le8\sqrt{n}}\Big[(q_{11}^\dagger(\alpha^\ast)q_{22}^\dagger(\alpha^\ast))^{\beta n}(q_{12}^\dagger(\alpha^\ast)q_{21}^\dagger(\alpha^\ast))^{(1-\beta)n}\Big]e^{-\frac{1}{2}\sum\alpha_{ij}^{\ast2}-\frac{1}{2}\sum\alpha_{ij,i'j'}^{\ast2}}d\alpha^\ast.
\end{equation}
Now notice that the initial portion of the integrand in square brackets has the property that it is symmetric under simultaneously switching $\alpha_{12}^\ast\leftrightarrow\alpha_{21}^\ast,\alpha_{12,11}^\ast\leftrightarrow\alpha_{21,11}^\ast,\alpha_{12,22}^\ast\leftrightarrow\alpha_{21,22}^\ast$ (doing so in a way that continues to enforce the symmetry $\alpha_{ij,i'j'}^\ast=\alpha_{i'j',ij}^\ast$). This can be seen as ``switching the roles of partitions $A$ and $B$'', and one can see that it amounts to switching the indices ``$12$'' and ``$21$'' wherever they appear. Additionally, the initial portion of the integrand is symmetric under simultaneously switching $\alpha_{11}^\ast\leftrightarrow\alpha_{22}^\ast,\alpha_{11,12}^\ast\leftrightarrow\alpha_{22,12}^\ast,\alpha_{11,21}^\ast\leftrightarrow\alpha_{22,21}^\ast$. This can be seen as ``switching the roles of $A_1$ and $B_2$, as well as $A_2$ and $B_1$'', and it amounts to switching ``$11$'' and ``$22$'' wherever they appear.

We can apply these symmetries and log-concavity to obtain an upper bound on the integrand in \cref{eq:q-dagger-second-moment}. (As a basic example, if $g(x,y)$ is log-concave and symmetric then $g(x,y)=\sqrt{g(x,y)g(y,x)}\le g((x+y)/2,(x+y)/2)$.) We obtain
\begin{equation}\label{eq:log-concavity-second-moment-1}(q_{11}^\dagger(\alpha^\ast)q_{22}^\dagger(\alpha^\ast))^{\beta n}(q_{12}^\dagger(\alpha^\ast)q_{21}^\dagger(\alpha^\ast))^{(1-\beta)n}\le(q_{11}^\dagger(\alpha')q_{22}^\dagger(\alpha'))^{\beta n}(q_{12}^\dagger(\alpha')q_{21}^\dagger(\alpha'))^{(1-\beta)n}
\end{equation}
where
\[\alpha_{11}'=\alpha_{22}'=\frac{\alpha_{11}^\ast+\alpha_{22}^\ast}{2},\quad\alpha_{12}'=\alpha_{21}'=\frac{\alpha_{12}^\ast+\alpha_{21}^\ast}{2},\]
\[\alpha_{11,12}'=\alpha_{11,21}'=\alpha_{22,12}'=\alpha_{22,21}'=\frac{\alpha_{11,12}^\ast+\alpha_{11,21}^\ast+\alpha_{22,12}^\ast+\alpha_{22,21}^\ast}{4},\]
\[\alpha_{12,21}'=\alpha_{12,21}^\ast,\quad\alpha_{11,22}'=\alpha_{11,22}^\ast.\]
We have
\[q_{11}^\dagger(\alpha')=g\Big(\beta,\frac{\alpha_{11}'}{\sqrt{\beta^2n}},\frac{\alpha_{11,12}'}{\sqrt{2\beta(1-\beta)n}},\frac{\alpha_{11,22}'}{\sqrt{2\beta^2n}},\frac{\alpha_{11,21}'}{\sqrt{2\beta(1-\beta)n}}\Big)\]
and similar for the other values. But from \cref{def:binomial-gaussian-probability-2} we see that $g(\tau,a_1,a_2,a_3,a_4)$ depends only on $\beta,a_1-a_3,a_2-a_4$. By $\alpha_{11,12}'=\alpha_{11,21}'$ and similar we may deduce that
\begin{align*}
q_{11}^\dagger(\alpha')&=g\Big(\beta,\frac{\alpha_{11}'\sqrt{2}-\alpha_{11,22}'}{\sqrt{2\beta^2n}},0,0,0\Big),\quad &q_{22}^\dagger(\alpha')&=g\Big(\beta,\frac{\alpha_{22}'\sqrt{2}-\alpha_{22,11}'}{\sqrt{2\beta^2n}},0,0,0\Big),\\
q_{12}^\dagger(\alpha')&=g\Big(1-\beta,\frac{\alpha_{12}'\sqrt{2}-\alpha_{12,21}'}{\sqrt{2(1-\beta)^2n}},0,0,0\Big),\quad &q_{21}^\dagger(\alpha')&=g\Big(1-\beta,\frac{\alpha_{21}'\sqrt{2}-\alpha_{21,12}'}{\sqrt{2(1-\beta)^2n}},0,0,0\Big).
\end{align*}
Now since $g(\tau,a_1,a_2,a_3,a_4)$ is log-concave in $(a_1,a_2,a_3,a_4)$, we see
\begin{align}\label{eq:log-concavity-second-moment-2}
q_{11}^\dagger(\alpha')q_{22}^\dagger(\alpha')&\le g\Big(\beta,\frac{\alpha_{11}'\sqrt{2}-\alpha_{11,22}'}{\sqrt{2\beta^2n}},0,0,0\Big)^2=f\Big(\beta,\frac{\alpha_{11,22}'-\alpha_{11}'\sqrt{2}}{2\sqrt{2n}}\Big)^2,\\
q_{12}^\dagger(\alpha')q_{21}^\dagger(\alpha')&\le g\Big(1-\beta,\frac{\alpha_{12}'\sqrt{2}-\alpha_{12,21}'}{\sqrt{2(1-\beta)^2n}},0,0,0\Big)^2=f\Big(1-\beta,\frac{\alpha_{12,21}'-\alpha_{12}'\sqrt{2}}{2\sqrt{2n}}\Big)^2,\label{eq:log-concavity-second-moment-3}
\end{align}
using $\alpha_{11}'=\alpha_{22}'$, $\alpha_{11,22}'=\alpha_{22,11}'$, and similar, as well as the relation between $g$ and $f$ described in \cref{def:binomial-gaussian-probability-2} (recall $f$ from \cref{def:special-func}). Now \cref{eq:q-dagger-second-moment,eq:log-concavity-second-moment-1,eq:log-concavity-second-moment-2,eq:log-concavity-second-moment-3} yield
\begin{equation}\label{eq:before-rotation}
\rho^\ast\lesssim\int_{\snorm{\alpha^\ast}_\infty\le8\sqrt{n}}f\Big(\beta,\frac{\alpha_{11,22}'-\alpha_{11}'\sqrt{2}}{2\sqrt{2n}}\Big)^{2\beta n}f\Big(1-\beta,\frac{\alpha_{12,21}'-\alpha_{12}'\sqrt{2}}{2\sqrt{2n}}\Big)^{2(1-\beta)n} e^{-\frac{1}{2}\sum\alpha_{ij}^{\ast2}-\frac{1}{2}\sum\alpha_{ij,i'j'}^{\ast2}}d\alpha^\ast.
\end{equation}
Now,
\[\frac{\alpha_{11,22}'-\alpha_{11}'\sqrt{2}}{\sqrt{2}},\frac{\alpha_{11,22}'+\alpha_{11}'\sqrt{2}}{\sqrt{2}},\frac{\alpha_{12,21}'-\alpha_{12}'\sqrt{2}}{\sqrt{2}},\frac{\alpha_{12,21}'+\alpha_{12}'\sqrt{2}}{\sqrt{2}}\]
are linear combinations of $\alpha^\ast$, i.e., $v_1\cdot\alpha^\ast,v_2\cdot\alpha^\ast,v_3\cdot\alpha^\ast,v_4\cdot\alpha^\ast$ for vectors $v_1,v_2,v_3,v_4\in\mb{R}^{10}$ (here we are assuming $\alpha^\ast$ only contains one of each symmetric pair $\alpha_{ij,i'j'}^\ast=\alpha_{i'j',ij}^\ast$). One can easily check that $v_1,v_2,v_3,v_4$ are orthonormal vectors in $\mb{R}^{10}$. Thus, there exists a completion to an orthonormal basis $\{v_1,\ldots,v_{10}\}$ so that
\begin{align*}
\sum\alpha_{ij}^{\ast2}+\sum\alpha_{ij,i'j'}^{\ast2}&=\frac{(\alpha_{11,22}'-\alpha_{11}'\sqrt{2})^2+(\alpha_{11,22}'+\alpha_{11}'\sqrt{2})^2+(\alpha_{12,21}'-\alpha_{12}'\sqrt{2})^2+(\alpha_{12,21}'+\alpha_{12}'\sqrt{2})^2}{2}\\
&\qquad\qquad+\sum_{j=5}^{10}(v_j\cdot\alpha^\ast)^2.
\end{align*}
Reparametrizing the integration coordinates of \cref{eq:before-rotation} via this basis change, we find
\begin{align}
\rho^\ast&\lesssim\int_{\snorm{x}_\infty\le32\sqrt{n}}f\Big(\beta,\frac{x_1}{2\sqrt{n}}\Big)^{2\beta n}f\Big(1-\beta,\frac{x_2}{2\sqrt{n}}\Big)^{2(1-\beta)n}e^{-\frac{1}{2}\sum_{j=1}^{10}x_j^2}dx\notag\\
&\lesssim\int_{|x_1|,|x_2|\le32\sqrt{n}}f\Big(\beta,\frac{x_1}{2\sqrt{n}}\Big)^{2\beta n}f\Big(1-\beta,\frac{x_2}{2\sqrt{n}}\Big)^{2(1-\beta)n}e^{-\frac{1}{2}(x_1^2+x_2^2)}dx\notag\\
&\lesssim n\int_{|x_1|,|x_2|\le16}f(\beta,x_1)^{2\beta n}f(1-\beta,x_2)^{2(1-\beta)n}\exp(-(2x_1^2+2x_2^2)n)dx\notag\\
&\lesssim n\int_{|x_1|,|x_2|\le16}\exp\big(n\big(-2x_1^2-2x_2^2+2\beta\log f(\beta,x_1)+2(1-\beta)\log f(1-\beta,x_2)\big)\big)dx.\label{eq:second-moment-log-1}
\end{align}

Let $Q$ be the region $[-16,16]^2\subseteq\mb{R}^2$. Finally, \cref{eq:second-moment-atypical,eq:second-moment-transference,eq:second-moment-rho-rho-star,eq:second-moment-log-1} demonstrate
\begin{align*}
\mb{P}[Y_A(n_1')=Y_B(n_1')=1]=\mb{P}[\mc{E}_2]&\lesssim e^{-32n}+n\int_Qe^{n(-2x_1^2-2x_2^2+2\beta\log f(\beta,x_1)+2(1-\beta)\log f(1-\beta,x_2))}dx.\\
&\lesssim n\int_Qe^{n(-2x_1^2-2x_2^2+2\beta\log f(\beta,x_1)+2(1-\beta)\log f(1-\beta,x_2))}dx.
\end{align*}
for $n_1'/n\in[c,1-c]$. Here we eliminated $e^{-32n}$ by noting that the region near $x_1=x_2=-2$ contributes at least $\exp(-32n)$, using $|\gamma|\le 1$.

Combining this with \cref{eq:second-moment,eq:second-moment-small-overlap} we find
\begin{align*}
\mb{E}X_\gamma^2&\lesssim(1+\eta)^n\mb{E}X_\gamma+\binom{2n}{n}\sum_{c\le\beta\le1-c}\binom{n}{\beta n}^2n\int_Qe^{n(-2x_1^2-2x_2^2+2\beta\log f(\beta,x_1)+2(1-\beta)\log f(1-\beta,x_2))}dx\\
&\lesssim\frac{1}{\sqrt{n}}\sum_{c\le\beta\le1-c}\int_Qe^{n(2\log 2-2\beta\log\beta-2(1-\beta)\log(1-\beta)-2x_1^2-2x_2^2+2\beta\log f(\beta,x_1)+2(1-\beta)\log f(1-\beta,x_2))}dx\\
&\qquad+(1+\eta)^n\mb{E}X_\gamma\\
&\lesssim(1+\eta)^n\mb{E}X_\gamma+\frac{1}{\sqrt{n}}\sum_{c\le\beta\le1-c}\int_Qe^{nF_2(\beta,x_1,x_2)}dx.
\end{align*}
Here the sum is over $\beta$ with $\beta n\in\mb{Z}$. We used Stirling's formula.

Now we split the sum over $\beta$ into two regions. First we consider $\beta\in[c,.001]\cup[.999,1-c]$. Note that $F_2$ is symmetric under $\beta\leftrightarrow1-\beta$, so up to a constant factor loss we may assume $\beta\in[c,.001]$. By the first part of \cref{asm:comp}, we see $F_2(\beta,x_1,x_2)\le 2\sup_{y\in\mb{R}}F_1(y)$, so
\[\frac{1}{\sqrt{n}}\sum_{\beta\in[c,.001]\cup[.999,1-c]}\int_Qe^{nF_2(\beta,x_1,x_2)}dx\lesssim(1+\eta/2)^n\exp(2\sup_{y\in\mb{R}}F_1(y)).\]
Combining with the lower bound in \cref{lem:first-moment} (established below), this is bounded by $(1+\eta)^n\mb{E}X_\gamma$ so can be absorbed.

Finally, we consider $\beta\in[.001,.999]$. By the second part of \cref{asm:comp}, the unique optimizer of $F_2(\beta,x_1,x_2)$ for $\beta\in[.001,.999]$ and $x_1,x_2\in\mb{R}$ occurs at $(1/2,x^\ast,x^\ast)$ where $x^\ast=\argmax_{x\in\mb{R}}F_1(x)$. Additionally, the Hessian of $F_2$ evaluated at $(1/2,x^\ast,x^\ast)$ is strictly negative definite. Therefore there is some small absolute constants $\theta,\theta'>0$ and a box $Q^\ast=(1/2,x^\ast,x^\ast)+[-\theta,\theta]^3$ such that for all $(\beta,x_1,x_2)\in Q^\ast$ we have
\[F_2(\beta,\alpha_1,\alpha_2)\le F_2(1/2,x^\ast,x^\ast)-\theta((\beta-1/2)^2+(x_1-x^\ast)^2+(x_2-x^\ast)^2)\]
and also $F_2(\beta,x_1,x_2)\le F_2(1/2,x^\ast,x^\ast)-\theta'$ for $(\beta,x_1,x_2)\notin Q^\ast$ satisfying $\beta\in[.001,.999]$. Letting $Q'=(x^\ast,x^\ast)+[-\theta,\theta]^2$, we therefore find
\begin{align*}
\frac{1}{\sqrt{n}}\sum_{.001\le\beta\le.999}\int_Qe^{nF_2(\beta,x_1,x_2)}dx&\lesssim\frac{1}{\sqrt{n}}\sum_{\beta\in[1/2\pm\theta]}\int_{Q'}e^{nF_2(1/2,x^\ast,x^\ast)-\theta n((\beta-1/2)^2+(x_1-x^\ast)^2+(x_2-x^\ast)^2)}dx_1dx_2\\
&\qquad+e^{(F_2(1/2,x^\ast,x^\ast)-\theta')n}\\
&\lesssim\frac{e^{nF_2(1/2,x^\ast,x^\ast)}}{n^{3/2}}\sum_{\beta\in[1/2\pm\theta]}e^{-\theta n(\beta-1/2)^2}+e^{(F_2(1/2,x^\ast,x^\ast)-\theta')n}\\
&\lesssim n^{-1}e^{nF_2(1/2,x^\ast,x^\ast)}\lesssim 16^{-n}\binom{2n}{n}^2\exp\bigg(n\sup_{\substack{\beta\in[0,1]\\\alpha_1,\alpha_2\in \mb{R}}} F_2(\beta,\alpha_1,\alpha_2)\bigg).
\end{align*}
We deduce
\[\mb{E}X_\gamma^2\lesssim(1+\eta)^n\mb{E}X_\gamma+16^{-n}\binom{2n}{n}^2\exp\bigg(n\sup_{\substack{\beta\in[0,1]\\\alpha_1,\alpha_2\in \mb{R}}} F_2(\beta,\alpha_1,\alpha_2)\bigg),\]
as desired, proving \cref{lem:second-moment}.

\subsection{The first moment of \texorpdfstring{$X_\gamma$}{Xgamma}}\label{sub:first-moment}
Now we prove \cref{lem:first-moment}. As discussed earlier, it is significantly easier and we truncate some details (although take note that \cref{lem:first-moment} provides a matching lower bound unlike \cref{lem:second-moment}, which we used in the proofs of \cref{thm:main,lem:second-moment}).

\subsubsection{Setup and transference}\label{sub:first-setup}
Let $A_1=[n]$ and $A_2=\{n+1,\ldots,2n\}$. Let $\mc{E}_1$ be the event that $G\sim\mb{G}(2n,1/2)$ is such that $A_1\cup A_2$ is a $\gamma\sqrt{n}$-friendly partition. By linearity of expectation we have
\begin{equation}\label{eq:first-moment-linearity}
\mb{E}X_\gamma=\binom{2n}{n}\mb{P}[\mc{E}_1].
\end{equation}
By \cref{prop:friendly-var-deg,prop:large-degree-1} with $\eta$ replaced by $\min(\eps_{\ref{lem:integrate-p-graph}},\eps_{\ref{lem:integrate-p-bip-graph}})/2$, we have
\begin{equation}\label{eq:first-moment-atypical}
\mb{P}[\mc{E}_1]=(1\pm n^{-\omega(1)})\mb{P}[\mc{E}_1\wedge\mc{E}_{\mr{irreg},1}^c\wedge\mc{E}_{\mr{var}}^\ast].
\end{equation}
Here we define $\mc{E}_{\mr{var}}^\ast$ just to be the appropriate set of conditions (in a sense a subset of those from $\mc{E}_{\mr{var}}$) guaranteeing that the degree sequences of $G[A_1],G[A_2],G[A_1,A_2]$ satisfy \cref{A2} of \cref{lem:integrate-p-graph} or \cref{B5} of \cref{lem:integrate-p-bip-graph}. So $\mc{E}_{\mr{irreg},1}^c,\mc{E}_{\mr{var}}^\ast$ ensure that for the $2+1$ graphs $G[A_1]$ and $G[A_2]$ and bipartite $G[A_1,A_2]$, the degree sequences satisfy \cref{A1,A2,A3} of \cref{lem:integrate-p-graph} and \cref{B1,B2,B3,B4,B5} of \cref{lem:integrate-p-bip-graph}, respectively. We deduce
\begin{equation}\label{eq:first-moment-transfer}
\mb{P}[\mc{E}_1\wedge\mc{E}_{\mr{irreg},1}^c\wedge\mc{E}_{\mr{var}}^\ast]\asymp\rho',\end{equation}
where $\rho'$ is the following probability:
\begin{itemize}
    \item For $i\in\{1,2\}$ sample the ``degrees'' $d_v^i$ for $v\in A_i$ (which should be thought of as corresponding to the internal degrees $\deg(v,A_i)$ for $v\in A_i$) from $\mc{I}_{1/2}^n$;
    \item Sample the ``degrees'' $d_v^{(1,2)}$ for $v\in A_1$ and $d_v^{(2,1)}$ for $v\in A_2$ (which correspond to the bipartite degrees $\deg(v,A_2)$ for $v\in A_1$ and vice versa) jointly from $\mc{I}_{1/2}^{n,n}$;
    \item Compute the probability that (a) for each $i\in\{1,2\}$ and $v\in A_i$ we have
    \begin{align}\label{eq:degree-ineq}
    d_v^i-d_v^{(i,i+1)}&\ge\gamma\sqrt{n},
    \end{align}
    where indices are taken $\imod{2}$, and (b) the resulting degree sequences satisfy $\mc{E}_{\mr{irreg},1}^c\wedge\mc{E}_{\mr{var}}^\ast$.
\end{itemize}
This can be interpreted as the probability that if the degrees of our $1+2$ graphs are sampled as in \cref{lem:integrate-p-graph,lem:integrate-p-bip-graph}, we have that $A_1\cup A_2$ is $\gamma\sqrt{n}$-friendly (and the degree sequences are not too irregular). Let $\rho$ be the same probability without condition (b).

Understanding the ``degree'' distribution is relatively simple:
\begin{itemize}
    \item Sample $2+1$ parameters $p_1,p_2\sim\mc{N}(1/2,1/(4n^2-4n))$ and $p_{1,2}\sim\mc{N}(1/2,1/(8n^2))$ (all conditional on being in $(0,1)$), and let $p_{2,1}=p_{1,2}$;
    \item Sample $d_v^i\sim\mr{Bin}(n-1,p_i)$ and $d_v^{(i,i+1)}\sim\mr{Bin}(n,p_{i,i+1})$ completely independently;
    \item Condition on the $2$ divisibility conditions and $1$ equality condition
    \begin{align}\label{eq:first-divisibility-1}
    2&\mid\sum_{v\in A_i}d_v^i,\\
    \sum_{v\in A_1}d_v^{(1,2)}&=\sum_{v\in A_2}d_v^{(2,1)}.\label{eq:first-divisibility-2}
    \end{align}
\end{itemize}
The first step can be handled by taking a Gaussian integral over the $2+1$ parameters, and the second step is independent. Again we handle the conditioning from the third step by use of Bayes' theorem and local central limit theorems similar to \cref{sub:second-bayes-ratio}, but now the local central limit results used will be purely $1$-dimensional. We will also require a two-sided bound unlike before.

\subsubsection{Upper bound}\label{sub:first-upper}
We clearly have $\rho'\le\rho$. Now let $p_i=1/2+\alpha_i/\sqrt{4n^2-4n}$ and $p_{1,2}=1/2+\alpha_{1,2}/\sqrt{8n^2}$. Let $\vec{d}$ denote the distribution over the joint ``degree'' distribution, and $\vec{d^\ast}$ denote the distribution prior to conditioning on \cref{eq:first-divisibility-1,eq:first-divisibility-2}. We have, similar to in \cref{sub:second-bayes-theorem},
\begin{align*}
\rho&\lesssim\int_\alpha\mb{P}_{\vec{d}}[\cref{eq:degree-ineq}]e^{-\frac{1}{2}(\alpha_1^2+\alpha_2^2+\alpha_{1,2}^2)}d\alpha\\
&\lesssim e^{-8n}+\int_{\snorm{\alpha}_\infty\le 4\sqrt{n}}\mb{P}_{\vec{d^\ast}}[\cref{eq:degree-ineq}]\frac{\mb{P}_{\vec{d^\ast}}[\text{\cref{eq:first-divisibility-1,eq:first-divisibility-2}}|\cref{eq:degree-ineq}]}{\mb{P}_{\vec{d^\ast}}[\text{\cref{eq:first-divisibility-1,eq:first-divisibility-2}}]}e^{-\frac{1}{2}(\alpha_1^2+\alpha_2^2+\alpha_{1,2}^2)}d\alpha.
\end{align*}

Now, the denominator is easily seen to be $\Theta(1/n)$. Indeed, for the denominator we require two independent sums of random variables to be even, and a signed sum of random variables to be $0$. It follows e.g.~from a local central limit theorem for log-concave random variables (\cite{Ben73}) and the fact that the mean of this signed sum can be seen to be within standard-deviation-range (up to a constant factor) of $0$. (Here we are using that the binomial random variable is log-concave, hence a pair of independent binomials is log-concave, hence conditioning on a single linear inequality keeps the distribution log-concave, and finally that the marginal of a log-concave discrete distribution is log-concave \cite{HKS19}.)

To estimate the numerator, first note that upon conditioning on \cref{eq:degree-ineq} and all values $d_v^{(1,2)},d_v^{(2,1)}$, we are left with independent randomness for the $d_v^1$ and $d_v^2$ values. The probability both sums are even is easily seen to be approximately $1/4$. Thus, we find
\[\mb{P}_{\vec{d^\ast}}[\text{\cref{eq:first-divisibility-1,eq:first-divisibility-2}}|\cref{eq:degree-ineq}]=(1/4+o(1))\mb{P}_{\vec{d^\ast}}[\cref{eq:first-divisibility-2}|\cref{eq:degree-ineq}].\]
Now note that the $d_v^{(1,2)},d_v^{(2,1)}$ are jointly independent and thus it is easy to see (e.g.~via the L\'evy--Kolmogorov--Rogozin inequality \cite{Rog61}) that $\mb{P}_{\vec{d^\ast}}[\cref{eq:first-divisibility-2}|\cref{eq:degree-ineq}]\lesssim 1/n$. Overall, we deduce
\[\rho'\le\rho\lesssim e^{-8n}+\int_{\snorm{\alpha}_\infty\le4\sqrt{n}}\mb{P}_{\vec{d^\ast}}[\cref{eq:degree-ineq}]d\alpha.\]

Now at this stage let us define $p_i=1/2+\alpha_i^\ast/(2n)$ and $p_{i,i+1}=1/2+\alpha_{i,i+1}^\ast/(2n\sqrt{2})$. We see $\alpha^\ast=(1+O(1/n))\alpha$ so it is not hard to deduce
\begin{equation}\label{eq:first-moment-rho-independent}
\rho'\lesssim e^{-8n}+\int_{\snorm{\alpha^\ast}_\infty\le4\sqrt{n}}\mb{P}_{\vec{d^\ast}}[\cref{eq:degree-ineq}]e^{-\frac{1}{2}(\alpha_1^{\ast2}+\alpha_2^{\ast2}+\alpha_{1,2}^{\ast2})}d\alpha^\ast
\end{equation}
similar to the beginning of \cref{sub:second-independent}. At this stage, note that
\begin{align*}
\mb{P}_{\vec{d^\ast}}[\cref{eq:degree-ineq}]=&\mb{P}[\mr{Bin}(n-1,1/2+\alpha_1^\ast/(2n))-\mr{Bin}(n,1/2+\alpha_{1,2}^\ast/(2n\sqrt{2}))\ge\gamma\sqrt{n}]^n\\
&\cdot\mb{P}[\mr{Bin}(n-1,1/2+\alpha_2^\ast/(2n))-\mr{Bin}(n,1/2+\alpha_{1,2}^\ast/(2n\sqrt{2}))\ge\gamma\sqrt{n}]^n.
\end{align*}
By \cref{lem:first-moment-estimate}, we deduce
\begin{align}
\mb{P}_{\vec{d^\ast}}[\cref{eq:degree-ineq}]&\asymp\bigg(\mb{P}_{Z\sim\mc{N}(0,1)}\Big[Z\ge\gamma\sqrt{2}+\frac{\alpha_{1,2}^\ast-\alpha_1^\ast\sqrt{2}}{2\sqrt{n}}\Big]\mb{P}_{Z\sim\mc{N}(0,1)}\Big[Z\ge\gamma\sqrt{2}+\frac{\alpha_{1,2}^\ast-\alpha_2^\ast\sqrt{2}}{2\sqrt{n}}\Big]\bigg)^n\notag\\
&\lesssim\mb{P}_{Z\sim\mc{N}(0,1)}\Big[Z\ge\gamma\sqrt{2}+\frac{\alpha_{1,2}^\ast-(\alpha_1^\ast+\alpha_2^\ast)\sqrt{2}/2}{2\sqrt{n}}\Big]^{2n}.\label{eq:first-moment-probability}
\end{align}
by \cref{lem:log-concave-1}.

Now similar to \cref{sub:second-independent}, we may combine \cref{eq:first-moment-rho-independent,eq:first-moment-probability} and integrate the Gaussian in the direction orthogonal to $\alpha_{1,2}^\ast-(\alpha_1^\ast+\alpha_2^\ast)\sqrt{2}/2$. Using an appropriate identity
\[\alpha_1^{\ast2}+\alpha_2^{\ast2}+\alpha_{1,2}^{\ast2}=\frac{(\alpha_{1,2}^\ast\sqrt{2}-\alpha_1^\ast-\alpha_2^\ast)^2+(v_2\cdot\alpha^\ast)+(v_3\cdot\alpha^\ast)}{4}\]
and changing variables (with $x=(\alpha_{1,2}^\ast\sqrt{2}-\alpha_1^\ast-\alpha_2^\ast)/2$), we deduce
\begin{align*}
\rho'&\lesssim e^{-8n}+\int_{|x|\le 8\sqrt{n}}\mb{P}_{Z\sim\mc{N}(0,1)}\Big[Z\ge\gamma\sqrt{2}+\frac{x}{\sqrt{2n}}\Big]^{2n}e^{-\frac{1}{2}x^2}dx\\
&\lesssim e^{-8n}+\int_{|x|\le4}\mb{P}_{Z\sim\mc{N}(0,1)}[Z\ge(\gamma+x)\sqrt{2}]^{2n}e^{-2x^2n}dx.
\end{align*}
Finally, note that $h(x):=\log(\mb{P}_{Z\sim\mc{N}(0,1)}[Z\ge(\gamma+x)\sqrt{2}])$ is concave hence adding $-x^2$ makes it strictly concave. Furthermore, let $x^\ast$ be the unique maximizer of $-x^2+h(x)$ on $\mb{R}$. We easily find
\[\rho'\lesssim e^{-8n}+e^{2n\sup_{\alpha\in\mb{R}}(F_1(\alpha)-\log 2)},\]
similar to \cref{sub:second-independent}. Furthermore, we easily check that for $|\gamma|\le 1$, $e^{-8n}$ is lower order than the supremum here.

Hence, $\rho'\lesssim 4^{-n}\exp(2n\sup_{\alpha\in\mb{R}}F_1(\alpha))$. Combining with \cref{eq:first-moment-linearity,eq:first-moment-atypical,eq:first-moment-transfer} we deduce
\[\mb{E}X_\gamma\lesssim 4^{-n}\binom{2n}{n}\exp(2n\sup_{\alpha\in\mb{R}}F_1(\alpha))\]
as desired for the upper bound.

\subsubsection{Lower bound}\label{sub:first-lower}
For the lower bound, we must first find a way of removing the condition (b) above. To this end, let us further consider the event defined by $\rho'$ with the additional conditions $|p_1-p_2|\le 2/n$ and $\snorm{p-1/2}_\infty\le1/\sqrt{n}$, and say it has probability $\rho_0'$. Similarly, the event defined by $\rho$ with the additional condition $|p_1-p_2|\le 2/n$ is $\rho_0$. Clearly $\rho'\ge\rho_0'$.

We claim $\rho_0'=(1\pm n^{-\omega(1)})\rho_0$. To see this, note that $\rho_0'/\rho_0$ is the following probability: sample the ``degree'' distribution as above with the additional conditions $|p_1-p_2|\le 2/n$ and $\snorm{p-1/2}_\infty\le1/\sqrt{n}$, then further condition on \cref{eq:degree-ineq}. Then consider the probability both $\mc{E}_{\mr{irreg},1}^c$ and $\mc{E}_{\mr{var}}^\ast$ hold. It is easy to see that this new distribution can be thought of as follows: sample $p$ appropriately, then for each $v\in A_i$ sample $(d_v^i,d_v^{(i,i+1)})\sim(\mr{Bin}(n-1,p_i),\mr{Bin}(n,p_{i,i+1}))$ conditional on $d_v^i-d_v^{(i,i+1)}\ge\gamma\sqrt{n}$, then further condition on $2|\sum_{v\in A_i}d_v^i$ and $\sum_{v\in A_1}d_v^{(1,2)}=\sum_{v\in A_2}d_v^{(2,1)}$.

We claim that these latter conditions then occur with probability $\Omega((1/2)^2(1/n))$. The evenness conditions can be handled as in \cref{sub:first-upper}; the only condition of interest is \cref{eq:first-divisibility-2}. For this, if $p_1=p_2$ then it is the probability an alternating sum of i.i.d.~binomials equals $0$, and a log-concave local limit theorem \cite{Ben73} will show the result. For $|\alpha_1-\alpha_2|=O(1/n)$, we again use \cite{Ben73} and simply note that the mean of the signed sum of binomials is within standard-deviation-range (i.e., $O(n)$) of $0$.

Letting the random model formed merely by sampling $p,d_v^i,d_v^{(i,i+1)}$ and no further conditioning be labeled $\vec{d^\ast}$, we therefore see it suffices to show
\begin{equation}\label{eq:irregular-transferred}
\mb{P}_{\vec{d^\dagger}}[\mc{E}_{\mr{irreg},1}^c\wedge\mc{E}_{\mr{var}}^\ast]\le n^{-\omega(1)}.
\end{equation}
This is easy to show using concentration of measure. Note that, e.g., the $d_v^{(1,2)}$ and $d_v^{(2,1)}$ are jointly independent in $\vec{d^\dagger}$. It is thus not too hard to use concentration of measure to show \cref{eq:irregular-transferred}, though we truncate the (relatively standard) details.

Thus, $\rho'\gtrsim\rho_0$. Now, $\rho_0$ is very similar to $\rho$ and can be analyzed as in \cref{sub:first-upper}. Let $R$ denote the region in $\mb{R}^3$ defined by $\snorm{\alpha}_\infty\le2\sqrt{n}$ and $|\alpha_1-\alpha_2|\le4$. We have
\begin{align*}
\rho'&\gtrsim\int_R\mb{P}_{\vec{d}}[\cref{eq:degree-ineq}]e^{-\frac{1}{2}(\alpha_1^2+\alpha_2^2+\alpha_{1,2}^2)}d\alpha\\
&\ge\int_R\mb{P}_{\vec{d^\ast}}[\cref{eq:degree-ineq}]\frac{\mb{P}_{\vec{d^\ast}}[\text{\cref{eq:first-divisibility-1,eq:first-divisibility-2}}|\cref{eq:degree-ineq}]}{\mb{P}_{\vec{d^\ast}}[\text{\cref{eq:first-divisibility-1,eq:first-divisibility-2}}]}e^{-\frac{1}{2}(\alpha_1^2+\alpha_2^2+\alpha_{1,2}^2)}d\alpha.
\end{align*}

Now, the denominator is easily seen to be $\Theta(1/n)$ similar to \cref{sub:first-upper} and above we deduced that the numerator is $\Omega(1/n)$. Further reparametrizing to $p_i=1/2+\alpha_1^\ast/(2n)$ and $p_{i,i+1}=1/2+\alpha_{i,i+1}^\ast/(2n\sqrt{2})$, overall we find
\begin{equation}\label{eq:first-moment-rho-lower}
\rho'\gtrsim\int_{R^\ast}\mb{P}_{\vec{d^\ast}}[\cref{eq:degree-ineq}]e^{-\frac{1}{2}(\alpha_1^2+\alpha_2^2+\alpha_{1,2}^2)}d\alpha
\end{equation}
where $R^\ast$ is the region $\snorm{\alpha^\ast}_\infty\le\sqrt{n}$ and $|\alpha_1-\alpha_2|\le 2$.

Now, using the first line of \cref{eq:first-moment-probability}, $|\alpha_1^\ast-\alpha_2^\ast|=O(1)$, and $\snorm{\alpha^\ast}_\infty=O(\sqrt{n})$, we easily deduce that
\[\mb{P}_{\vec{d^\ast}}[\cref{eq:degree-ineq}]\gtrsim\mb{P}_{Z\sim\mc{N}(0,1)}\Big[Z\ge\gamma\sqrt{2}+\frac{\alpha_{1,2}^\ast-(\alpha_1^\ast+\alpha_2^\ast)\sqrt{2}/2}{2\sqrt{n}}\Big]^{2n}.\]
Let us focus on the region $Q^\ast$ where additionally
\[|\alpha_{1,2}^\ast-(\alpha_1^\ast+\alpha_2^\ast)\sqrt{2}/2-(2\sqrt{2n})x^\ast|\le\sqrt{2},\]
where $x^\ast=\argmax_{x\in\mb{R}}F_1(x)$.

Change coordinates in \cref{eq:first-moment-rho-lower} similar to \cref{sub:first-upper}, with $x=(\alpha_{1,2}^\ast\sqrt{2}-\alpha_1^\ast-\alpha_2^\ast)/2$. Integrating with respect to the other coordinates, we find
\begin{align*}
\rho'&\gtrsim\int_{|x-(2\sqrt{n})x^\ast|\le1}\mb{P}_{Z\sim\mc{N}(0,1)}[Z\ge(\gamma+x/(2\sqrt{n}))\sqrt{2}]^{2n}e^{-\frac{1}{2}x^2}dx\\
&\gtrsim\sqrt{n}\int_{|x-x^\ast|\le1/(2\sqrt{n})}\mb{P}_{Z\sim\mc{N}(0,1)}[Z\ge(\gamma+x)\sqrt{2}]^{2n}e^{-2x^2n}dx\\
&\gtrsim\sqrt{n}\int_{|x-x^\ast|\le1/(2\sqrt{n})}\exp(2n(F_1(x)-\log 2))dx.
\end{align*}
Since $x=x^\ast$ maximizes the expression in the exponent of the integrand, we see that $2n(F_1(x)-\log 2)$ stays within a range of size $O(1)$ when $|x-x^\ast|\le1/(2\sqrt{n})$, hence
\[\rho'\gtrsim\sqrt{n}\cdot\frac{1}{\sqrt{n}}\cdot\exp(2n\sup_{x\in\mb{R}}(F_1(x)-\log 2)).\]

Combining this with \cref{eq:first-moment-linearity,eq:first-moment-atypical,eq:first-moment-transfer}, we obtain
\[\mb{E}X_\gamma\gtrsim 4^{-n}\binom{2n}{n}\exp(2n\sup_{\alpha\in\mb{R}}F_1(\alpha)),\]
completing the lower bound of \cref{lem:first-moment} and hence the proof.

\bibliographystyle{amsplain0.bst}
\bibliography{main.bib}

\appendix
\section{Binomial Distribution Computations}\label{app:binomial}
We will require a number of precise asymptotics associated to the binomial random variables.
\begin{lemma}\label{lem:binom-comp-init}
Fix $C\ge 1$ such that $\max\{|a_1|,|a_2|\}\le C$. Then for $n$ large,
\begin{align*}
\mb{P}[&\mr{Bin}(n,1/2+a_2/(2\sqrt{n}))-\mr{Bin}(n-\ell,1/2+a_1/(2\sqrt{n})) = t] \\
&=\frac{1}{\sqrt{\pi n}}\exp\bigg(-\bigg(\frac{a_2-a_1}{2}+\frac{-t+\ell/2}{\sqrt{n}}\bigg)^2 + O_{C,|\ell|}\bigg(\frac{1}{n}+\frac{t^4}{n^3}\bigg)\bigg)\pm\exp(-\Omega((\log n)^2)).
\end{align*}
\end{lemma}
\begin{proof}
If $|t|\ge\sqrt{n}\log n/4$ then the result easily follows by the Azuma--Hoeffding inequality so assume $|t|<\sqrt{n}\log n/4$. Consider the range of values $k\in[0,n]$ and let $\tau=k-n/2$. Note that 
\begin{align*}
\mb{P}[&\mr{Bin}(n,1/2+a_2/(2\sqrt{n}))-\mr{Bin}(n-\ell,1/2+a_1/(2\sqrt{n})) = t]\\
&=\sum_{|k-n/2|\le\sqrt{n}\log n}\binom{n}{k}\binom{n-\ell}{k-t}\bigg(\frac{1}{2}+\frac{a_2}{2\sqrt{n}}\bigg)^{k}\bigg(\frac{1}{2}-\frac{a_2}{2\sqrt{n}}\bigg)^{n-k}\bigg(\frac{1}{2}+\frac{a_1}{2\sqrt{n}}\bigg)^{k-t}\bigg(\frac{1}{2}-\frac{a_1}{2\sqrt{n}}\bigg)^{n-\ell-k+t}\\
&\qquad \qquad \pm \exp(-\Omega((\log n)^2))\\
&=\sum_{|k-n/2|\le\sqrt{n}\log n}\binom{n}{n/2+\tau}\binom{n-\ell}{n/2-t+\tau}\bigg(\frac{1}{4}-\frac{a_2^2}{4n}\bigg)^{n/2}\bigg(\frac{1/2 + a_2/(2\sqrt{n})}{1/2-a_2/(2\sqrt{n})}\bigg)^{\tau}\\
&\qquad\qquad\bigg(\frac{1}{4}-\frac{a_1^2}{4n}\bigg)^{(n-\ell)/2}\bigg(\frac{1/2 + a_1/(2\sqrt{n})}{1/2-a_1/(2\sqrt{n})}\bigg)^{\tau+\ell/2-t}\pm \exp(-\Omega((\log n)^2))\\
&=\frac{2}{\pi n}\sum_{|k-n/2|\le\sqrt{n}\log n}\exp\bigg(-\frac{a_1^2}{2}-\frac{a_2^2}{2}+\frac{2a_2\tau}{\sqrt{n}}+\frac{2a_1(\tau+\ell/2-t)}{\sqrt{n}}-\frac{2\tau^2}{n}-\frac{2(\tau-t+\ell/2)^2}{n}\bigg)\\
&\qquad \cdot \bigg(1+O_{C,|\ell|}\bigg(\frac{1}{n} + \frac{\tau^4+t^4}{n^3}\bigg)\bigg)\pm \exp(-\Omega((\log n)^2))\\
&=\frac{e^{-\frac{1}{4n}(\ell-a_1\sqrt{n}+a_2\sqrt{n}-2t)^2}}{\sqrt{\pi n}}\sum_{|k-n/2|\le\sqrt{n}\log n}\frac{1}{\sqrt{2\pi(n/8)}}\exp\bigg(-\frac{4}{n}\bigg(\tau+\frac{\ell-a_1\sqrt{n}-a_2\sqrt{n}-2t}{4}\bigg)^2\bigg)\\
&\qquad \cdot \bigg(1+O_{C,|\ell|}\bigg(\frac{1}{n} + \frac{\tau^4+t^4}{n^3}\bigg)\bigg)\pm \exp(-\Omega((\log n)^2))\\
&= \frac{1}{\sqrt{\pi n}}\exp\bigg(-\bigg(\frac{a_2-a_1}{2}+\frac{-t+\ell/2}{\sqrt{n}}\bigg)^2 + O_{C,|\ell|}\bigg(\frac{1}{n}+\frac{t^4}{n^3}\bigg)\bigg)\pm\exp(-\Omega((\log n)^2))
\end{align*}
where we have used Stirling's formula and that $\exp(2t)-(1+t)/(1-t) = O(t^3)$ for $|t|\le 1/2$ and the approximation of sum via a Riemann integral.
\end{proof}

We will now derive a number of further estimates which ultimately are based on \cref{lem:binom-comp-init}.

\begin{lemma}\label{lem:first-moment-estimate}
Fix $C\ge 1$ such that $\max\{|a_1|,|a_2|\}\le C$. Then 
\[\mb{P}\bigg[\mr{Bin}\bigg(n-1,\frac{1}{2}+\frac{a_1}{2\sqrt{n}}\bigg)- \mr{Bin}\bigg(n,\frac{1}{2}+\frac{a_2}{2\sqrt{n}}\bigg)\ge k\bigg] = \mb{P}_{Z\sim \mc{N}(0,1)}\bigg[Z\ge \frac{k\sqrt{2}}{\sqrt{n}} + \frac{a_2-a_1}{\sqrt{2}}\bigg]+ O_C\bigg(\frac{1}{n}\bigg).\]
\end{lemma}
\begin{proof}
Applying \cref{lem:binom-comp-init} we have that 
\begin{align*}
\mb{P}&\bigg[\mr{Bin}\bigg(n-1,\frac{1}{2}+\frac{a_1}{2\sqrt{n}}\bigg)- \mr{Bin}\bigg(n,\frac{1}{2}+\frac{a_2}{2\sqrt{n}}\bigg)\ge k\bigg] \\
 &= \mb{P}\bigg[\mr{Bin}\bigg(n,\frac{1}{2}+\frac{a_2}{2\sqrt{n}}\bigg)- \mr{Bin}\bigg(n-1,\frac{1}{2}+\frac{a_1}{2\sqrt{n}}\bigg)\le -k\bigg] \\
 &=\sum_{-\sqrt{n}\log n\le t\le -k}\frac{1}{\sqrt{\pi n}}\exp\bigg(-\bigg(\frac{a_2-a_1}{2}+\frac{-t+1/2}{\sqrt{n}}\bigg)^2 + O_{C}\bigg(\frac{1}{n}+\frac{|t|^4}{n^3}\bigg)\bigg) \pm \exp(-\Omega((\log n)^2))\\
  &=\sum_{-\sqrt{n}\log n\le t\le -k}\frac{1}{\sqrt{\pi n}}\exp\bigg(-\bigg(\frac{a_2-a_1}{2}+\frac{-t+1/2}{\sqrt{n}}\bigg)^2 \bigg)+ O_C\bigg(\frac{1}{n}\bigg)\\
  &=\int_{-\infty}^{-k+1/2}\frac{1}{\sqrt{\pi n}}\exp\bigg(-\bigg(\frac{a_2-a_1}{2}+\frac{-t+1/2}{\sqrt{n}}\bigg)^2 \bigg)dt + O_C\bigg(\frac{1}{n}\bigg)\\
  &=\int_{-\infty}^{-k}\frac{1}{\sqrt{\pi n}}\exp\bigg(-\bigg(\frac{a_1-a_2}{2}+\frac{t}{\sqrt{n}}\bigg)^2 \bigg)dt + O_C\bigg(\frac{1}{n}\bigg)\\
  &=\int_{-\infty}^{\frac{-k\sqrt{2}}{\sqrt{n}}}\frac{1}{\sqrt{2\pi}}\exp\bigg(-\bigg(\frac{a_1-a_2}{2}+\frac{z}{\sqrt{2}}\bigg)^2 \bigg)dz + O_C\bigg(\frac{1}{n}\bigg)\\
  &= \mb{P}_{Z\sim \mc{N}(0,1)}\bigg[Z\ge \frac{k\sqrt{2}}{\sqrt{n}} + \frac{a_2-a_1}{\sqrt{2}}\bigg] + O_{C}\bigg(\frac{1}{n}\bigg)
\end{align*}
where to replace the sum by the corresponding integral we have used midpoint rule for approximating integrals (which in turn easily follows from the Euler--Maclaurin formula).
\end{proof}

For the second estimate, we define notation for ease of use.
\begin{definition}\label{def:binomial-gaussian-probability-2}
Given $\gamma\in[-1,1]$ and $\beta\in(0,1)$ and $a_1,a_2,a_3,a_4\in\mb{R}$, we define
\begin{align*}
g_\gamma(\beta,a_1,a_2,a_3,a_4)&=\mb{P}_{Z_i\sim \mc{N}(0,1)}\bigg[\sqrt{\frac{\beta}{2}}Z_1 + \sqrt{\frac{1-\beta}{2}}Z_2\ge \gamma + \frac{(a_3-a_1)\beta+(a_4-a_2)(1-\beta)}{2}\\
&\qquad\qquad\qquad\quad\wedge \sqrt{\frac{\beta}{2}}Z_1 - \sqrt{\frac{1-\beta}{2}}Z_2\ge\gamma + \frac{(a_3-a_1)\beta-(a_4-a_2)(1-\beta)}{2}\bigg].
\end{align*}
We will often suppress the $\gamma$ in the notation. We see $g(\beta,a_1,0,0,0)=f(\beta,-a_1\beta/2)$ with notation as in \cref{def:special-func}.
\end{definition}
\begin{lemma}\label{lem:second-moment-estimate}
Fix $C\ge 1$. Let $|a_i|\le C$ for $i\in [4]$ and $k = \beta n$ where $\min(\beta,1-\beta)\ge1/C$. Furthermore let
$X_1\sim \mr{Bin}\big(k-1,\frac{1}{2}+\frac{a_1}{2\sqrt{n}}\big)$, $X_2\sim \mr{Bin}\big(n-k,\frac{1}{2}+\frac{a_2}{2\sqrt{n}}\big)$, $X_3\sim \mr{Bin}\big(k,\frac{1}{2}+\frac{a_3}{2\sqrt{n}}\big)$, and $X_4\sim \mr{Bin}\big(n-k,\frac{1}{2}+\frac{a_4}{2\sqrt{n}}\big)$. Finally let $\Gamma$ be an integer and $\gamma=\Gamma/\sqrt{n}$. Then
\begin{align*}
&\mb{P}[X_1-X_3+X_2-X_4\ge \Gamma \wedge X_1-X_3-X_2+X_4\ge\Gamma]\\
&=g_\gamma(\beta,a_1,a_2,a_3,a_4) + O_C\bigg(\frac{1}{n}\bigg).
\end{align*}
\end{lemma}
\begin{proof}
Let $T_1 = X_1-X_3$ and $T_2 = X_2 - X_4$. By \cref{lem:binom-comp-init} we have
\begin{align*}
\mb{P}&[X_1-X_3+X_2-X_4\ge \Gamma \wedge X_1-X_3-X_2+X_4\ge\Gamma] = \sum_{t_1 \pm t_2\ge \Gamma}\mb{P}[T_1 = t_1] \mb{P}[T_2 = t_2]\\
& = \sum_{\Gamma\le t_1 \pm t_2\le \sqrt{n}\log n}\frac{1}{\pi \sqrt{k(n-k)}}\exp\bigg(-\bigg(\frac{(a_3-a_1)\sqrt{k}}{2\sqrt{n}}+\frac{t_1+1/2}{\sqrt{k}}\bigg)^2 - \bigg(\frac{(a_4-a_2)\sqrt{n-k}}{2\sqrt{n}}+\frac{t_2}{\sqrt{n-k}}\bigg)^2 \\
&\qquad + O_{C}\bigg(\frac{1}{n} + \frac{t_1^4 + t_2^4}{n}\bigg)\bigg) + \exp(-\Omega((\log n)^2))\\
& = \sum_{\Gamma\le t_1 \pm t_2\le \sqrt{n}\log n}\frac{1}{\pi \sqrt{k(n-k)}}\exp\bigg(-\bigg(\frac{(a_3-a_1)\sqrt{k}}{2\sqrt{n}}+\frac{t_1+1/2}{\sqrt{k}}\bigg)^2 - \bigg(\frac{(a_4-a_2)\sqrt{n-k}}{2\sqrt{n}}+\frac{t_2}{\sqrt{n-k}}\bigg)^2\bigg) \\
&+ O_{C}\bigg(\frac{1}{n}\bigg)\\
& = \sum_{\Gamma\le t_1 \pm t_2\le \sqrt{n}\log n}\int_{t_1-1/2}^{t_1+1/2}\int_{t_2-1/2}^{t_2+1/2}\frac{1}{\pi \sqrt{k(n-k)}}\exp\bigg(-\bigg(\frac{(a_3-a_1)\sqrt{k}}{2\sqrt{n}}+\frac{\ell + 1/2}{\sqrt{k}}\bigg)^2 \\
&\qquad- \bigg(\frac{(a_4-a_2)\sqrt{n-k}}{2\sqrt{n}}+\frac{m}{\sqrt{n-k}}\bigg)^2\bigg)dmd\ell+ O_{C}\bigg(\frac{1}{n}\bigg)\\
& = \sum_{\Gamma\le t_1 \pm t_2\le \sqrt{n}\log n}\int_{t_1}^{t_1+1}\int_{t_2-1/2}^{t_2+1/2}\frac{1}{\pi \sqrt{k(n-k)}}\exp\bigg(-\bigg(\frac{(a_3-a_1)\sqrt{k}}{2\sqrt{n}}+\frac{\ell }{\sqrt{k}}\bigg)^2 \\
&\qquad- \bigg(\frac{(a_4-a_2)\sqrt{n-k}}{2\sqrt{n}}+\frac{m}{\sqrt{n-k}}\bigg)^2\bigg)dmd\ell+ O_{C}\bigg(\frac{1}{n}\bigg)\\
& = \sum_{\Gamma\le t_1 \pm t_2\le \sqrt{n}\log n}\int_{t_1/\sqrt{k}}^{(t_1+1)/\sqrt{k}}\int_{(t_2-1/2)/\sqrt{n-k}}^{(t_2+1/2)/\sqrt{n-k}}\frac{1}{\pi}\exp\bigg(-\bigg(\frac{(a_3-a_1)\sqrt{\beta}}{2}+\ell\bigg)^2 \\
&\qquad- \bigg(\frac{(a_4-a_2)\sqrt{1-\beta}}{2}+m\bigg)^2\bigg)dmd\ell+ O_{C}\bigg(\frac{1}{n}\bigg)\\
& = \sum_{\Gamma\le t_1 \pm t_2\le \infty}\int_{t_1/\sqrt{k}}^{(t_1+1)/\sqrt{k}}\int_{(t_2-1/2)/\sqrt{n-k}}^{(t_2+1/2)/\sqrt{n-k}}\frac{1}{\pi}\exp\bigg(-\bigg(\frac{(a_3-a_1)\sqrt{\beta}}{2}+\ell\bigg)^2 \\
&\qquad- \bigg(\frac{(a_4-a_2)\sqrt{1-\beta}}{2}+m\bigg)^2\bigg)dmd\ell+ O_{C}\bigg(\frac{1}{n}\bigg)\\
& = \int_{\Gamma\le\ell\sqrt{k} \pm m\sqrt{n-k}}\frac{1}{\pi}\exp\bigg(-\bigg(\frac{(a_3-a_1)\sqrt{\beta}}{2}+\ell\bigg)^2 - \bigg(\frac{(a_4-a_2)\sqrt{1-\beta}}{2}+m\bigg)^2\bigg)\bigg)d\ell dm+ O_{C}\bigg(\frac{1}{n}\bigg)\\
&=\mb{P}_{Z_i\sim \mc{N}(0,1)}\bigg[\sqrt{\frac{\beta}{2}}Z_1 + \sqrt{\frac{1-\beta}{2}}Z_2\ge \frac{\Gamma}{\sqrt{n}} + \frac{(a_3-a_1)\beta+(a_4-a_2)(1-\beta)}{2}\\
&\qquad \wedge \sqrt{\frac{\beta}{2}}Z_1 - \sqrt{\frac{1-\beta}{2}}Z_2\ge\frac{\Gamma}{\sqrt{n}} + \frac{(a_3-a_1)\beta-(a_4-a_2)(1-\beta)}{2}\bigg] + O_C\bigg(\frac{1}{n}\bigg).
\end{align*}
In the second-to-last line, one can see the necessary equality by considering the error terms at the line $\Gamma=\ell\sqrt{k}+m\sqrt{n-k}$ versus the given boxes (similar to the Euler--Maclaurin transference between sum and integral in the proof of \cref{lem:first-moment-estimate}).
\end{proof}

\section{Computer Assisted Verification}\label{app:computer-assist}
We now proceed via a delicate computer assisted verification in order to prove \cref{asm:comp}. We state a series of numerical claims which will be used to prove \cref{asm:comp}. Notice by symmetry of the variational problem that we may assume that $\beta \in [0,1/2]$. The first claim will handle the most numerically unstable part of the claim when $\beta$ is contained in the initial segment $\beta \in [0,.001]$.

\begin{claim}\label{clm:initial-interval}
Let $\gamma \in [\gamma_{\mr{crit}}-\eps_{\ref{asm:comp}}, \gamma_{\mr{crit}} + \eps_{\ref{asm:comp}}]$. Then 
\[\sup_{\substack{\beta\in[0,.001]\\\alpha_1, \alpha_2\in \mb{R}}} F_2(\beta,\alpha_1,\alpha_2) = 2\sup_{\alpha\in \mb{R}}F_1(\alpha).\]
\end{claim}

We next control the case when $\beta \in [.001,.495]$. \begin{claim}\label{clm:middle-segment}
Let $\gamma \in [\gamma_{\mr{crit}}-\eps_{\ref{asm:comp}}, \gamma_{\mr{crit}} + \eps_{\ref{asm:comp}}]$. Then 
\[\sup_{\substack{\beta\in[.001,.495]\\\alpha_1, \alpha_2\in \mb{R}}} F_2(\beta,\alpha_1,\alpha_2) < - 10^{-5}.\]
\end{claim}

We next localize the region of interest to $\beta\in [.495,.5]$ and $\alpha_i\in [-.449,-.441]$.
\begin{claim}\label{clm:middle-segment-local}
Let $\gamma \in [\gamma_{\mr{crit}}-\eps_{\ref{asm:comp}}, \gamma_{\mr{crit}} + \eps_{\ref{asm:comp}}]$. Then 
\[\sup_{\substack{\beta\in[.495,.5]\\(\alpha_1, \alpha_2)\notin [-.449,-.441]^2}} F_2(\beta,\alpha_1,\alpha_2) < 10^{-6}.\]
\end{claim}

Finally in this local region we check that the associated Hessian is strictly negative definite.
\begin{claim}\label{clm:hessian-verification}
Let $\gamma \in [\gamma_{\mr{crit}}-\eps_{\ref{asm:comp}}, \gamma_{\mr{crit}} + \eps_{\ref{asm:comp}}]$. There exists an absolute constant $\delta > 0$ such that $F_2(\beta,\alpha_1,\alpha_2)\colon\mb{R}^3\to \mb{R}$ satisfies 
\[\begin{pmatrix}
\frac{\partial^2 F_2}{\partial \alpha_1^2} & \frac{\partial^2 F_2}{\partial \alpha_1\partial \alpha_2} & \frac{\partial^2 F_2}{\partial \alpha_1\partial \beta}\\
\frac{\partial^2 F_2}{\partial \alpha_1\partial \alpha_2} & \frac{\partial^2 F_2}{\partial \alpha_2^2} & \frac{\partial^2 F_2}{\partial \alpha_2\partial \beta}\\
\frac{\partial^2 F_2}{\partial \alpha_1\partial \beta} & \frac{\partial^2 F_2}{\partial \alpha_2\partial \beta} & \frac{\partial^2 F_2}{\partial \beta^2}\\
\end{pmatrix}\preceq -\delta I_3\]
for $\beta\in [.495,.505]$ and $\alpha_i\in [-.449,-.441]$, where $I_3$ is the $3\times 3$ identity matrix and $\preceq$ denotes the semidefinite order.
\end{claim}

We now deduce the result given the claims outlined above.
\begin{proof}[Proof of \cref{asm:comp}]
We first notice that 
\[\sup_{\substack{\beta\in[0,1]\\\alpha_1, \alpha_2\in \mb{R}}} F_2(\beta,\alpha_1,\alpha_2) \ge \max(4\sup_{\alpha\in\mb{R}}F_1(\alpha), 2\sup_{\alpha\in \mb{R}}F_1(\alpha)).\] 
This follows by evaluating $F_2(\beta,\alpha_1,\alpha_2)$ at $(\beta,\alpha_1,\alpha_2)$ at $(1/2, \alpha(\gamma), \alpha(\gamma))$ and $(0, -1, \alpha(\gamma))$ respectively. Furthermore note that for $\gamma \in [\gamma_{\mr{crit}}-\eps_{\ref{asm:comp}}, \gamma_{\mr{crit}} + \eps_{\ref{asm:comp}}]$, we have 
\begin{align*}
\log 2+\sup_{\alpha\in\mb{R}}-\alpha^2 + \mb{P}_{Z\sim{\mc{N}(0,1)}}[Z\ge(\gamma + \alpha)\sqrt{2}]&= \log 2+\sup_{\alpha\in\mb{R}}-\alpha^2 + \mb{P}_{Z\sim{\mc{N}(0,1)}}[Z\ge(\gamma_{\mr{crit}} + \alpha)\sqrt{2}] \pm \eps_{\ref{asm:comp}}\\
&\in[-\eps_{\ref{asm:comp}},\eps_{\ref{asm:comp}}]
\end{align*}
where we have used that the derivative of $x\to \mb{P}[Z\ge x]$ is bounded by $1/\sqrt{2\pi}$ in magnitude. So $\sup_{\alpha\in\mb{R}}F_1(\alpha)\in[-\eps_{\ref{asm:comp}},\eps_{\ref{asm:comp}}]$ for $\gamma$ in this range.

Therefore by \cref{clm:middle-segment} and $\eps_{\ref{asm:comp}}=10^{-25}$ it suffices to consider when $\beta\in [0,.001]$ or when $\beta \in [.495,.5]$. For the former case note that the result follows immediately from \cref{clm:initial-interval}. For the latter case note that it suffices to check $\beta \in [.495,.5]$ and $\alpha_1,\alpha_2\in [-.449,-.441]$ by \cref{clm:middle-segment-local}.  Note that this also implies for $\gamma \in [\gamma_{\mr{crit}}-\eps_{\ref{asm:comp}}, \gamma_{\mr{crit}} + \eps_{\ref{asm:comp}}]$ that $\alpha(\gamma)\in [-.449,-.441]$. 

Finally note that $F_2$ is symmetric under the transformation $F_2(\beta,\alpha_1,\alpha_2) = F_2(1-\beta,\alpha_2,\alpha_1)$. This implies that the gradient of $F_2$ vanishes at $(1/2, \alpha(\gamma), \alpha(\gamma))$. By \cref{clm:hessian-verification} we know that $F_2$ is concave on this region, so the desired result follows immediately. 
\end{proof}

We now proceed with the proof of each individual claim in the remaining subsections. The proof of each of these claims is computer-assisted, but the vast majority of the computational effort occurs in \cref{clm:middle-segment} and \cref{clm:middle-segment-local}.

\subsection{Properties of \texorpdfstring{$f(\beta,\alpha)$}{f(beta,alpha)}}
We collect a series of general properties of $f(\beta,\alpha)$ which will be used throughout the verification procedure. The first is an explicit formula for $f(\beta,\alpha)$ which serves to make it more amenable to direct computation.
\begin{claim}\label{clm:gaussian-rot}
We have that 
\[f(\beta,\alpha) = \int_{\gamma+\alpha}^\infty\int_{\gamma+\alpha}^\infty\frac{1}{2\pi\sqrt{\beta(1-\beta)}}\exp\bigg(-\frac{1}{4\beta(1-\beta)}(t_1^2+t_2^2)-\frac{1/2-\beta}{\beta(1-\beta)}t_1t_2\bigg)dt_1dt_2\]
\end{claim}
\begin{proof}
This is an immediate consequence of the fact that $T_1 = \sqrt{\beta/2}Z_1+\sqrt{(1-\beta)/2}Z_2$ and $T_2 = \sqrt{\beta/2}Z_1-\sqrt{(1-\beta)/2}Z_2$ are jointly Gaussian each having variance $1/2$ and with covariance $\beta-1/2$ between the two. In particular, $T_1,T_2$ have covariance matrix $\Sigma$ where
\[\Sigma = \begin{bmatrix}1/2&\beta-1/2\\\beta-1/2&1/2\end{bmatrix},\quad\Sigma^{-1} = \frac{1}{\beta(1-\beta)}\begin{bmatrix}1/2&1/2-\beta\\1/2-\beta&1/2\end{bmatrix}\]
and thus
\[f(\beta,\alpha) = \int_{\gamma+\alpha}^\infty\int_{\gamma+\alpha}^\infty\frac{1}{2\pi\sqrt{\beta(1-\beta)}}\exp\bigg(-\frac{1}{4\beta(1-\beta)}(t_1^2+t_2^2)-\frac{1/2-\beta}{\beta(1-\beta)}t_1t_2\bigg)dt_1dt_2.\qedhere\]
\end{proof}

We will also require the following substantially more efficient and more numerically stable version of \cref{clm:gaussian-rot}. The gain in efficiency stems from the fact that two-dimensional integrals are substantially more difficult to compute and less accurate than their one-dimensional counterparts.

\begin{claim}\label{clm:numer-stable}
We have that 
\[f(\beta,\alpha) = \frac{1}{2\pi}\int_{0}^{\rho}(1-x^2)^{-1/2}\exp\bigg(\frac{-2(\gamma+\alpha)^2}{1+x}\bigg)dx + \bigg(\int_{-\infty}^{-\sqrt{2}(\gamma+\alpha)}\frac{e^{-x^2/2}}{\sqrt{2\pi}}dx\bigg)^2\]
where $\rho = 2\beta - 1$. 
\end{claim}
\begin{proof}
This is precisely \cite[(6)]{DW90} with $\rho=2\beta-1$ and $h=k=(\gamma+\alpha)\sqrt{2}$ (note that the Gaussians $T_1,T_2$ in the proof of \cref{clm:gaussian-rot} have variance $1/2$, not $1$, hence the normalization).
\end{proof}

We note that when computing $f(\beta,\alpha)$ and the associated derivatives (see \cref{clm:derivative}) in the rigorous \texttt{python-flint}, we must truncate various integral to a large finite region. We truncate such integrals, treating replacing $\infty$ by a cutoff amount ($10^{8}$) and straightforward bounds verify that this truncation is easily absorbed in the reported error bounds. 

A crucial portion of our analysis will rely on the following upper and lower bounds for $f(\beta,\alpha)$ when $\beta$ and $\alpha$ are restricted in particular intervals. These will allow us to bound $F_2(\beta,\alpha_1,\alpha_2)$ by a uniform function which handles all $\beta$ in a specified interval at once. This is crucial as the associated envelope function on the interval will furthermore be convex in $\alpha_1,\alpha_2$ and therefore finding the corresponding maximum will be amenable to a fixed-point iteration. We return to the precise form of this iteration in the subsequent subsections. In the statement below, the probabilistic event $X\pm Y\ge t$ will serve as shorthand notation for $X+Y\ge t \wedge X-Y\ge t$.

\begin{lemma}\label{lem:upper-envelope}
Fix $\gamma$ and suppose that $\beta\in [\eta_1,\eta_2]\subseteq[0,1]$. If $-\gamma\le\alpha\le 0$ then
\[f(\eta_1,\alpha)\le f(\beta,\alpha)\le f(\eta_2,\alpha).\]
Else if $\alpha\le -\gamma$ then
\[\mb{P}\bigg[\sqrt{\frac{\eta_1}{2}}Z_1\pm \sqrt{\frac{1-\eta_1}{2}}Z_2\ge\sqrt{\frac{\eta_1}{\eta_2}}(\gamma+\alpha)\bigg]\le f(\beta,\alpha)\le\mb{P}\bigg[\sqrt{\frac{\eta_2}{2}}Z_1\pm \sqrt{\frac{1-\eta_2}{2}}Z_2\ge\sqrt{\frac{\eta_2}{\eta_1}}(\gamma+\alpha)\bigg].\]
\end{lemma}
\begin{proof}
We consider the upper bound for the second inequality (the case $\alpha\le-\gamma$). Note that $f(\beta,\alpha)$ corresponds to integrating the standard bivariate normal density between the lines $\sqrt{\frac{\beta}{2}}z_1+\sqrt{\frac{1-\beta}{2}}z_2\ge\gamma+\alpha$ and $\sqrt{\frac{\beta}{2}}z_1-\sqrt{\frac{1-\beta}{2}}z_2\ge\gamma+\alpha$. Note that these lines intersect at the point $\big(\sqrt{\frac{2}{\beta}}(\gamma + \alpha), 0\big)$ and the lines emanating from this point have slopes $\sqrt{\frac{\beta}{1-\beta}}$ and $-\sqrt{\frac{\beta}{1-\beta}}$. Note that shifting $\beta$ to $\eta_2$ and the intersection point of the lines to $\big(\sqrt{\frac{2}{\eta_1}}(\gamma + \alpha), 0\big)$ gives a strictly larger region and corresponds to the upper bound. The remaining three inequalities are obtained in an identical manner.
\end{proof}

We will also require the following explicit formula for the partial derivative of $f(\beta,\alpha)$ in $\alpha$.
\begin{claim}\label{clm:derivative}
We have that
\[\partial_af(\beta,a) = -\frac{2}{\sqrt{\pi}}\exp(-(\gamma+a)^2)\mb{P}[\sqrt{2\beta(1-\beta)}Z\ge(2-2\beta)(\gamma+a)].\]
\end{claim}
\begin{proof}
By symmetry we compute that
\begin{align*}
\partial_af(\beta,a) &= -2\int_{\gamma+a}^\infty\frac{1}{2\pi\sqrt{\beta(1-\beta)}}\exp\bigg(-\frac{1}{4\beta(1-\beta)}(t_1^2+(\gamma+a)^2)-\frac{1/2-\beta}{\beta(1-\beta)}(\gamma+a)t_1\bigg)dt_1\\
&= -\frac{2}{\sqrt{\pi}}\exp(-(\gamma+a)^2)\int_{\gamma+a}^\infty\frac{1}{\sqrt{4\pi\beta(1-\beta)}}\exp\bigg(-\frac{(t+(1-2\beta)(\gamma+a))^2}{4\beta(1-\beta)}\bigg)dt\\
&= -\frac{2}{\sqrt{\pi}}\exp(-(\gamma+a)^2)\mb{P}[\sqrt{2\beta(1-\beta)}Z\ge(2-2\beta)(\gamma+a)]. \qedhere
\end{align*}
\end{proof}

Finally, we will repeatedly require the following elementary estimate that a point with small derivative can be used, when combined with an \emph{a priori} second derivative estimate and a gradient bound, to derive an upper bound on the function.

\begin{lemma}\label{lem:concave-bound-trick}
Let $f\colon\mb{R}\to\mb{R}$ such that $\frac{\partial^2}{\partial x^2}f(x)\le -M<0$. Then
\[\sup_{x\in \mb{R}}f(x)\le \inf_{z\in \mb{R}} f(z) + \frac{f'(z)^2}{2M}.\]
\end{lemma}
\begin{proof}
Given any $z\in\mb{R}$, a straightforward application of Taylor's theorem shows that
\[f(x)\le f(z)+f'(z)(x-z)-\frac{M}{2}(x-z)^2.\]
The quadratic on the right is maximized when $x=z-f'(z)/M$, which implies $f(x)\le f(z)+f'(z)^2/(2M)$. Now taking a supremum over $x$ and infimum over $z$ finishes.
\end{proof}

An identical bivariate version of the above claim follows from the one-variable version, considering a line between the points $x,z$ of interest.
\begin{lemma}\label{lem:concave-bound-trick-2}
Let $f\colon\mb{R}^2\to\mb{R}$ such that $\nabla^2 f(x)\preceq -MI_2\prec 0$, where $I_2$ is the $2\times 2$ identity matrix and $\preceq$ denotes the semidefinite order. Then
\[\sup_{\vec{x}\in \mb{R}^2}f(x)\le \inf_{\vec{z}\in \mb{R}^2} f(\vec{z}) + \frac{\snorm{\nabla f(\vec{z})}_2^2}{2M}.\]
\end{lemma}

\subsection{Bounds on \texorpdfstring{$\gamma_{\mr{crit}}$}{gamma crit}}
In order to proceed with our analysis we will require sufficiently precise bounds on the value of $\gamma_{\mr{crit}}$.

\begin{claim}\label{clm:numeric-bounds}
We have 
\[.24841951\le \gamma_{\mr{crit}}\le .24841959.\]
\end{claim}
\begin{proof}
Notice that $F_1(\alpha)$ satisfies $\frac{\partial^2}{\partial \alpha^2}F_1(\alpha) \le -2$ due to \cref{lem:log-concave-1}.

Letting $\gamma = .24841951$ and $\alpha = -0.445183267$ we have $F_1(\alpha) \ge 4 \cdot 10^{-8}$ by numerical computation and this provides the necessary lower bound on $\gamma_{\mr{crit}}$.

For the upper bound we set $\gamma = .24841959$ and $\alpha = -0.44518333$. We have $F_1(\alpha)\le -2 \cdot 10^{-8}$ and $|F_1'(\alpha)|\le 10^{-5}$ by numerical computation. This implies the desired result via \cref{lem:concave-bound-trick}.
\end{proof}

The most important technical point for \cref{clm:numeric-bounds} is hidden under the hood of producing the candidate values for \cref{clm:numeric-bounds}. In theory these bounds could be produced via a repeated bisection procedure, but a more tailored optimization approach exists which relies on the structure of $F_1$ and $F_2$ in a precise manner. The key idea is to rely on a fixed-point iteration where the fixed point corresponds to a critical point of our desired function. In particular note that 
\[F_1'(\alpha) = -2\alpha - \frac{\exp(-(\gamma+\alpha)^2)}{\sqrt{\pi}\mb{P}[Z\ge (\gamma + \alpha)\sqrt{2}]}.\]
Rearranging the equation $F_1'(\alpha) = 0$ we are naturally led to an iterative procedure with
\[\alpha_{i+1} = \frac{-\exp(-(\gamma+\alpha_i)^2)}{2\sqrt{\pi}\mb{P}[Z\ge (\gamma + \alpha_{i})\sqrt{2}]}.\]
Starting with even very crude estimates for $\alpha_0$ the above procedure converges to a numerical fixed point extraordinarily quickly (being that this is in essence a proxy for Newton iteration tailored to this problem).

Furthermore the separable nature of $F_2(\beta,\alpha_1,\alpha_2)$ with respect to the variables $\alpha_1,\alpha_2$ allows an essentially identical procedure to produce upper bound on $F_2(\beta,\alpha_1,\alpha_2)$ for fixed $\beta$. The crucial issue for \cref{lem:upper-envelope} therefore is that the functional form we use is chosen to bound an interval of $\beta$ values uniformly, while still allowing for a fixed point iteration of the form specified to succeed.

\subsection{Proof of \texorpdfstring{\cref{clm:initial-interval}}{Claim B.1}}
For the proof of \cref{clm:initial-interval}, the key idea is that for fixed $\alpha_1,\alpha_2$ the derivative in $\beta$ of $F_2(\beta,\alpha_1,\alpha_2)$ is $-\infty$ at $\beta=0$. This is a  manifestation of the fact that the model exhibits a frozen 1-RSB structure and the proof is a sufficiently quantitative version of this fact.
\begin{proof}[Proof of \cref{clm:initial-interval}]
Notice that
\[F_2(\beta,\alpha_1,\alpha_2)\le 2\log 2 - 2\beta\log\beta - 2(1-\beta)\log(1-\beta) - 2\alpha_2^2+ 2(1-\beta)\log f(1-\beta,\alpha_2).\]
So, letting
\[G_2(\beta,\alpha_2) = 2\log 2 - 2\beta\log\beta - 2(1-\beta)\log(1-\beta) - 2\alpha_2^2+ 2(1-\beta)\log f(1-\beta,\alpha_2),\]
we will prove the stronger claim that 
\[\sup_{\substack{\beta\in[0,.001]\\\alpha_2\in\mb{R}}} G_2(\beta,\alpha_2) = 2\sup_{\alpha\in \mb{R}}F_1(\alpha).\]
We first prove that it suffices to restrict attention to the case where $-.53\le \alpha_2\le -.37$. By numerical computation and $\beta\in[0,.001]$, we find
\[G_2(\beta,\alpha_2) \le 2\log 2 +.01582 + 2(1-.001)(-\alpha_2^2 + \log\mb{P}[Z\ge (\gamma + \alpha_2)\sqrt{2}]).\]
Note that this upper bound, call it $H(\alpha_2)$, is a concave function of $\alpha_2$ by \cref{lem:log-concave-1}. We can check $H(\alpha_2)\le -10^{-3}$ for $\alpha_2\in\{-.53,-.37\}$ and $H(-.45)\ge 10^{-2}$ by numerical computation. Thus $H$ is increasing on $(-\infty,-.53]$ and decreasing on $[-.37,+\infty)$, so $H(\alpha_2)\le -10^{-3}$ for $\alpha_2\notin[-.53,-.37]$. From the proof of \cref{asm:comp} recall that $\sup_{\alpha\in\mb{R}}F_1(\alpha)\in[-\eps_{\ref{asm:comp}},\eps_{\ref{asm:comp}}]$. Thus we immediately find that it suffices to consider $\alpha_2\in [-.53,-.37]$.

A trivial inequality and the second part of \cref{lem:upper-envelope} (recall \cref{clm:numeric-bounds}) yields
\[f(1-\beta,\alpha_2)\ge f(1-\beta,-.37)\ge \mb{P}\bigg[\sqrt{\frac{.999}{2}}Z_1 \pm \sqrt{\frac{.001}{2}}Z_2\ge\sqrt{.999}(\gamma-.37)\bigg]\ge .538\]
by numerical computation for the final inequality. We next note that 
\begin{align*}
\frac{\partial}{\partial \beta}G_2(\beta,\alpha_2) &=2\log(1-\beta) -2\log\beta - 2\log f(1-\beta,\alpha_2) + \frac{2(1-\beta)}{f(1-\beta,\alpha_2)}\frac{\partial}{\partial \beta}[f(1-\beta,\alpha_2)]\\
&\le 1.239 -2\log\beta + \frac{2(1-\beta)}{f(1-\beta,\alpha_2)}\frac{\partial}{\partial \beta}[f(1-\beta,\alpha_2)]
\end{align*}
by the previous inequality, numerical computation, and $\beta\in[0,.001]$.

Notice from \cref{clm:gaussian-rot} and differentiation under the integral sign that
\begin{align*}
\frac{\partial}{\partial \beta}[f(1-\beta,\alpha_2)] &= \int_{[\gamma + \alpha,\infty)^2}\frac{P(t_1,t_2,\beta)}{4(\beta(1-\beta))^2}\Bigg(\frac{\exp\big(-\frac{1}{4\beta(1-\beta)}(t_1^2+t_2^2)-\frac{\beta-1/2}{\beta(1-\beta)}t_1t_2\big)}{2\pi\sqrt{\beta(1-\beta)}}\Bigg)dt_1dt_2
\end{align*}
where 
\begin{align*}
P(t_1,t_2,\beta) &= 2(1-\beta)(\beta)(-1+2\beta) + (1-\beta)^2(t_1-t_2)^2-\beta^2(t_1+t_2)^2\\
&\le 2(1-\beta)\beta^2+(1-\beta)^2((t_1-t_2)^2-2\beta).
\end{align*}
This implies that 
\begin{align*}
\frac{\partial}{\partial \beta}[f(1-\beta,\alpha_2)] & \le \frac{f(1-\beta,\alpha_2)}{2(1-\beta)} + \int_{[\gamma + \alpha,\infty]^2}\frac{((t_1-t_2)^2-2\beta)}{4\beta^2}\Bigg(\frac{\exp\big(-\frac{1}{4\beta(1-\beta)}(t_1^2+t_2^2)-\frac{\beta-1/2}{\beta(1-\beta)}t_1t_2\big)}{2\pi\sqrt{\beta(1-\beta)}}\Bigg)dt_1dt_2.
\end{align*}
Performing an identical change of variables to that in \cref{clm:gaussian-rot} (in the reverse direction and with $\beta$ replaced by $1-\beta$), we find 
\begin{align*}
\frac{\partial}{\partial \beta}G_2(\beta,\alpha_2) &\le 1.239-2\log\beta + \frac{2(1-\beta)}{f(1-\beta,\alpha_2)}\frac{\partial}{\partial \beta}[f(1-\beta,\alpha_2)]\\
&\le 2.239 -2\log\beta + \frac{(1-\beta)}{\beta}\frac{\mb{E}\Big[(Z_2^2-1)\mbm{1}\Big[\sqrt{\frac{1-\beta}{2}}Z_1\pm \sqrt{\frac{\beta}{2}}Z_2\ge \gamma + \alpha_2\Big]\Big]}{f(1-\beta,\alpha_2)}.
\end{align*}
Notice that a trivial inequality and \cref{lem:upper-envelope} yield
\[f(1-\beta,\alpha_2)\le f(1-\beta,-.53)\le \mb{P}\bigg[Z_1 \ge \frac{1}{\sqrt{.999}}(\gamma-.53)\bigg]\le 0.611\]
by numerical computation.

Now we find $\mb{E}\Big[(Z_2^2-1)\mbm{1}\Big[\sqrt{\frac{1-\beta}{2}}Z_1\pm \sqrt{\frac{\beta}{2}}Z_2\ge \gamma + \alpha_2\Big]\Big]\le 0$ since it is even true for any fixed value of $Z_1$ (conditioning a Gaussian to lie in a symmetric interval containing $0$ only decreases its variance). In fact, using the transformation $z_1=\sqrt{2/(1-\beta)}(\gamma+\alpha_2)+\sqrt{\beta/(1-\beta)}t$ we have the inequality
\begin{align*}
\mb{E}&\Big[(Z_2^2-1)\mbm{1}\Big[\sqrt{\frac{1-\beta}{2}}Z_1\pm \sqrt{\frac{\beta}{2}}Z_2\ge \gamma + \alpha_2\Big]\Big]\\
&\qquad=\sqrt{\frac{\beta}{1-\beta}}\int_{t\ge 0}\frac{1}{\sqrt{2\pi}}\exp\bigg(-\frac{1}{2}\bigg(\sqrt{\frac{2}{1-\beta}}(\gamma+\alpha_2)+t\sqrt{\frac{\beta}{1-\beta}}\bigg)^2\bigg)\int_{-t}^t\frac{(x^2-1)e^{-x^2/2}}{\sqrt{2\pi}}dxdt\\
&\qquad\le\frac{\exp(-(\gamma-.53)^2/.999)}{\sqrt{2\pi}}\cdot \sqrt{\beta}\int_{0}^{10}\int_{-t}^{t}\frac{(x^2-1)e^{-x^2/2}}{\sqrt{2\pi}}dxdt\\
&\qquad\le -0.29 \sqrt{\beta}
\end{align*}
by numerical computation for the last line. In the first inequality, we used that for $t\in[0,10]$ we have $\gamma-.53\le z_1\le-(\gamma-.53)$ by numerical computation. We also used the earlier observation that conditioning a Gaussian to lie in $[-t,t]$ will only decrease its variance, which corresponds to the fact $\int_{-t}^t(x^2-1)e^{-x^2/2}/\sqrt{2\pi}dx<0$.

Putting together the previous three centered inequalities we have that
\begin{align*}
\frac{\partial}{\partial\beta}G_2(\beta,\alpha_2) &\le 2.239 -2\log\beta + \frac{1.635}{\beta}\mb{E}\Big[(Z_2^2-1)\mbm{1}\Big[\sqrt{\frac{1-\beta}{2}}Z_1\pm \sqrt{\frac{\beta}{2}}Z_2\ge \gamma + \alpha_2\Big]\Big]\\
&\le 2.239-\log\beta-\frac{.47}{\sqrt{\beta}}
\end{align*}
by numerical computation.

The integral of the right side above from $\beta = 0 $ to $.001$ is strictly negative by numerical computation. Therefore, the maximum of $G_2(\beta,\alpha_2)$ is achieved when $\beta = 0$ for $\alpha_2\in [-.53,-.37]$. The result then follows by definition of $F_1$ and of $G_2$.
\end{proof}

\subsection{Proof of \texorpdfstring{\cref{clm:middle-segment}}{Claim B.2}}
Note that since $f(\beta,\alpha)$ is a decreasing function of $\gamma$ it suffices to verify \cref{clm:middle-segment} when $\gamma = .2484195\le \gamma_{\mr{crit}}-\eps_{\ref{asm:comp}}$ (where we have used \cref{clm:numeric-bounds}).

\subsubsection{Initial segment of \texorpdfstring{\cref{clm:middle-segment}}{Claim B.2}}
We first handle the case where $\beta\in[.001,.005]$ via dropping the variable $\alpha_1$ in analogy with the proof of \cref{clm:initial-interval}; we do this since $f(\beta,\alpha_1)$ is numerically sensitive to compute for sufficiently small values of $\beta$.

\begin{claim}\label{clm:interval-1}
Let $\gamma = .2484195$. We have 
\[\sup_{\substack{\beta\in[.001,.005]\\\alpha_1, \alpha_2\in \mb{R}\\\alpha_2\le -\gamma}} F_2(\beta,\alpha_1,\alpha_2)\le -10^{-3}.\]
\end{claim}
\begin{proof}
Suppose that $\beta\in[\eta_1,\eta_2]$ with $\eta_2\le 1/2$ and $\alpha_2\le -\gamma$. Notice that 
\begin{align*}
F_2(\beta,\alpha_1,\alpha_2) &= 2\log 2-2\beta\log\beta-2(1-\beta)\log(1-\beta) -2 \alpha_1^2-2\alpha_2^2 \\
&\qquad+ 2\beta\log f(\beta,\alpha_1) + 2(1-\beta)\log f(1-\beta,\alpha_2)\\
&\le 2\log 2-2\beta\log\beta-2(1-\beta)\log(1-\beta)-2\alpha_2^2 + 2(1-\beta)\log f(1-\beta,\alpha_2)\\
&\le 2\log 2-2\eta_2\log\eta_2-2(1-\eta_2)\log(1-\eta_2)-2\alpha_2^2 + 2(1-\eta_2)\log f(1-\beta,\alpha_2)\\
&\le 2\log 2-2\eta_2\log\eta_2-2(1-\eta_2)\log(1-\eta_2)-2\alpha_2^2 \\
&\qquad + 2(1-\eta_2)\log\bigg(\mb{P}\bigg[\sqrt{\frac{1-\eta_1}{2}}Z_1\pm \sqrt{\frac{\eta_1}{2}}Z_2\ge\sqrt{\frac{1-\eta_1}{1-\eta_2}}(\gamma+\alpha_2)\bigg]\bigg)
\end{align*}
where we used \cref{lem:upper-envelope} in the last line.

We now segment $\beta$ into the intervals $[j\cdot 10^{-4}, (j+1)\cdot 10^{-4}]$ for $j = 10$ to $j = 49$. For each of these intervals considered, write it as $[\eta_1,\eta_2]$ and note that the above inequality provides a single univariate function in $\alpha_2$, independent of $\beta$, which we can use to provide an upper bound to $F_2(\beta,\alpha_1,\alpha_2)$ for $\beta$ in the specified interval and $\alpha_2\le-\gamma$. Furthermore, this function satisfies a second derivative guarantee necessary to apply \cref{lem:concave-bound-trick}, which allows us to reduce the checking to a finite numerical computation.
\end{proof}

\begin{claim}\label{clm:interval-2}
Let $\gamma = .2484195$. We have 
\[\sup_{\substack{\beta\in[.001,.005]\\\alpha_1, \alpha_2\in \mb{R}\\\alpha_2\ge -\gamma}} F_2(\beta,\alpha_1,\alpha_2)\le -10^{-2}.\]
\end{claim}
\begin{proof}
Suppose that $\beta\in [\eta_1,\eta_2]$ with $\eta_2\le 1/2$ and $\alpha_2\ge -\gamma$. Notice that 
\begin{align*}
F_2(\beta,\alpha_1,\alpha_2) &= 2\log 2-2\beta\log\beta-2(1-\beta)\log(1-\beta) -2 \alpha_1^2-2\alpha_2^2\\
&\qquad+2\beta\log f(\beta,\alpha_1)+2(1-\beta)\log f(1-\beta,\alpha_2)\\
&\le 2\log 2-2\eta_2\log\eta_2-2(1-\eta_2)\log(1-\eta_2)-2\alpha_2^2 + 2(1-\eta_2)\log f(1-\eta_1, \alpha_2)
\end{align*}
by \cref{lem:upper-envelope}. We now segment $\beta$ into the intervals $[j\cdot 10^{-3}, (j+1)\cdot 10^{-3}]$ for $j = 1$ to $j = 4$. It is important to note that the $\alpha_2\in\mb{R}$ attaining the maximum of the right will lie outside the range where $\alpha_2\ge -\gamma$.

Let
\[G(\alpha_2) = 2\log 2-2\eta_2\log\eta_2-2(1-\eta_2)\log(1-\eta_2)-2\alpha_2^2 + 2(1-\eta_2)\log f(1-\eta_1, \alpha_2)\]
and $\alpha_2^\ast = \argmax_{\alpha_2\in\mb{R}} G(\alpha_2)$. By the log-concavity of $f(1-\eta_2,\alpha_2)$ (\cref{lem:log-concave-2}), we have that $G''(\alpha_2)\le -2$ and thus
\[\wt{G}(\alpha_2) = G(\alpha_2) + 2(\alpha_2-\alpha_2^{\ast})^2\]
is concave. As $\wt{G}'(\alpha_2^\ast) = 0$ and $\wt{G}$ is concave, we have that $\wt{G}(\alpha_2)\le \wt{G}(\alpha_2^\ast) = G(\alpha_2^\ast)$ for all $\alpha_2$. This implies that 
\[G(\alpha_2)\le G(\alpha_2^{\ast}) - 2(\alpha_2-\alpha_2^{\ast})^2\] and therefore
\[\sup_{\alpha_2\ge -\gamma} G(\alpha_2) \le \sup_{\alpha_2\in R} G(\alpha_2) - 2(\alpha_2^\ast+\gamma)^2\mbm{1}_{\alpha_2^\ast\le - \gamma}.\]
This allows us to incorporate an additional correction term in certain cases. We verify this by finding a ``numerical optimizer'' of $G$ and certify it is close to $\alpha_2^\ast$ by checking it has small derivative. Numerical computation finishes.
\end{proof}

\subsubsection{Upper segment of \texorpdfstring{\cref{clm:middle-segment}}{Claim B.2} in non-critical regions}
For $\gamma \approx \gamma_{\mr{crit}}$ and $\beta \approx 1/2$ we have that the maximizers satisfy $\alpha_1, \alpha_2 \approx -.445$. Therefore if not both $\gamma + \alpha_1\le 0$ and $\gamma + \alpha_2\le 0$ then given a sufficiently fine decomposition into intervals one can certify the necessary upper bound, even up to $\beta=1/2$ instead of just $.495$. We carry out this procedure in this subsection. As before, since we are concerned only with upper bounds we may assume that $\gamma = .2484195$; additionally, we will be increasingly brief in these proofs since they are quite close in nature to the bounds derived in \cref{clm:interval-1,clm:interval-2}.

\begin{claim}\label{clm:interval-3}
Let $\gamma = .2484195$. We have 
\[\sup_{\substack{\beta\in[.005,.5]\\\alpha_1,\alpha_2\ge -\gamma}} F_2(\beta,\alpha_1,\alpha_2)\le -10^{-3}.\]
\end{claim}
\begin{proof}
Suppose that $\beta\in [\eta_1,\eta_2]$ with $\eta_2\le 1/2$ and $\alpha_1, \alpha_2\ge -\gamma$. Notice that by \cref{lem:upper-envelope},
\begin{align*}
F_2(\beta,\alpha_1,\alpha_2) &\le 2\log 2-2\eta_2\log\eta_2-2(1-\eta_2)\log(1-\eta_2)-2\alpha_1^2 + 2\eta_1\log f(\eta_2,\alpha_1)\\
&\qquad-2\alpha_2^2 + 2(1-\eta_2)\log f(1-\eta_1,\alpha_2).
\end{align*}
We now segment $\beta$ into $[j\cdot 10^{-3}, (j+1)\cdot 10^{-3}]$ for $j = 5$ to $j = 19$ and $[j\cdot 10^{-2}, (j+1)\cdot 10^{-2}]$ for $j = 2$ to $j = 49$ and incorporate corrections as carried out in \cref{clm:interval-2} for both $\alpha_1$ and $\alpha_2$ when applicable, i.e., when the optimum over all of $\mb{R}$ disagrees with the optimum in the desired interval. (Notice that we can perform the optimizations over the two variables independently and sum the resulting values.) Numerical computation finishes.
\end{proof}

\begin{claim}\label{clm:interval-4}
Let $\gamma = .2484195$. We have
\[\sup_{\substack{\beta\in[.005,.5]\\\alpha_1\le -\gamma ,\alpha_2\ge -\gamma}} F_2(\beta,\alpha_1,\alpha_2)\le -10^{-3}.\]
\end{claim}
\begin{proof}
Suppose that $\beta\in[\eta_1,\eta_2]$ with $\eta_2\le 1/2$ and $\alpha_1\le -\gamma$ and $\alpha_2\ge -\gamma$. \cref{lem:upper-envelope} gives
\begin{align*}
F_2(\beta,\alpha_1,\alpha_2) &\le 2\log 2-2\eta_2\log\eta_2-2(1-\eta_2)\log(1-\eta_2)-2\alpha_2^2 + 2(1-\eta_2)\log f(1-\eta_1, \alpha_2)\\
&\qquad-2\alpha_1^2 + 2\eta_1\log\bigg(\mb{P}\bigg[\sqrt{\frac{\eta_2}{2}}Z_1\pm \sqrt{\frac{1-\eta_2}{2}}Z_2\ge\sqrt{\frac{\eta_2}{\eta_1}}(\gamma+\alpha_1)\bigg]\bigg)
\end{align*}
We now segment $\beta$ into $[j\cdot 10^{-3}, (j+1)\cdot 10^{-3}]$ for $j = 5$ to $j = 19$ and $[j\cdot 10^{-2}, (j+1)\cdot 10^{-2}]$ for $j = 2$ to $j = 49$ and incorporate the corrections as carried out in \cref{clm:interval-2} for both $\alpha_1$ and $\alpha_2$ when applicable. Numerical computation finishes.
\end{proof}

\begin{claim}\label{clm:interval-5}
Let $\gamma = .2484195$. We have 
\[\sup_{\substack{\beta\in[.005,.5]\\\alpha_1\ge -\gamma ,\alpha_2\le -\gamma}} F_2(\beta,\alpha_1,\alpha_2)\le -10^{-3}.\]
\end{claim}
\begin{proof}
Suppose that $\beta\in[\eta_1,\eta_2]$ with $\eta_2\le 1/2$ and $\alpha_1\ge -\gamma$ and $\alpha_2\le -\gamma$. \cref{lem:upper-envelope} gives 
\begin{align*}
F_2(\beta,\alpha_1,\alpha_2) &\le 2\log 2-2\eta_2\log\eta_2-2(1-\eta_2)\log(1-\eta_2)-2\alpha_1^2 + 2\eta_1\log f(\eta_2, \alpha_1)\\
&\qquad-2\alpha_2^2 + 2(1-\eta_2)\log\bigg(\mb{P}\bigg[\sqrt{\frac{1-\eta_1}{2}}Z_1\pm \sqrt{\frac{\eta_1}{2}}Z_2\ge\sqrt{\frac{1-\eta_1}{1-\eta_2}}(\gamma+\alpha_2)\bigg]\bigg).
\end{align*}
We now segment $\beta$ into $[j\cdot 10^{-3}, (j+1)\cdot 10^{-3}]$ for $j = 5$ to $j = 19$ and $[j\cdot 10^{-2}, (j+1)\cdot 10^{-2}]$ for $j = 2$ to $j = 49$ and incorporate the corrections as specified in \cref{clm:interval-2} for both $\alpha_1$ and $\alpha_2$ when applicable. Numerical computation finishes.
\end{proof}

\subsubsection{Upper segment of \texorpdfstring{\cref{clm:middle-segment}}{Claim B.2} in critical region}
We now handle the final region of \cref{clm:middle-segment}. This procedure is essentially identical to the previous claims, but the set of intervals required is substantially more intricate as the maximum, which is near $0$, is attained in this region.

\begin{claim}\label{clm:interval-6}
Let $\gamma = .2484195$. We have that 
\[\sup_{\substack{\beta\in[.005,.495]\\\alpha_1,\alpha_2\le -\gamma}} F_2(\beta,\alpha_1,\alpha_2)\le -10^{-5}.\]
\end{claim}
\begin{proof}
Suppose that $\beta\in[\eta_1,\eta_2]$ with $\eta_2\le 1/2$ and $\alpha_1, \alpha_2\le -\gamma$. Notice that 
\begin{align*}
F_2(\beta,\alpha_1,\alpha_2) &\le 2\log(2)-2\eta_2\log\eta_2-2(1-\eta_2)\log(1-\eta_2)\\
&\qquad-2\alpha_1^2 + 2\eta_1\log\bigg(\mb{P}\bigg[\sqrt{\frac{\eta_2}{2}}Z_1\pm \sqrt{\frac{1-\eta_2}{2}}Z_2\ge\sqrt{\frac{\eta_2}{\eta_1}}(\gamma+\alpha_1)\bigg]\bigg)\\
&\qquad-2\alpha_2^2 + 2(1-\eta_2)\log\bigg(\mb{P}\bigg[\sqrt{\frac{1-\eta_1}{2}}Z_1\pm \sqrt{\frac{\eta_1}{2}}Z_2\ge\sqrt{\frac{1-\eta_1}{1-\eta_2}}(\gamma+\alpha_2)\bigg]\bigg).
\end{align*}
The partition for $\beta$ in this region is substantially more involved; we take $[j\cdot 10^{-3}, (j+1)\cdot 10^{-3}]$ for $j = 5$ to $j = 19$, $[j\cdot 10^{-2}, (j+1)\cdot 10^{-2}]$ for $j = 2$ to $j = 24$, $[j\cdot 10^{-3}, (j+1)\cdot 10^{-3}]$ for $j = 250$ to $j = 424$, $[j\cdot 10^{-4}, (j+1)\cdot 10^{-4}]$ for $j = 4250$ to $j = 4724$, $[j\cdot (5\cdot 10^{4})^{-1}, (j+1)\cdot (5\cdot 10^{4})^{-1}]$ for $j = 23625$ to $j = 24299$, $[j\cdot 10^{-5}, (j+1)\cdot 10^{-5}]$ for $j = 48600$ to $j = 48799$, $[j\cdot (2\cdot 10^{5})^{-1}, (j+1)\cdot (2\cdot 10^{5})^{-1}]$ for $j = 97600$ to $j = 98749$, and finally $[j\cdot (4\cdot 10^{5})^{-1}, (j+1)\cdot (4\cdot 10^{5})^{-1}]$ for $j = 197500$ to $j = 197999$. Numerical computation finishes.
\end{proof}

We are now in position to prove \cref{clm:middle-segment}.
\begin{proof}[Proof of \cref{clm:middle-segment}]
This is an immediate consequence of combining \cref{clm:interval-1,clm:interval-2,clm:interval-3,clm:interval-4,clm:interval-5,clm:interval-6}.
\end{proof}

\subsection{Proof of \texorpdfstring{\cref{clm:middle-segment-local}}{Claim B.3}}
The proof of \cref{clm:middle-segment-local} is closely related to the proofs of various claims given in \cref{clm:middle-segment}. Note that \cref{clm:interval-3,clm:interval-4,clm:interval-5} handle $\beta$ up to $1/2$, so we only need to refine our estimates in the critical region where $\alpha_1,\alpha_2\le-\gamma$. Additionally, we may assume $\beta\in[.495,.5]$ by \cref{clm:interval-6}.

However, in the neighborhood of $(\alpha_1,\alpha_2)=(-.445,-.445)$, our function might be positive. So, some care is required to show that it is negative away from a neighborhood of this point. To do this, we show that there is a point in this neighborhood which has not too large value (potentially positive), and small partial derivatives in $\alpha_1,\alpha_2$. However, our function is strongly concave. Therefore, moving far enough from this point will force our function to be negative.

\begin{proof}[Proof of \cref{clm:middle-segment-local}]
As before it suffices to prove the claim conditional on $\gamma = .2484195$. Furthermore, it suffices to consider $\alpha_1, \alpha_2\le -\gamma$ by \cref{clm:interval-3,clm:interval-4,clm:interval-5} and $\beta\in[.495,.5]$ by \cref{clm:interval-6}. Now, if $\beta\in[\eta_1,\eta_2]$ with $\eta_2\le 1/2$ and $\alpha_1, \alpha_2\le -\gamma$ then 
\begin{align*}
F_2(\beta,\alpha_1,\alpha_2) &\le 2\log 2-2\eta_2\log\eta_2-2(1-\eta_2)\log(1-\eta_2)\\
&\qquad-2\alpha_1^2 + 2\eta_1\log\bigg(\mb{P}\bigg[\sqrt{\frac{\eta_2}{2}}Z_1\pm \sqrt{\frac{1-\eta_2}{2}}Z_2\ge\sqrt{\frac{\eta_2}{\eta_1}}(\gamma+\alpha_1)\bigg]\bigg)\\
&\qquad-2\alpha_2^2 + 2(1-\eta_2)\log\bigg(\mb{P}\bigg[\sqrt{\frac{1-\eta_1}{2}}Z_1\pm \sqrt{\frac{\eta_1}{2}}Z_2\ge\sqrt{\frac{1-\eta_1}{1-\eta_2}}(\gamma+\alpha_2)\bigg]\bigg)\\
&=:G_2(\alpha_1,\alpha_2).
\end{align*}
We break $\beta$ into segments $[j\cdot (4\cdot 10^{5})^{-1}, (j+1)\cdot (4\cdot 10^{5})^{-1}]$ for $j = 198000$ to $j = 199999$. For each interval of the form $[\eta_1,\eta_2]$ using numerical computation we produce a point $(\alpha_1^\ast,\alpha_2^\ast)$ such that 
\begin{itemize}
    \item $G_2(\alpha_1^{\ast},\alpha_2^{\ast})\le 10^{-5}$;
    \item $|\alpha_1^{\ast} + .445|\le 1.5 \cdot 10^{-3}$ and $|\alpha_2^{\ast} + .445|\le 1.5 \cdot 10^{-3}$;
    \item $|\frac{\partial}{\partial \alpha_1}G_2(\alpha_1,\alpha_2)|_{(\alpha_1,\alpha_2) = (\alpha_1^{\ast},\alpha_2^{\ast})}|+ |\frac{\partial}{\partial \alpha_2}G_2(\alpha_1,\alpha_2)|_{(\alpha_1,\alpha_2) = (\alpha_1^{\ast},\alpha_2^{\ast})}|\le 10^{-5}$.
\end{itemize}
These conditions together with the strong concavity of $G_2$ imply the claimed result along with numerical computation. In particular, we use that the second derivatives in $\alpha_1,\alpha_2$ are uniformly less than $-2$, and that $G_2$ is the sum of a function of $\alpha_1$ and of $\alpha_2$ separately.
\end{proof}

\subsection{Proof of \texorpdfstring{\cref{clm:hessian-verification}}{Claim B.4}}
In order to prove the desired Hessian bound we prove various estimates on the coordinate second partial derivatives in the region for \cref{clm:hessian-verification}. To prove the desired claim we will prove the following set of second derivative estimates.

\begin{lemma}\label{lem:hess-estimates}
Let $\gamma = .2484195$. At points for which $\beta\in [.495,.505]$ and $\alpha_1,\alpha_2\in [-.45,-.44]$ we have for $i\in\{1,2\}$ that
\[\frac{\partial^2 F_2}{\partial \beta^2} \le -2.15,~~\bigg|\frac{\partial^2 F_2}{\partial \beta\partial \alpha_i}\bigg| \le 2.05.\]
\end{lemma}

Given the above pair of estimates we now prove the desired result regarding the Hessian of $F_2$.
\begin{proof}[Proof of \cref{clm:hessian-verification}]
Notice that by the log-concavity of the function $f(\beta,\alpha)$ from \cref{lem:log-concave-2}, by direct computation for all $\gamma$ and $\beta,\alpha_i$ such that $F_2$ is defined we have 
\[\frac{\partial^2 F_2}{\partial \alpha_i^2}\le -4.\]
Furthermore the functional form of $F_2$ immediately implies that 
\[\frac{\partial^2 F_2}{\partial \alpha_1\partial \alpha_2} = 0.\]
Additionally, note that a small shift in $\gamma$ can be equivalently recast as slightly shifting the values of $\alpha_1,\alpha_2$ when computing $\frac{\partial^2 F_2}{\partial \beta^2}, \frac{\partial^2 F_2}{\partial \beta\partial \alpha_i}$. Thus, applying \cref{lem:hess-estimates}, for $\gamma$ in the specified range and $\alpha_1,\alpha_2\in [-.449,-.441]$ the second derivative bounds listed in \cref{lem:hess-estimates} still apply.

The desired result then follows via considering the negation of the Hessian and applying Sylvester's criterion on determinants of minors for a matrix to be positive semidefinite. The matrix is strictly negative semidefinite as the determinant is also seen to be strictly bounded away from zero, and compactness guarantees the existence of an absolute constant $\delta > 0$ for the desired result.
\end{proof}

The proof of \cref{lem:hess-estimates} is mostly calculation with the function form given in \cref{clm:gaussian-rot} and various grid procedures to verify certain elementary two-dimensional inequalities are satisfied. Since the inequalities we which to certify are not particularly delicate, the procedure is mechanical, if not completely pleasant. 

We will require the following triplet of claims regarding the values of $f(\beta,\alpha)$ and certain specially chosen combinations of derivatives. The combinations naturally fall out of an analysis based on explicit computation; we present these claims as a series of unmotivated numerical claims for the sake of simplifying the verification.

\begin{claim}\label{clm:prob-bound}
Let $\gamma = .2484195$, $\beta\in [.495,.505]$, and $\alpha\in [-.45,-.44]$. Then we have 
\[f(\beta,\alpha)\in[.36544,.37761].\]
\end{claim}
\begin{proof}
By \cref{lem:upper-envelope} and numerical computation we have 
\[f(\beta,\alpha)\le f(\beta, -.45)\le \mb{P}\bigg[\sqrt{\frac{.505}{2}}Z_1\pm \sqrt{\frac{.495}{2}}Z_2\ge (\gamma -.45)\sqrt{\frac{.505}{.495}}\bigg]\le .37761\]
and
\[f(\beta,\alpha)\ge f(\beta, -.44)\ge \mb{P}\bigg[\sqrt{\frac{.495}{2}}Z_1\pm \sqrt{\frac{.505}{2}}Z_2\ge (\gamma -.44)\sqrt{\frac{.495}{.505}}\bigg]\ge .36544.\qedhere\]
\end{proof}

The next claim bounds the size of the first derivative of $f(\beta,\alpha)$ in $\beta$; the proof involves reducing to a certain tractable two-variable integral whose range can be determined via a direct grid search in combination with a crude bounding procedure. As these computations are substantially less numerically delicate we carried these computations out only in \texttt{scipy}. 

\begin{claim}\label{clm:first-derivative}
Let $\gamma = .2484195$, $\beta\in [.495,.505]$, and $\alpha\in[-.45,-.44]$. Then we have 
\[\frac{\partial f(\beta,\alpha)}{\partial \beta}\in[.2780,.3110].\]
\end{claim}
\begin{proof}
Let $\eta=1/2-\beta$. By differentiation under the integral sign in combination with the form presented in \cref{clm:gaussian-rot} we have that 
\begin{align*}
\frac{\partial f(\beta,\alpha)}{\partial \beta} &= \int_{\gamma + \alpha}^{\infty} \int_{\gamma + \alpha}^{\infty}\frac{t_1t_2 + \eta(2t_1^2+2t_2^2-1) + 4\eta^2t_1t_2 + 4\eta^3}{8\pi (\beta(1-\beta))^{5/2}}\\
&\qquad\qquad\qquad\qquad\times\exp\bigg(-\frac{1}{4\beta(1-\beta)}(t_1^2+t_2^2)-\frac{\eta}{\beta(1-\beta)}t_1t_2\bigg)dt_1dt_2.
\end{align*}
We break this integral into two pieces. Notice that
\begin{align*}
&\bigg|\int_{\gamma + \alpha}^{\infty} \int_{\gamma + \alpha}^{\infty}\frac{\eta(2t_1^2+2t_2^2-1) + 4\eta^2t_1t_2 + 4\eta^3}{8\pi (\beta(1-\beta))^{5/2}} \exp\bigg(-\frac{1}{4\beta(1-\beta)}(t_1^2+t_2^2)-\frac{\eta}{\beta(1-\beta)}t_1t_2\bigg)dt_1dt_2\bigg|\\
&\le \int_{\gamma -.45}^{\infty} \int_{\gamma -.45}^{\infty}\frac{|\eta(2t_1^2+2t_2^2-1)| + |4\eta^2t_1t_2| + |4\eta^3|}{8\pi (\beta(1-\beta))^{5/2}} \exp\bigg(-\frac{1}{4\beta(1-\beta)}(t_1^2+t_2^2)-\frac{\eta}{\beta(1-\beta)}t_1t_2\bigg)dt_1dt_2\\
&\le \int_{\gamma -.45}^{\infty} \int_{\gamma -.45}^{\infty}\frac{|\eta(2t_1^2+2t_2^2-1)| + |4\eta^2t_1t_2| + |4\eta^3|}{8\pi (\beta(1-\beta))^{5/2}} \exp\bigg(-.99(t_1^2+t_2^2)\bigg)dt_1dt_2\le .00975
\end{align*}
by upper bounding $|\eta|\le .005$ and $(\beta(1-\beta))^{-5/2}\le (.495 \cdot .505)^{-5/2}$ and then computing the resulting integral numerically. Furthermore, we have 
\begin{align*}
&\bigg|\int_{\gamma + \alpha}^{\infty} \int_{\gamma + \alpha}^{\infty}\bigg(\frac{t_1t_2}{8\pi (\beta(1-\beta))^{5/2}}-\frac{4t_1t_2}{\pi}\bigg)\exp\bigg(-\frac{1}{4\beta(1-\beta)}(t_1^2+t_2^2)-\frac{\eta}{\beta(1-\beta)}t_1t_2\bigg)dt_1dt_2\bigg|\\
&\le \int_{\gamma + \alpha}^{\infty} \int_{\gamma + \alpha}^{\infty}|t_1t_2|\bigg(\frac{1}{8\pi (.495\cdot .505)^{5/2}}-\frac{4}{\pi}\bigg)\exp\bigg(-.99(t_1^2+t_2^2)\bigg)dt_1dt_2\le .0001
\end{align*}
by numerical computation.

Thus it suffices to understand the value of 
\[\int_{\gamma + \alpha}^{\infty} \int_{\gamma + \alpha}^{\infty}\frac{4t_1t_2}{\pi } \exp\bigg(-\frac{1}{4\beta(1-\beta)}(t_1^2+t_2^2)-\frac{1/2-\beta}{\beta(1-\beta)}t_1t_2\bigg)dt_1dt_2\]
for $\beta \in [.495, .505]$ and $\alpha\in [-.45,-.44]$. We first handle shifts in $\alpha$; notice that the derivative in absolute value of the above expression in $\alpha$ is at most 
\[
\frac{8}{\pi}\int_{\gamma+\alpha}^{\infty}|(\gamma + \alpha)t_1|\cdot \exp(-.99(\gamma + \alpha)^2 - .99t_1^2)dt_1\le \frac{8|\gamma -.45|}{\pi}\int_{\gamma-.45}^{\infty}|t_1|\cdot \exp(- .99t_1^2)dt_1\le .27\]
and thus shifting $\alpha$ to $\alpha'$ the difference in the integral is bounded by $.27|\alpha-\alpha'|$.

Similarly, by the inequality $|\exp(x) - \exp(y)|\le\exp(\max(x,y))|x-y|$, we have that shifting $\beta$ to $\beta'$ induces an error bounded by
\[|\beta-\beta'|\int_{\gamma -.45}^{\infty} \int_{\gamma -.45}^{\infty}\frac{4|t_1t_2|}{\pi } \exp\big(-.99(t_1^2+t_2^2)\big)(.041 (t_1^2+t_2^2) + 4.01|t_1t_2|)dt_1dt_2\le 1.1 |\beta-\beta'|\]
by numerical computation. We take a net over $\beta$ and $\alpha$ each of granularity $.00025$. Applying a grid verification and applying the errors given above we find the specified result by numerical computation.
\end{proof}

Finally, we will require an upper bound on the following mixed partial derivative expression. The strategy is essentially identical to the previous proof however the precise formulas are substantially less pleasant and therefore the computation is less clean. 

\begin{claim}\label{clm:second-derivative}
Let $\gamma = .2484195$, $\beta\in[.495,.505]$, and $\alpha\in[-.45,-.44]$. Then we have
\[\frac{\beta\partial^2 f(\beta,\alpha)}{\partial \beta^2}+ \frac{2\partial f(\beta,\alpha)}{\partial \beta}\le .630.\]
\end{claim}
\begin{proof}
Let $\eta = 1/2-\beta$ and apply differentiation under the integral sign to \cref{clm:gaussian-rot}. We have
\[\frac{\beta\partial^2 f(\beta,\alpha)}{\partial \beta^2}+ \frac{2\partial f(\beta,\alpha)}{\partial \beta} = \int_{\gamma + \alpha}^{\infty} \int_{\gamma + \alpha}^{\infty}\frac{P(t_1,t_2,\beta)}{32\pi \beta^{7/2}(1-\beta)^{9/2}} \exp\bigg(-\frac{1}{4\beta(1-\beta)}(t_1^2+t_2^2)-\frac{\eta}{\beta(1-\beta)}t_1t_2\bigg)dt_1dt_2\]
where 
\begin{align*}
&P(t_1,t_2,\beta)\\
&= 16 \eta^5 (-1 + 2 t_1 t_2) + 4 \eta^4 (1 - 4 t_1t_2 + 6 t_1^2 + 6t_2^2  + 4 t_1^2t_2^2) + 
 8 \eta^3 (1 - (t_1-t_2)^2 + 2 t_1t_2(t_1^2+t_2^2))\\
&\qquad + 2 \eta^2 (-1 - 2 t_1^2 -2t_2^2 + 8 t_1^2 t_2^2 + 2 t_1^4 + 2 t_2^4 ) +\eta (-1 + 2 t_1^2 - 6 t_1t_2 + 2t_2^2+ 4 t_1t_2(t_1^2+t_2^2))\\
&\qquad + 1/4 (1 -2(t_1-t_2)^2 + 4 t_1^2t_2^2)\\
&\le \eta^2 (4.5 t_1^4 + 4.5 t_2^4 + 17t_1^2t_2^2) +\eta (-1 + 2 t_1^2 - 6 t_1t_2 + 2t_2^2+ 4 t_1t_2(t_1^2+t_2^2))\\
&\qquad + 1/4 (1 -2(t_1-t_2)^2 + 4 t_1^2t_2^2)
\end{align*}
where we have very crudely applied $|\eta|\le.005$ and the AM-GM inequality to eliminate various higher order terms which appear in the coefficients of powers of $\eta$. For the sake of simplicity, let
\begin{align*}
Q_1(t_1,t_2,\beta) &= \eta^2 (4.5 t_1^4 + 4.5 t_2^4 + 17t_1^2t_2^2),\\
Q_2(t_1,t_2,\beta) &= (-1 + 2 t_1^2 - 6 t_1t_2 + 2t_2^2+ 4 t_1t_2(t_1^2+t_2^2))\eta+ 1/4 (1 - 2(t_1-t_2)^2 + 4 t_1^2t_2^2).
\end{align*}
We have
\begin{align*}
&\int_{\gamma + \alpha}^{\infty} \int_{\gamma + \alpha}^{\infty}\frac{Q_1(t_1,t_2,\beta)}{32\pi \beta^{7/2}(1-\beta)^{9/2}} \exp\bigg(-\frac{1}{4\beta(1-\beta)}(t_1^2+t_2^2)-\frac{\eta}{\beta(1-\beta)}t_1t_2\bigg)dt_1dt_2\\
&\le \int_{\gamma -.45}^{\infty} \int_{\gamma -.45}^{\infty}\frac{Q_1(t_1,t_2,\beta)}{32\pi \beta^{7/2}(1-\beta)^{9/2}} \exp\big(-.99(t_1^2+t_2^2)\big)dt_1dt_2\le .0007
\end{align*}
by numerical computation and $|\eta|\le.005$.

We next simplify the remaining integral term further. Notice that
\begin{align*}
&\bigg|\int_{\gamma + \alpha}^{\infty} \int_{\gamma + \alpha}^{\infty}\frac{Q_2(t_1,t_2,\beta)}{32\pi \beta^{7/2}(1-\beta)^{9/2}}-\frac{4Q_2(t_1,t_2,\beta)}{\pi (1-\beta)} \exp\bigg(-\frac{1}{4\beta(1-\beta)}(t_1^2+t_2^2)-\frac{\eta}{\beta(1-\beta)}t_1t_2\bigg)dt_1dt_2\bigg|\\
&\le 9\cdot 10^{-4}\cdot \int_{\gamma -.45}^{\infty} \int_{\gamma -.45}^{\infty}|Q_2(t_1,t_2,\beta)|\exp\big(-.99(t_1^2+t_2^2)\big)dt_1dt_2\le .00041
\end{align*}
by numerical computation and $|\eta|\le.005$ and $\alpha\ge-.45$.

Therefore it suffices to understand 
\[\int_{\gamma + \alpha}^{\infty} \int_{\gamma + \alpha}^{\infty}\frac{4Q_2(t_1,t_2,\beta)}{\pi (1-\beta)} \exp\bigg(-\frac{1}{4\beta(1-\beta)}(t_1^2+t_2^2)-\frac{\eta}{\beta(1-\beta)}t_1t_2\bigg)dt_1dt_2.\]
Note that if the above integral is negative we immediately obtain the desired bound, so it suffices to provide an upper bound to
\[\frac{4}{\pi (.495)} \int_{\gamma + \alpha}^{\infty} \int_{\gamma + \alpha}^{\infty}Q_2(t_1,t_2,\beta)\exp\bigg(-\frac{1}{4\beta(1-\beta)}(t_1^2+t_2^2)-\frac{\eta}{\beta(1-\beta)}t_1t_2\bigg)dt_1dt_2.\]

We wish to replace this with the following simpler integral to consider:
\begin{equation}\label{eq:second-derivative-1}
\frac{4}{\pi (.495)} \int_{\gamma + \alpha}^{\infty} \int_{\gamma + \alpha}^{\infty}Q_2(t_1,t_2,\beta)\exp\big(-(1+4\eta^2)(t_1^2+t_2^2)-4\eta t_1t_2\big)dt_1dt_2.
\end{equation}
Note that the difference between the above two integrals can be bounded using that $|\exp(x)-\exp(y)|\le |x-y|\exp(\max(x,y))$ and that $\exp(\cdot)$ is an increasing function. Using this and numerical computation, we can find that those integrals differ by at most
\[\frac{4}{\pi (.495)} \int_{\gamma -.45}^{\infty} \int_{\gamma -.45}^{\infty}|Q_2(t_1,t_2,\beta)|\cdot (2\cdot 10^{-8}\cdot (t_1^2+t_2^2) + 2\cdot 10^{-6}|t_1t_2|)\exp\big(-.99(t_1^2+t_2^2)\big)dt_1dt_2\le 2\cdot 10^{-6}.\]

We next remove $\eta$ out of the exponent of \cref{eq:second-derivative-1} by comparing it to the integral
\begin{equation}\label{eq:second-derivative-2}
\frac{4}{\pi (.495)} \int_{\gamma + \alpha}^{\infty} \int_{\gamma + \alpha}^{\infty}Q_2(t_1,t_2,\beta) (1-4\eta^2(t_1^2+t_2^2)-4\eta t_1t_2)\exp(-t_1^2-t_2^2)dt_1dt_2
\end{equation}
at the cost of an additive error bounded by
\begin{align*}
&\frac{4}{\pi (.495)} \int_{\gamma -.45}^{\infty} \int_{\gamma -.45}^{\infty}|Q_2(t_1,t_2,\beta)|\exp(-t_1^2-t_2^2)\\
&\qquad\qquad\qquad\qquad\qquad\times|1-4\eta^2(t_1^2+t_2^2)-4\eta t_1t_2)-\exp(-4\eta^2(t_1^2+t_2^2)-4\eta t_1t_2)|dt_1dt_2\\
&\qquad\le .0003
\end{align*}
using numerical computation, relying on the inequality that $|\exp(x)-x-1|\le \exp(|x|)-|x|-1$.

We now are in position to perform a grid search on \cref{eq:second-derivative-2}; notice that the derivative in $\alpha$ is bounded by
\[\frac{8}{\pi (.495)} \int_{\gamma + \alpha}^{\infty} \sup_{\substack{t_2\in [\gamma-.45,\gamma-.44]\\|\eta|\le .005}}|Q_2(t_1,t_2,1/2-\eta)|\cdot|1-4\eta^2(t_1^2+t_2^2)-4\eta t_1t_2|\exp(-t_1^2)dt_1\]
and can be computed to be bounded by $4.2$ by numerical computation. Thus shifting $\alpha$ to $\alpha'$ causes a shift bounded by $4.2|\alpha-\alpha'|$. To compute the shift in $\beta$, notice that $Q_2(t_1,t_2,\beta) (1-4\eta^2(t_1^2+t_2^2)-4\eta t_1t_2)$ is a polynomial in $\beta$. So, differentiating under the integral sign it suffices to bound 
\[\frac{8}{\pi (.495)} \int_{\gamma -.45}^{\infty} \int_{\gamma -.45}^{\infty}\bigg|\frac{\partial}{\partial \beta}Q_2(t_1,t_2,\beta) (1-4\eta^2(t_1^2+t_2^2)-4\eta t_1t_2)\bigg|\exp(-t_1^2-t_2^2)dt_1dt_2.\]
(Recall $\eta = 1/2-\beta$.) Via a straightforward numerical computation we find that theisabove is bounded by $8.76$ and thus shifting $\beta$ to $\beta'$ causes a shift bounded by $8.76|\beta-\beta'|$.

Take a grid in $\beta$ of width $.00025$ and in $\alpha$ of width $.00025$ and note that the closest grid points are at most half the width away moving in both coordinate directions. By numerical computation and bounding the error of rounding to grid points, we obtain a bound on \cref{eq:second-derivative-2} and can use this along with the prior analysis to find that the maximum of the desired value is bounded by $.630$ as claimed.
\end{proof}

We now proceed with the proof of \cref{lem:hess-estimates}. As we will see the estimates in this region are not particularly sharp and in particular we can tolerate a substantial amount of numerical error. 
\begin{proof}[Proof of \cref{lem:hess-estimates}]
Notice that
\begin{align*}
\frac{\partial^2F_2}{\partial\beta^2} &= \frac{-2}{\beta(1-\beta)} + \frac{\partial^2}{\partial\beta^2}(2\beta \log f(\beta,\alpha_1)) + \frac{\partial^2}{\partial\beta^2}(2(1-\beta) \log f(1-\beta,\alpha_2))\\
&\le -8 + 4\max_{\substack{\beta\in [.495,.505]\\\alpha_1\in [-.45,-.44]}}\frac{\partial^2}{\partial\beta^2}(\beta\log f(\beta,\alpha_1))\\
&\le -8 + 4\max_{\substack{\beta\in [.495,.505]\\\alpha_1\in [-.45,-.44]}}\frac{\beta\frac{\partial^2}{\partial\beta^2}f(\beta,\alpha_1) + 2\frac{\partial}{\partial\beta}f(\beta,\alpha_1)}{f(\beta,\alpha_1)} -\frac{\beta(\frac{\partial}{\partial\beta}f(\beta,\alpha_1))^2}{f(\beta,\alpha_1)^2}\\
&\le -8 + 4\bigg(\frac{.630}{.36544}-.495\bigg(\frac{.2775}{.37761}\bigg)^2\bigg)\le -2.17,
\end{align*}
using \cref{clm:prob-bound,clm:first-derivative,clm:second-derivative} and numerical computation. This completes the proof of the first item of \cref{lem:hess-estimates}. For the second item by symmetry (switching $\alpha_1,\alpha_2$ and replacing $\beta$ by $1-\beta$) we may assume $i=1$. Notice that by \cref{clm:derivative},
\begin{align*}
\bigg|&\frac{\partial^2F_2}{\partial\beta\partial\alpha_1}\bigg| = \bigg|\frac{\partial^2}{\partial\beta\partial\alpha_1}(2\beta\log f(\beta,\alpha_1))\bigg| \\
&= \bigg|\frac{\partial}{\partial\beta}\bigg(\frac{\beta}{f(\beta,\alpha_1)}\mb{P}\bigg[Z\ge (\gamma + \alpha_1)\sqrt{\frac{2(1-\beta)}{\beta}}\bigg]\bigg)\bigg| \cdot \bigg|\frac{4\exp(-(\gamma + \alpha_1)^2)}{\sqrt{\pi}}\bigg|\\
&\le 2.176 \cdot \bigg|\frac{\mb{P}\big[Z\ge (\gamma + \alpha_1)\sqrt{\frac{2(1-\beta)}{\beta}}\big]}{f(\beta,\alpha_1)}\bigg(1-\frac{\beta\frac{\partial}{\partial \beta}f(\beta,\alpha_1)}{f(\beta,\alpha_1)}\bigg) + \frac{(\gamma + \alpha_1)\exp(-(\gamma+\alpha_1)^2(1-\beta)/\beta)}{2\sqrt{\pi\beta(1-\beta)}f(\beta,\alpha_1)}\bigg|\\
&\le 5.955\bigg|\mb{P}\bigg[Z\ge (\gamma + \alpha_1)\sqrt{\frac{2(1-\beta)}{\beta}}\bigg]\bigg(1-\frac{\beta\frac{\partial}{\partial \beta}f(\beta,\alpha_1)}{f(\beta,\alpha_1)}\bigg) + \frac{(\gamma + \alpha_1)\exp(-(\gamma+\alpha_1)^2(1-\beta)/\beta)}{2\sqrt{\pi\beta(1-\beta)}}\bigg|
\end{align*}
where we have used \cref{clm:prob-bound}. We note via crude and direct bounding we have that 
\[\frac{(\gamma+\alpha_1)\exp(-(\gamma+\alpha_1)^2(1-\beta)/\beta)}{2\sqrt{\pi\beta(1-\beta)}}\in [-.110,-.103].\]
Therefore it suffices to prove that 
\[\bigg|\mb{P}\bigg[Z\ge (\gamma + \alpha_1)\sqrt{\frac{2(1-\beta)}{\beta}}\bigg]\bigg(1-\frac{\beta\frac{\partial}{\partial \beta}f(\beta,\alpha_1)}{f(\beta,\alpha_1)}\bigg)\bigg|\in [-.234,.447].\]
Notice that 
\[\mb{P}\bigg[Z\ge (\gamma + \alpha_1)\sqrt{\frac{2(1-\beta)}{\beta}}\bigg]\le \mb{P}\bigg[Z\ge (\gamma -.45)\sqrt{\frac{2(.505)}{.495}}\bigg]\le .614\]
Thus it suffices to prove that 
\[\bigg|1-\frac{\beta\frac{\partial}{\partial \beta}f(\beta,\alpha_1)}{f(\beta,\alpha_1)}\bigg|\in [-.381,.728],\text{ or equivalently}\frac{\beta\frac{\partial}{\partial \beta}f(\beta,\alpha_1)}{f(\beta,\alpha_1)}\in [-.272,.619].\]
This follows from \cref{clm:prob-bound,clm:first-derivative} as well as $\beta\in[.495,.505]$. We are (finally) done.
\end{proof}

\end{document}